\documentclass{elsarticle}

\usepackage[T1]{fontenc}
\usepackage{amsthm, amssymb, amsmath}
\usepackage{mathrsfs, latexsym, color}
\usepackage[pdftex]{}
\numberwithin{equation}{section}

\theoremstyle{plain}% default
\newtheorem{theorem}{Theorem}[section]

\newtheorem{lemma}[theorem]{Lemma}

\theoremstyle{definition}

\theoremstyle{remark}

\renewcommand{\bar}{\overline}

\newcommand{\inprod}[1]{\langle#1\rangle}

\newcommand{\abs}[1]{\lvert#1\rvert}

\newcommand{\N}{\mathbb{N}}
\newcommand{\Z}{\mathbb{Z}}
\newcommand{\Q}{\mathbb{Q}}
\newcommand{\C}{\mathbb{C}}

\usepackage[all]{xy}
\usepackage{pgf}
\usepackage{tikz}
\usetikzlibrary{arrows,decorations.markings,shapes}
\tikzstyle{vertex}=[circle, draw, fill=black, inner sep=0pt, minimum width=6pt]
\tikzstyle{vert}=[circle, draw=black, fill=white, inner sep=0pt, minimum width=6pt]
\tikzstyle{zc}=[circle, draw, fill=black, inner sep=0pt, minimum width=6pt]
\tikzstyle{pc}=[circle, draw=black, inner sep=1pt, minimum width=10pt, font=\tiny] %positive charge
\tikzstyle{nc}=[circle, draw=black, inner sep=1pt, minimum width=10pt, font=\tiny] % style=dashed,
\tikzstyle{pedge}=[draw,-]
\tikzstyle{wedge}=[draw,-,postaction={decorate}, decoration={markings,mark = at position 0.55 with {\arrow{stealth} } }]
\tikzstyle{dpedge}=[draw,-,postaction={decorate}]
\tikzstyle{wwedge}=[draw,-,postaction={decorate}, decoration={markings,mark = at position 0.5 with {\arrow{stealth} }, mark = at position 0.6 with {\arrow{stealth} } }]
\tikzstyle{wwedge2}=[draw,-,postaction={decorate}, decoration={markings,mark = at position 0.45 with {\arrow{stealth} }, mark = at position 0.65 with {\arrow{stealth} } }]
\tikzstyle{wwwedge}=[draw,-,postaction={decorate}, decoration={markings,mark = at position 0.65 with {\arrow{stealth} },mark = at position 0.45 with {\arrow{stealth} }, mark = at position 0.55 with {\arrow{stealth} } }]
\tikzstyle{wwwedge2}=[draw,-,postaction={decorate}, decoration={markings,mark = at position 0.7 with {\arrow{stealth} },mark = at position 0.4 with {\arrow{stealth} }, mark = at position 0.55 with {\arrow{stealth} } }]
\tikzstyle{wwwwedge}=[draw,-,postaction={decorate}, decoration={markings,mark = at position 0.7 with {\arrow{stealth} },mark = at position 0.4 with {\arrow{stealth} }, mark = at position 0.5 with {\arrow{stealth} }, mark = at position 0.6 with {\arrow{stealth} } }]
\tikzstyle{wwwnedge}=[draw,densely dashed,postaction={decorate}, decoration={markings,mark = at position 0.65 with {\arrow{stealth} },mark = at position 0.45 with {\arrow{stealth} }, mark = at position 0.55 with {\arrow{stealth} } }]
\tikzstyle{wwnedge}=[draw,densely dashed,postaction={decorate}, decoration={markings,mark = at position 0.5 with {\arrow{stealth} }, mark = at position 0.6 with {\arrow{stealth} } }]
\tikzstyle{wnedge}=[draw,densely dashed,postaction={decorate}, decoration={markings,mark = at position 0.55 with {\arrow{stealth} } }]
\tikzstyle{wnedge2}=[draw,densely dashed,postaction={decorate}, decoration={markings,mark = at position 0.6 with {\arrow{stealth} } }]
\tikzstyle{dnedge}=[draw,densely dashed,postaction={decorate}]
\tikzstyle{nedge}=[draw,densely dashed]
\tikzstyle{weight2}= [draw=white, fill=white, font=\scriptsize]
\tikzstyle{weight}= [font=\scriptsize]
\tikzstyle{empty}=[circle, draw=white, inner sep=2pt, fill=white, minimum width=4pt]
\tikzstyle{ghost}=[circle, draw=black, inner sep=1pt, style=densely dashed, minimum width=6pt, font=\tiny]
\tikzstyle{ghostc}=[circle, draw=black, inner sep=1pt, style=densely dashed, minimum width=10pt, font=\tiny]
\tikzstyle{dedge}=[draw,very thick,dotted]
\usepackage{fullpage}

\begin{document}
	\begin{frontmatter}
		\title{ Cyclotomic matrices over the Eisenstein and Gaussian integers }
		\author{ Gary Greaves }
		\address{Mathematics Department, Royal Holloway, Egham, Surrey, TW20 0EX, UK.}
		\ead{g.greaves@rhul.ac.uk}

		\begin{abstract}
				We classify all cyclotomic matrices over the Eisenstein and Gaussian integers, that is, all Hermitian matrices over the Eisenstein and Gaussian integers that have all their eigenvalues in the interval $[-2,2]$.
		\end{abstract}

	\end{frontmatter}
		
	\section{Introduction}
	
	Define a \textbf{cyclotomic} matrix to be a Hermitian matrix $A$ with integral characteristic polynomial $\chi_A(x) = \det(xI - A)$ whose zeros are contained in the interval $[-2,2]$.
	For Hermitian matrices $A$ over an imaginary quadratic ring $\mathcal O_{\Q(\sqrt{d})}$, the integrality of the characteristic polynomial is automatic.
	The nontrivial Galois automorphism $\sigma$ of $\Q(\sqrt{d})$ (with $d \in \Z_-$) over $\Q$ is simply complex conjugation.
	Applying $\sigma$ to the coefficients of $\chi_A$ gives $\sigma(\chi_A(x)) = \det(xI - \sigma(A)) = \det(xI - A^\top) = \chi_A(x)$.
	Hence, the coefficients of $\chi_A$ are rational and, since they are also algebraic integers, they must be in $\Z$.
	Cyclotomic matrices over the integers and over the imaginary quadratic integer rings $\mathcal O_{\Q(\sqrt{d})}$ for $d \ne -1$ and $d \ne -3$ have been classified by McKee and Smyth \cite{McKee:IntSymCyc07} and Taylor \cite{GTay:cyclos10} respectively.
	Therefore the completion of the classification of cyclotomic matrices over imaginary quadratic integer rings reduces to classifying cyclotomic matrices over the Gaussian integers $\Z[i]$ and over the Eisenstein integers $\Z[\omega]$, where $\omega = 1/2 + \sqrt{-3}/2$.
	
	Define the \textbf{Mahler measure} \cite{Mahler:Measure1962} of a monic polynomial $f$ as 
	\[
	M(f) = \prod_{j=1}^d \max(1,\abs{\alpha_j}),	
	\]
	where $f(x)$ has factorisation $f(x) = (x-\alpha_1)\dots(x-\alpha_d) \in \C[x]$.
	The Mahler measure of an $n \times n$ Hermitian matrix $A$ is defined as $M(A) = M(z^n \chi_A(z + 1/z))$.
	(This definition was used earlier by McKee and Smyth \cite{McKee:Salem05}.)
	By a theorem of Kronecker~\cite{Kron:cyclo57}, every eigenvalue of a cyclotomic matrix $A$ can be written as $\mu+1/\mu$, for some root of unity $\mu$.
	Therefore, for a cyclotomic matrix $A$, the polynomial $z^n \chi_A(z+1/z)$ is a product of cyclotomic polynomials.
	Moreover, the Mahler measure of a cyclotomic matrix is equal to $1$.
	In 1933, Lehmer~\cite{Lehmer:33Cyclo} asked whether or not there exists a monic integer polynomial $f$, having Mahler measure $1 < M(f) < \Omega$, where $\Omega = 1.176280818\dots$ is the larger real zero of the polynomial
	\[
		z^{10} + z^9 -z^7 - z^6 - z^5 - z^4 - z^3 + z + 1.
	\]
	
	The conjecture, that the answer is negative is called Lehmer's conjecture; for a discussion of its history see Smyth \cite{Smyth:MMsurvey08}.
	Lehmer's number $\Omega$ is also the smallest \emph{known} Salem number.
	Let $S$ and $T$ denote the set of Pisot numbers and the set of Salem numbers respectively.
	That the set $S$ is the set of limit points of $T$ and furthermore that $S \cup T$ is closed is a conjecture of Boyd~\cite{Boyd:Small77}.
	McKee and Smyth~\cite{McKee:Salem05,McKee:noncycISM09} used cyclotomic matrices to prove partial results in support of Lehmer's conjecture and Boyd's conjecture.
	
	Estes and Guralnick~\cite{EsGu:MinPolISM} conjectured that any totally real separable monic integer polynomial can occur as the minimal polynomial of an integer symmetric matrix.
	Infinitely many counterexamples to this conjecture were found by Dobrowolski~\cite{Dob:Int08}, the smallest of these having degree $2880$.
	McKee~\cite{McKee:SmallSpan10} used cyclotomic matrices to find counterexamples of small degree, the smallest having degree $6$.
	
	The motivation for our current work is to complete the classification of cyclotomic matrices over imaginary quadratic integer rings and in doing so, enable the study of broader classes of integer polynomials against both Lehmer's conjecture and Boyd's conjecture using ideas similar to those used in the citations \cite{McKee:Salem05,McKee:noncycISM09} above.
	
	The paper is organised as follows.
	In Section~\ref{sec:MaG} we explain a convenient way of picturing matrices and introduce the notions of equivalence and maximality.
	We state our results (the classification) in Section~\ref{sec:results} and set up the necessary machinery for the proofs in Section~\ref{sec:exGr}.
	The results are proved in Sections~\ref{sec:growZi}, \ref{sec:proof2}, \ref{sec:growZicharge}, and \ref{sec:the_eisenstein_integers}.
	In the final section we indicate the extra work required for the proofs to be extended to work for any quadratic integer ring.
	
	\section{Preliminaries}
	\label{sec:MaG}
	
	\subsection{Viewing matrices as graphs} % (fold)
	\label{sub:viewing_matrices_as_graphs}
	
	% subsection viewing_matrices_as_graphs (end)

	In our current setting, it is natural to think of Hermitian matrices as the adjacency matrices of directed weighted graphs.
	Let $S$ be a subset of $\C$.
	For an element $x \in \C$ we write $\bar x$ for the complex conjugate of $x$.
	An \textbf{$S$-graph} $G$ is a directed weighted graph $(G,w)$ whose weight function $w$ maps pairs of vertices to elements of $S$ and satisfies $w(u,v) = \bar w(v,u)$ for all vertices $u,v \in V(G)$.
	The adjacency matrix $A = (a_{uv})$ of $G$ is given by $a_{uv} = w(u,v)$.
	For every vertex $v$, the \textbf{charge} of $v$ is just the number $w(v,v)$.
	A vertex with nonzero charge is called \textbf{charged}, those with zero charge are called \textbf{uncharged}.
	By simply saying ``$G$ is a graph,'' we mean that $G$ is a $T$-graph where $T$ is some unspecified subset of the complex numbers.

	Now we outline our graph drawing conventions.
	We are interested in $S$-graphs where $S = \Z[i]$ and $S = \Z[\omega]$.
	Edges are drawn in accordance with Tables~\ref{tab:ziedge}~and~\ref{tab:zwedge}.
	It will become clear later why the edges in these tables are sufficient for our purposes.
	For edges with a real edge-weight, the direction of the edge does not matter, and so to reduce clutter, we do not draw arrows for these edges. 
	For all other edges, the number of arrowheads reflects the norm of the edge-weight.
	\begin{table}[htbp]
		\begin{center}
		\begin{tabular}{c | c}
			Edge-weight & Visual representation \\
			\hline
			$1$ & \tikz { \path[pedge] (0,0) -- (1,0); } \\
			$-1$ & \tikz { \path[nedge] (0,0) -- (1,0); } \\
			$i$ & \tikz { \path[wedge] (0,0) -- (1,0); } \\
			$-i$ & \tikz { \path[wnedge] (0,0) -- (1,0); } \\
			$1 + i$ & \tikz { \path[wwedge] (0,0) -- (1,0); } \\
			$-1 - i$ & \tikz { \path[wwnedge] (0,0) -- (1,0); } \\
			$2$ & \tikz { \path[pedge] (0,0) -- node[weight2] {$2$} (1,0); }
		\end{tabular}
		\end{center}
		\caption{Edge drawing convention for $\Z[i]$-graphs.}
		\label{tab:ziedge}
	\end{table}
	\begin{table}[htbp]
		\begin{center}
		\begin{tabular}{c | c}
			Edge-weight & Visual representation \\
			\hline
			$1$ & \tikz { \path[pedge] (0,0) -- (1,0); } \\
			$-1$ & \tikz { \path[nedge] (0,0) -- (1,0); } \\
			$\omega$ & \tikz { \path[wedge] (0,0) -- (1,0); } \\
			$-\omega$ & \tikz { \path[wnedge] (0,0) -- (1,0); } \\
			$1 + \omega$ & \tikz { \path[wwwedge] (0,0) -- (1,0); } \\
			$-1 - \omega$ & \tikz { \path[wwwnedge] (0,0) -- (1,0); } \\
			$2$ & \tikz { \path[pedge] (0,0) -- node[weight2] {$2$} (1,0); }
		\end{tabular}
		\end{center}
		\caption{Edge drawing convention for $\Z[\omega]$-graphs.}
		\label{tab:zwedge}
	\end{table}
	
	A vertex with charge $1$ is drawn as \tikz {\node[pc] {$+$};} and a vertex with charge $-1$ is drawn as \tikz {\node[nc] {$-$};}.
	For charge $2$ we draw \tikz {\node[nc] {$2$};}.
	And if a vertex is uncharged, we simply draw \tikz {\node[zc] {};}.
	By a \textbf{subgraph} $H$ of $G$ we mean an induced subgraph; a subgraph obtained by deleting vertices and their incident edges.
	We say that $G$ \textbf{contains} $H$ and that $G$ is a \textbf{supergraph} of $H$.
	The notions of a cycle/path/triangle etc.\ carry through in an obvious way from those of an undirected unweighted graph.
	A graph is called \textbf{charged} if it contains at least one charged vertex, otherwise it is called \textbf{uncharged}.
	We will interchangeably speak of both graphs and their adjacency matrices.
	
	\subsection{Equivalence and switching}
	% \label{sec:equivalence}
	
	Let $K = \Q(\sqrt{d})$ where $d \in \Z_-$ and let $R$ be a subring of $\mathcal O_K$, the ring of integers of $K$.
	We write $M_n(R)$ to denote the ring of $n \times n$ matrices over $R$. 
	Let $U_n(R)$ denote the group of unitary matrices $Q \in M_n(R)$ which satisfy $QQ^* = Q^*Q = I$, where $Q^*$ denotes the Hermitian transpose of $Q$.
	Conjugation of a matrix $M \in M_n(R)$ by a matrix in $U_n(R)$ preserves the eigenvalues of $M$ and the base ring $R$.
	Now, $U_n(R)$ is generated by permutation matrices and diagonal matrices of the form
	\[
		\operatorname{diag}(1,\dots,1,u,1,\dots,1),
	\]
	where $u$ is a unit in $R$.
	Let $D$ be such a diagonal matrix having $u$ in the $j$-th position.
	Conjugation by $D$ is called a $u$-\textbf{switching} at vertex $j$.
	This has the effect of multiplying all the out-neighbour edge-weights of $j$ by $u$ and all the in-neighbour edge-weights of $j$ by $\bar u$.
	The effect of conjugation by permutation matrices is just a relabelling of the vertices of the corresponding graph.
	
	We have seen above that for all $A \in M_n(R)$, the characteristic polynomial $\chi_A$ has integer coefficients.
	Observe that, since they are integers, the coefficients of the characteristic polynomials of such matrices are invariant under the action of automorphisms from the Galois group $\operatorname{Gal}(K/\Q)$ of $K$ over $\Q$.
	Let $A$ and $B$ be two matrices in $M_n(R)$.
	We say that $A$ is \textbf{strongly equivalent} to $B$ if $A = \sigma(QBQ^*)$ for some $Q \in U_n(R)$ and some $\sigma \in \operatorname{Gal}(K/\Q)$, where $\sigma$ is applied componentwise to $QBQ^*$.
	The matrices $A$ and $B$ are merely called \textbf{equivalent} if $A$ is strongly equivalent to $\pm B$.
	The notions of equivalence and strong equivalence carry through to graphs in the natural way and, since all possible labellings of a graph are strongly equivalent, we do not need to label the vertices, i.e., we do not need to assign an order to the vertices.
	
	Let $S$ be a subset of $\mathcal O_K$.
	When working with $S$-matrices and $S$-graphs, for our notion of equivalence, we take $R$ to be the ring generated by the elements of $S$.
	
	% \begin{lemma}\label{lem:bipartiteEquiv}
	% 	Let $G$ be a bipartite $S$-graph.
	% 	Then $G$ is strongly equivalent to $-G$.
	% \end{lemma}
	% \begin{proof}
	% 	We can partition $G$ into two sets of vertices $V_1$ and $V_2$, with all adjacencies going from a vertex in $V_i$ to a vertex in $V_j$ and $i \ne j$. Perform a $(-1)$-switching on all the vertices in $V_1$.
	% \end{proof}
	
	\subsection{Interlacing and maximal indecomposable cyclotomic matrices}
	% \label{sec:interlacing}

	We use repeatedly the following theorem of Cauchy \cite{Cau:Interlace,Fisk:Interlace05, Hwang:Interlace04}.

	\begin{theorem}[Interlacing] \label{thm:interlacing} Let $A$ be an $n~\times~n$ Hermitian matrix having eigenvalues $\lambda_1 \leqslant \dots \leqslant \lambda_n$. Let $B$ be an $(n - 1)~\times~(n - 1)$ principal submatrix of $A$ having eigenvalues $\mu_1 \leqslant \dots \leqslant \mu_{n-1}$. Then the eigenvalues of $A$ and $B$ interlace. Namely,
		\[
			\lambda_1 \leqslant \mu_1 \leqslant \lambda_2 \leqslant \mu_2 \leqslant \dots \leqslant \mu_{n-1} \leqslant \lambda_n.
		\]
	\end{theorem}
	
	A matrix that is equivalent to a block diagonal matrix of more than one block is called \textbf{decomposable}, otherwise it is called \textbf{indecomposable}.
	A matrix is indecomposable if and only if its underlying graph is connected.
	The eigenvalues of a decomposable matrix are found by pooling together the eigenvalues of its blocks.
	It is therefore sufficient to restrict our classification of cyclotomic matrices to indecomposable matrices.
	An indecomposable cyclotomic matrix that is not a principal submatrix of any other indecomposable cyclotomic matrix is called a \textbf{maximal} indecomposable cyclotomic matrix.
	The corresponding graph is called a maximal connected cyclotomic graph.
	% A graph corresponding to a maximal connected cyclotomic matrix is also simply called maximal.
	
	Define the \textbf{degree} of a vertex $v \in V(G)$ as
	\[
		\sum_{u \in V(G)} |w(u,v)|^2.
	\]
	
	\begin{lemma}\label{lem:maxDeg4}
		Let $G$ be a graph with a vertex $v$ of degree $d > 4$.
		Then $G$ is not cyclotomic.
	\end{lemma}
	\begin{proof}
		Let $A$ be an adjacency matrix of $G$ with $v$ corresponding to the first row.
		The first entry of the first row of $A^2$ is $d$.
		Therefore, by interlacing, the largest eigenvalue of $A^2$ is at least $d$, and so the largest modulus of the eigenvalues of $A$ is at least $|\sqrt{d}| > 2$.
	\end{proof}
	
	\section{Main Results}
	\label{sec:results}
	
	% Now we state our results.
	\subsection{Classification of cyclotomic matrices over $\Z[i]$}
	
	We split up the classification of cyclotomic matrices over $\Z[i]$ into three parts and prove each part separately.
	
	% \input{ZmaxCyclos}
	% % We split up the proof of the classification to deal with restricted sets of matrices.
	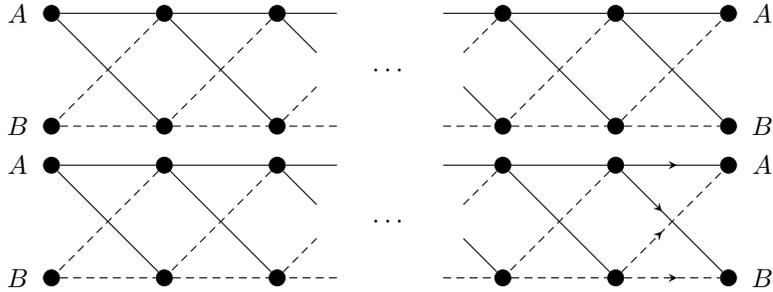
\begin{figure}[htbp]
		\centering
			\begin{tikzpicture}[auto, scale=1.5]
				\begin{scope}
					\foreach \type/\pos/\name in {{vertex/(0,0)/a2}, {vertex/(0,1)/a1}, {vertex/(1,1)/b1}, {vertex/(1,0)/b2}, {vertex/(2,0)/e2}, {vertex/(2,1)/e1}, {empty/(2.6,1)/b11}, {empty/(2.6,0)/b21}, {empty/(2.4,0.6)/b12}, {empty/(2.4,0.4)/b22}, {empty/(3.4,1)/c11}, {empty/(3.4,0)/c21}, {empty/(3.6,0.6)/c12}, {empty/(3.6,0.4)/c22}, {vertex/(4,1)/c1}, {vertex/(4,0)/c2}, {vertex/(5,1)/d1}, {vertex/(5,0)/d2}, {vertex/(6,1)/f1}, {vertex/(6,0)/f2}}
						\node[\type] (\name) at \pos {};
					\foreach \pos/\name in {{(3,0.5)/\dots}, {(-.3,1)/A}, {(-0.3,0)/B}, {(6.3,1)/A}, {(6.3,0)/B}}
						\node at \pos {$\name$};
					\foreach \edgetype/\source/ \dest/\weight in {nedge/b1/a2/{}, pedge/a1/b1/{}, pedge/a1/b2/{}, nedge/a2/b2/{}, nedge/e1/b2/{}, pedge/b1/e1/{}, pedge/b1/e2/{}, nedge/b2/e2/{}, nedge/b21/e2/{}, pedge/e1/b11/{}, pedge/e1/b12/{}, nedge/e2/b22/{}, pedge/c11/c1/{}, nedge/c12/c1/{}, pedge/c22/c2/{}, nedge/c21/c2/{}, nedge/d1/c2/{}, pedge/c1/d1/{}, pedge/c1/d2/{}, nedge/c2/d2/{}, pedge/d1/f1/{}, nedge/d2/f2/{}}
						\path[\edgetype] (\source) -- node[weight] {$\weight$} (\dest);
					% \node at (3,-0.4) {$T_{2k} (k \geqslant 3)$};
				\end{scope}
				\begin{scope}
					\foreach \edgetype/\source/ \dest/\weight in {nedge/d2/f1/{}, pedge/d1/f2/{}}
						\path[\edgetype] (\source) -- node[weight] {$\weight$} (\dest);
				\end{scope}
			\end{tikzpicture}
			\begin{tikzpicture}[auto, scale=1.5]
				\begin{scope}
					\foreach \type/\pos/\name in {{vertex/(0,0)/a2}, {vertex/(0,1)/a1}, {vertex/(1,1)/b1}, {vertex/(1,0)/b2}, {vertex/(2,0)/e2}, {vertex/(2,1)/e1}, {empty/(2.6,1)/b11}, {empty/(2.6,0)/b21}, {empty/(2.4,0.6)/b12}, {empty/(2.4,0.4)/b22}, {empty/(3.4,1)/c11}, {empty/(3.4,0)/c21}, {empty/(3.6,0.6)/c12}, {empty/(3.6,0.4)/c22}, {vertex/(4,1)/c1}, {vertex/(4,0)/c2}, {vertex/(5,1)/d1}, {vertex/(5,0)/d2}, {vertex/(6,1)/f1}, {vertex/(6,0)/f2}}
						\node[\type] (\name) at \pos {};
					\foreach \pos/\name in {{(3,0.5)/\dots}, {(-.3,1)/A}, {(-0.3,0)/B}, {(6.3,1)/A}, {(6.3,0)/B}}
						\node at \pos {$\name$};
					\foreach \edgetype/\source/ \dest/\weight in {nedge/b1/a2/{}, pedge/a1/b1/{}, pedge/a1/b2/{}, nedge/a2/b2/{}, nedge/e1/b2/{}, pedge/b1/e1/{}, pedge/b1/e2/{}, nedge/b2/e2/{}, nedge/b21/e2/{}, pedge/e1/b11/{}, pedge/e1/b12/{}, nedge/e2/b22/{}, pedge/c11/c1/{}, nedge/c12/c1/{}, pedge/c22/c2/{}, nedge/c21/c2/{}, nedge/d1/c2/{}, pedge/c1/d1/{}, pedge/c1/d2/{}, nedge/c2/d2/{}, wedge/d1/f1/{}, wnedge/d2/f2/{}}
						\path[\edgetype] (\source) -- node[weight] {$\weight$} (\dest);
					% \node at (3,-0.4) {$T^{(x)}_{2k} (k \geqslant 3)$};
				\end{scope}
				\begin{scope}[decoration={markings,mark = at position 0.4 with {\arrow{stealth} } }]
					\foreach \edgetype/\source/ \dest/\weight in {dnedge/d2/f1/{}, dpedge/d1/f2/{}}
						\path[\edgetype] (\source) -- node[weight] {$\weight$} (\dest);
				\end{scope}
			\end{tikzpicture}
		\caption{The infinite families $T_{2k}$ and $T^{(x)}_{2k}$ (respectively) of $2k$-vertex maximal connected cyclotomic $\Z[x]$-graphs, for $k \geqslant 3$ and $x \in \{i, \omega \}$. (The two copies of vertices $A$ and $B$ should be identified to give a toral tessellation.) }
		\label{fig:maxcycs1}
	\end{figure}

	\begin{figure}[htbp]
		\centering
			\begin{tikzpicture}[scale=1.5, auto]
				\begin{scope}
					\foreach \type/\pos/\name in {{vertex/(1,1)/b1}, {vertex/(1,0)/b2}, {vertex/(2,0)/e2}, {vertex/(2,1)/e1}, {empty/(2.6,1)/b11}, {empty/(2.6,0)/b21}, {empty/(2.4,0.6)/b12}, {empty/(2.4,0.4)/b22}, {empty/(3.4,1)/c11}, {empty/(3.4,0)/c21}, {empty/(3.6,0.6)/c12}, {empty/(3.6,0.4)/c22}, {vertex/(4,1)/c1}, {vertex/(4,0)/c2}, {vertex/(5,1)/d1}, {vertex/(5,0)/d2}}
						\node[\type] (\name) at \pos {};
					\foreach \type/\pos/\name in {{vertex/(0,0.5)/bgn}, {vertex/(6,0.5)/end}}
						\node[\type] (\name) at \pos {};
					\foreach \pos/\name in {{(3,0.5)/\dots}}
						\node at \pos {\name};
					\foreach \edgetype/\source/ \dest in {nedge/e1/b2, pedge/b1/e1, pedge/b1/e2, nedge/b2/e2, nedge/b21/e2, pedge/e1/b11, pedge/e1/b12, nedge/e2/b22, pedge/c11/c1, nedge/c12/c1, pedge/c22/c2, nedge/c21/c2, nedge/d1/c2, pedge/c1/d1, pedge/c1/d2, nedge/c2/d2}
						\path[\edgetype] (\source) -- (\dest);
					\foreach \edgetype/\source/\dest in {wwedge/b1/bgn, wwedge/b2/bgn, wwnedge/d2/end, wwedge/d1/end}
						\path[\edgetype] (\source) -- (\dest);
					% \node at (3,-0.4) {$C_{2k} (k \geqslant 2)$};
				\end{scope}
			\end{tikzpicture}
		\caption{The infinite family of $2k$-vertex maximal connected cyclotomic $\Z[i]$-graphs $C_{2k}$ for $k \geqslant 2$.}
		\label{fig:maxcycs2}
	\end{figure}
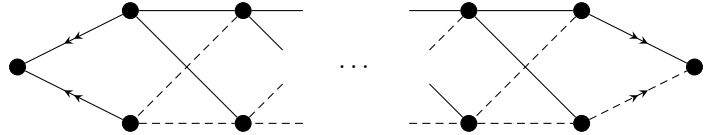

	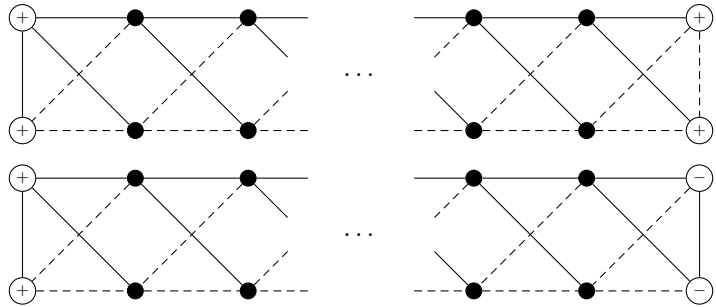
\begin{figure}[htbp]
		\centering
			\begin{tikzpicture}[scale=1.5, auto]
					\foreach \type/\pos/\name in {{vertex/(1,1)/b1}, {vertex/(1,0)/b2}, {vertex/(2,0)/e2}, {vertex/(2,1)/e1}, {empty/(2.6,1)/b11}, {empty/(2.6,0)/b21}, {empty/(2.4,0.6)/b12}, {empty/(2.4,0.4)/b22}, {empty/(3.4,1)/c11}, {empty/(3.4,0)/c21}, {empty/(3.6,0.6)/c12}, {empty/(3.6,0.4)/c22}, {vertex/(4,1)/c1}, {vertex/(4,0)/c2}, {vertex/(5,1)/d1}, {vertex/(5,0)/d2}}
						\node[\type] (\name) at \pos {};
					\foreach \type/\pos/\name in {{pc/(0,0)/a2}, {pc/(0,1)/a1}, {pc/(6,1)/f1}, {pc/(6,0)/f2}}
						\node[\type] (\name) at \pos {$+$};
					\foreach \pos/\name in {{(3,0.5)/\dots}}
						\node at \pos {$\name$};
					\foreach \edgetype/\source/ \dest in {pedge/a1/a2, nedge/b1/a2, pedge/a1/b1, pedge/a1/b2, nedge/a2/b2, nedge/e1/b2, pedge/b1/e1, pedge/b1/e2, nedge/b2/e2, nedge/b21/e2, pedge/e1/b11, pedge/e1/b12, nedge/e2/b22, pedge/c11/c1, nedge/c12/c1, pedge/c22/c2, nedge/c21/c2, nedge/d1/c2, pedge/c1/d1, pedge/c1/d2, nedge/c2/d2, nedge/f1/d2, pedge/d1/f1, pedge/d1/f2, nedge/d2/f2, nedge/f1/f2}
						\path[\edgetype] (\source) -- (\dest);
					% \node at (3,-0.4) {$C_{2k}^{++} (k \geqslant 2)$};
					\node at (3,-0.2) {};
				\end{tikzpicture}
				\begin{tikzpicture}[scale=1.5, auto]
					\foreach \type/\pos/\name in {{vertex/(1,1)/b1}, {vertex/(1,0)/b2}, {vertex/(2,0)/e2}, {vertex/(2,1)/e1}, {empty/(2.6,1)/b11}, {empty/(2.6,0)/b21}, {empty/(2.4,0.6)/b12}, {empty/(2.4,0.4)/b22}, {empty/(3.4,1)/c11}, {empty/(3.4,0)/c21}, {empty/(3.6,0.6)/c12}, {empty/(3.6,0.4)/c22}, {vertex/(4,1)/c1}, {vertex/(4,0)/c2}, {vertex/(5,1)/d1}, {vertex/(5,0)/d2}}
						\node[\type] (\name) at \pos {};
					\foreach \type/\pos/\name in {{+/(0,0)/a2}, {+/(0,1)/a1}, {-/(6,1)/f1}, {-/(6,0)/f2}}
						\node[pc] (\name) at \pos {$\type$};
					\foreach \pos/\name in {{(3,0.5)/\dots}}
						\node at \pos {$\name$};
					\foreach \edgetype/\source/ \dest in {pedge/a1/a2, nedge/b1/a2, pedge/a1/b1, pedge/a1/b2, nedge/a2/b2, nedge/e1/b2, pedge/b1/e1, pedge/b1/e2, nedge/b2/e2, nedge/b21/e2, pedge/e1/b11, pedge/e1/b12, nedge/e2/b22, pedge/c11/c1, nedge/c12/c1, pedge/c22/c2, nedge/c21/c2, nedge/d1/c2, pedge/c1/d1, pedge/c1/d2, nedge/c2/d2, nedge/f1/d2, pedge/d1/f1, pedge/d1/f2, nedge/d2/f2, pedge/f1/f2}
						\path[\edgetype] (\source) -- (\dest);
					% \node at (3,-0.4) {$C_{2k}^{+-} (k \geqslant 2)$};
			\end{tikzpicture}
		\caption{The infinite families of $2k$-vertex maximal connected cyclotomic $\Z$-graphs $C_{2k}^{++}$ and $C_{2k}^{+-}$ for $k \geqslant 2$.}
		\label{fig:maxcycs3}
	\end{figure}

	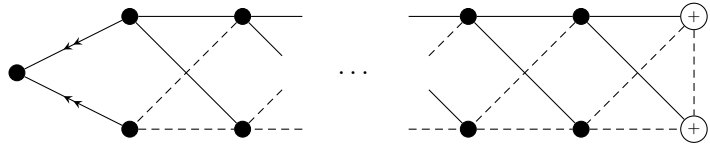
\begin{figure}[h!tbp]
		\centering
			\begin{tikzpicture}[scale=1.5, auto]
				\begin{scope}
					\foreach \type/\pos/\name in {{vertex/(0,0.5)/bgn}, {vertex/(1,1)/b1}, {vertex/(1,0)/b2}, {vertex/(2,0)/e2}, {vertex/(2,1)/e1}, {empty/(2.6,1)/b11}, {empty/(2.6,0)/b21}, {empty/(2.4,0.6)/b12}, {empty/(2.4,0.4)/b22}, {empty/(3.4,1)/c11}, {empty/(3.4,0)/c21}, {empty/(3.6,0.6)/c12}, {empty/(3.6,0.4)/c22}, {vertex/(4,1)/c1}, {vertex/(4,0)/c2}, {vertex/(5,1)/d1}, {vertex/(5,0)/d2}}
						\node[\type] (\name) at \pos {};
					\foreach \type/\pos/\name in {{pc/(6,1)/f1}, {pc/(6,0)/f2}}
						\node[\type] (\name) at \pos {$+$};
					\foreach \pos/\name in {{(3,0.5)/\dots}}
						\node at \pos {$\name$};
					\foreach \edgetype/\source/ \dest in {nedge/e1/b2, pedge/b1/e1, pedge/b1/e2, nedge/b2/e2, nedge/b21/e2, pedge/e1/b11, pedge/e1/b12, nedge/e2/b22, pedge/c11/c1, nedge/c12/c1, pedge/c22/c2, nedge/c21/c2, nedge/d1/c2, pedge/c1/d1, pedge/c1/d2, nedge/c2/d2, nedge/f1/d2, pedge/d1/f1, pedge/d1/f2, nedge/d2/f2, nedge/f1/f2}
						\path[\edgetype] (\source) -- (\dest);
					\foreach \edgetype/\source/\dest in {wwedge/b1/bgn, wwedge/b2/bgn}
						\path[\edgetype] (\source) -- (\dest);
					% \node at (3,-0.4) {$C_{2k+1} (k \geqslant 1)$};
				\end{scope}
			\end{tikzpicture}
		\caption{The infinite family of $(2k+1)$-vertex maximal connected cyclotomic $\Z[i]$-graphs $C_{2k+1}$ for $k \geqslant 1$.}
		\label{fig:maxcycs4}
	\end{figure}
	
	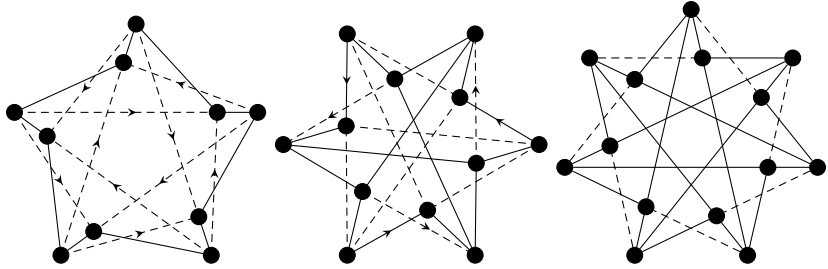
\begin{figure}[htbp]
		\centering
		\begin{tikzpicture}
			\newdimen\rad
			\rad=1.7cm
			\newdimen\radi
			\radi=1.2cm

		    % Indicate the boundary of the regular polygons
			\foreach \x in {90,162,234,306,378}
			{
		    	\draw (\x:\rad) node[vertex] {};
				\draw (\x+8:\radi) node[vertex] {};
				\draw[pedge] (\x:\rad) -- (\x+8:\radi);
				\draw[pedge] (\x:\rad) -- (\x-72+8:\radi);
				\draw[wnedge2] (\x:\rad) -- (\x+72+8:\radi);
				\draw[wnedge2] (\x:\rad) -- (\x+216+8:\radi);
		    }
		\end{tikzpicture}
		\begin{tikzpicture}
			\newdimen\rad
			\rad=1.7cm
			\newdimen\radi
			\radi=0.9cm
			\def\shift{344}
		    % Indicate the boundary of the regular polygons
			\foreach \x in {0,120,240}
			{
		    	\draw (\x:\rad) node[vertex] {};
				\draw (\x+\shift:\radi) node[vertex] {};
				\draw[pedge] (\x:\rad) -- (\x+\shift:\radi);
				\draw[nedge] (\x:\rad) -- (\x-60+\shift:\radi);
				\draw[nedge] (\x:\rad) -- (\x+180+\shift:\radi);
				\draw[wedge] (\x:\rad) -- (\x+60+\shift:\radi);
		    }
			\foreach \x in {60,180,300}
			{
		    	\draw (\x:\rad) node[vertex] {};
				\draw (\x+\shift:\radi) node[vertex] {};
				\draw[pedge] (\x:\rad) -- (\x+\shift:\radi);
				\draw[pedge] (\x:\rad) -- (\x+60+\shift:\radi);
				\draw[pedge] (\x:\rad) -- (\x+180+\shift:\radi);
				\draw[wnedge2] (\x-60+\shift:\radi) -- (\x:\rad);
		    }
		\end{tikzpicture}
			\begin{tikzpicture}
				\begin{scope}[auto, scale=1.5]
					\foreach \type/\pos/\name in {{vertex/(0,0)/a}, {vertex/(1,0)/b}, {vertex/(1.62,0.78)/c}, {vertex/(1.4,1.75)/d}, {vertex/(0.5,2.18)/e}, {vertex/(-0.4,1.75)/f}, {vertex/(-0.62,0.78)/g}, {vertex/(0.1,0.43)/t}, {vertex/(0.725,0.35)/u}, {vertex/(1.18,0.78)/v}, {vertex/(1.122,1.4)/w}, {vertex/(0.6,1.75)/x}, {vertex/(0,1.56)/y}, {vertex/(-0.22,0.97)/z}}
						\node[\type] (\name) at \pos {};
					\foreach \edgetype/\source/ \dest in {pedge/a/e, nedge/a/z, pedge/a/d, pedge/a/u, nedge/b/t, pedge/b/f, pedge/b/e, pedge/b/v, nedge/c/u, pedge/c/g, pedge/c/f, pedge/c/w, nedge/d/v, pedge/d/g, pedge/d/x, nedge/e/w, pedge/e/y, nedge/f/x, pedge/f/z, nedge/g/y, pedge/g/t}
						\path[\edgetype] (\source) -- (\dest);
				\end{scope}
			\end{tikzpicture}
	 	\caption{The sporadic maximal connected cyclotomic $\Z[\omega]$-graphs $S_{10}$, $S_{12}$, and $S_{14}$ of orders $10$, $12$, and $14$ respectively. The $\Z$-graph $S_{14}$ is also a $\Z[i]$-graph.}
		\label{fig:maxcycs5}
	\end{figure}

	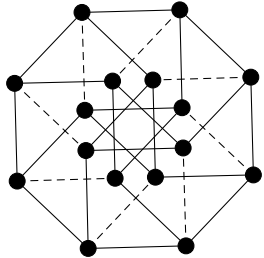
\begin{figure}[htbp]
		\centering
		\begin{tikzpicture}
			\newdimen\rad
			\rad=1.7cm
			\newdimen\radi
			\radi=0.7cm
			% \def\shift{0}
		    % Indicate the boundary of the regular polygons
			\foreach \x in {69,114,159,204,249,294,339,384}
			{
				\draw (\x:\radi) node[vertex] {};
				% \draw[pedge] (\x:\radi) -- (\x+225:\radi);
				\draw[pedge] (\x:\radi) -- (\x+135:\radi);
		    }
			\foreach \x in {69,159,249,339}
			{
		    	\draw (\x:\rad) node[vertex] {};
				\draw[pedge] (\x:\rad) -- (\x+45:\rad);
				\draw[nedge] (\x:\rad) -- (\x+45:\radi);
				\draw[pedge] (\x:\rad) -- (\x-45:\radi);
		    }
			\foreach \x in {114,204,294,384}
			{
		    	\draw (\x:\rad) node[vertex] {};
				\draw[pedge] (\x:\rad) -- (\x+45:\rad);
				\draw[nedge] (\x:\rad) -- (\x+45:\radi);
				\draw[pedge] (\x:\rad) -- (\x-45:\radi);
		    }
		\end{tikzpicture}
		\caption{The sporadic maximal connected cyclotomic $\Z$-hypercube $S_{16}$.}
		\label{fig:maxcycs6}
	\end{figure}

	\begin{figure}[htbp]
		\centering
		\begin{tikzpicture}[scale=1.5, auto]
		\begin{scope}	
			\foreach \pos/\name/\type/\charge in {{(0,0.4)/a/pc/{2}}}
				\node[\type] (\name) at \pos {$\charge$}; % {$\name$};
			\node at (0,-0.4) {$S_1$};
		\end{scope}
		\end{tikzpicture}
		\begin{tikzpicture}[scale=1.4]
		\begin{scope}	
			\foreach \pos/\name/\type/\charge in {{(0,0)/a/zc/{}}, {(0,1)/b/zc/{}}}
				\node[\type] (\name) at \pos {$\charge$}; % {$\name$};
			\foreach \edgetype/\source/ \dest in {pedge/a/b}
			\path[\edgetype] (\source) -- node[weight2] {$2$} (\dest);
			\node at (0,-0.4) {$S_2$};
		\end{scope}
		\end{tikzpicture}
		\begin{tikzpicture}[scale=1.3, auto]
		\begin{scope}	
			\foreach \pos/\name/\type/\charge in {{(0,0)/a/pc/-}, {(0,1)/b/nc/+}}
				\node[\type] (\name) at \pos {$\charge$}; % {$\name$};
			\foreach \edgetype/\source/ \dest / \weight in {wwwedge/b/a/{}}
			\path[\edgetype] (\source) -- node[weight] {$\weight$} (\dest);
			\node at (0,-0.4) {$S_2^\dag$};
		\end{scope}
		\end{tikzpicture}
		\begin{tikzpicture}[scale=1.3, auto]
		\begin{scope}	
			\foreach \pos/\name/\type/\charge in {{(0,0)/a/zc/{}}, {(0,1)/b/zc/{}}, {(1,0)/c/zc/{}}, {(1,1)/d/zc/{}}}
				\node[\type] (\name) at \pos {$\charge$}; % {$\name$};
			\foreach \edgetype/\source/ \dest / \weight in {wwwnedge/a/c/{}, wwwedge/b/d/{}, pedge/a/b/{}, pedge/c/d/{}}
			\path[\edgetype] (\source) -- node[weight] {$\weight$} (\dest);
		\end{scope}
		\node at (0.5,-0.4) {$S_4^\ddag$};
		\end{tikzpicture}
		\begin{tikzpicture}[scale=1, auto]
		\begin{scope}	
			\foreach \pos/\name/\type/\charge in {{(-0.3,0.9)/a/zc/{}}, {(1.7,0.9)/b/zc/{}}, {(0,0)/c/pc/+}, {(1.4,0)/d/pc/+}, {(0.7,1.7)/e/zc/{}}}
				\node[\type] (\name) at \pos {$\charge$}; % {$\name$};
			\foreach \edgetype/\source/ \dest / \weight in {wnedge/a/c/{}, wnedge/b/d/{}, pedge/c/e/{}, pedge/e/d/{}, pedge/a/d/{}, pedge/b/c/{}, wnedge/e/b/{}, pedge/a/b/{}, wnedge/e/a/{}}
			\path[\edgetype] (\source) -- node[weight] {$\weight$} (\dest);
			\node at (0.7,-0.4) {$S_5$};
		\end{scope}
		\end{tikzpicture}
		\begin{tikzpicture}[scale=1.3, auto]
		\begin{scope}
			\foreach \pos/\name/\sign/\charge in {{(-0.1,0)/a/pc/-}, {(0.7,0)/b/nc/+}, {(1.2,0.5)/c/pc/-}, {(-0.6,0.5)/d/nc/+}, {(-0.1,1)/e/pc/-}, {(0.7,1)/f/nc/+}}
				\node[\sign] (\name) at \pos {$\charge$};
			\foreach \edgetype/\source/\dest/\weight in {{wnedge/b/a/{}}, {wnedge/a/d/{}}, {wnedge/f/e/{}}, {pedge/b/c/{}}, {wnedge/e/d/{}}, {pedge/f/c/{}}, {pedge/d/c/{}}, {pedge/a/f/{}}, {pedge/b/e/{}}}
				\path[\edgetype] (\source) -- node[weight] {$\weight$} (\dest);
				\node at (0.3,-0.4) {$S_6$};
		\end{scope}
		\end{tikzpicture}
		\begin{tikzpicture}[scale=1.3, auto]
		\begin{scope}
			\foreach \pos/\name/\sign/\charge in {{(-0.1,0)/a/zc/{}}, {(0.7,0)/b/zc/{}}, {(1.2,0.5)/c/nc/-}, {(-0.6,0.5)/d/pc/+}, {(-0.1,1)/e/zc/{}}, {(0.7,1)/f/zc/{}}}
				\node[\sign] (\name) at \pos {$\charge$};
			\foreach \edgetype/\source/\dest/\weight in {{wnedge/b/a/{}}, {wnedge/a/d/{}}, {wnedge/f/e/{}}, {pedge/b/c/{}}, {wnedge/e/d/{}}, {pedge/f/c/{}}, {pedge/d/c/{}}, {nedge/e/a/{}}, {pedge/b/f/{}}, {pedge/a/f/{}}, {pedge/b/e/{}}}
				\path[\edgetype] (\source) -- node[weight] {$\weight$} (\dest);
				\node at (0.3,-0.4) {$S_6^\dag$};
		\end{scope}
		\end{tikzpicture}
		\caption{The sporadic maximal connected cyclotomic $\Z[\omega]$-graphs of orders $1$, $2$, $4$, $5$, and $6$. The $\Z$-graphs $S_1$ and $S_2$ are also $\Z[i]$-graphs.}
		\label{fig:maxcycs8}
	\end{figure}
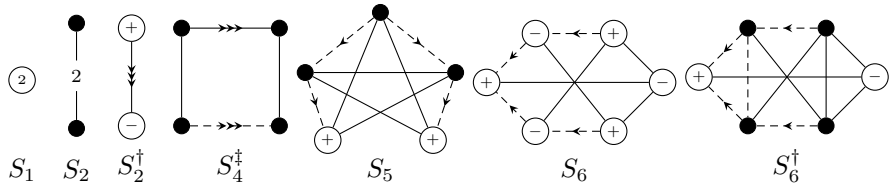

	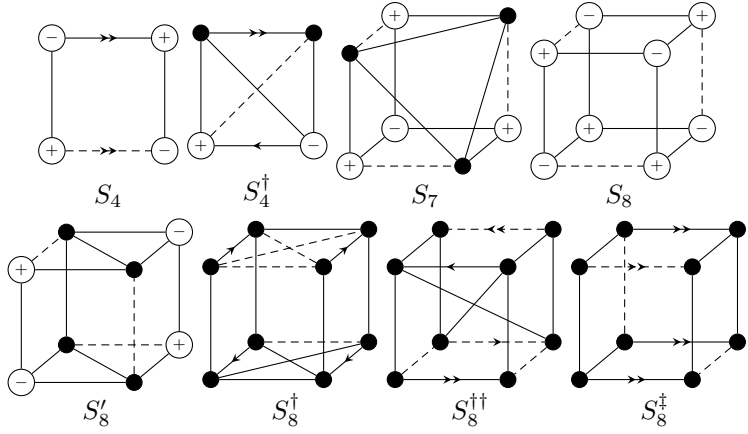
\begin{figure}[htbp]
		\centering
		\begin{tikzpicture}[scale=1.5, auto]
		\begin{scope}	
			\foreach \pos/\name/\type/\charge in {{(0,0)/a/pc/+}, {(0,1)/b/nc/-}, {(1,0)/c/nc/-}, {(1,1)/d/pc/+}}
				\node[\type] (\name) at \pos {$\charge$}; % {$\name$};
			\foreach \edgetype/\source/ \dest / \weight in {wwnedge/a/c/{}, wwedge/b/d/{}, pedge/b/a/{}, pedge/c/d/{}}
			\path[\edgetype] (\source) -- node[weight] {$\weight$} (\dest);
			\node at (0.5,-0.4) {$S_4$};
		\end{scope}
		\end{tikzpicture}
		\begin{tikzpicture}[scale=1.5, auto]
		\begin{scope}	
			\foreach \pos/\name/\type/\charge in {{(0,0)/a/pc/+}, {(0,1)/b/zc/{}}, {(1,0)/c/nc/-}, {(1,1)/d/zc/{}}}
				\node[\type] (\name) at \pos {$\charge$}; % {$\name$};
			\foreach \edgetype/\source/ \dest / \weight in {wedge/c/a/{}, wwedge/b/d/{}, pedge/b/a/{}, pedge/c/d/{}, nedge/a/d/{}, pedge/c/b/{}}
			\path[\edgetype] (\source) -- node[weight] {$\weight$} (\dest);
			\node at (0.5,-0.4) {$S_4^\dag$};
		\end{scope}
		\end{tikzpicture}
			\begin{tikzpicture}
				\def\XthreeDadj{0.6}
				\def\YthreeDadj{0.5}
				\begin{scope}
					\foreach \pos/\name/\sign/\charge in {{(0,0)/a/pc/+}, {(0,1.5)/b/zc/{}}, {(1.5,0)/c/zc/{}}, {(0 + \XthreeDadj,0 + \YthreeDadj)/d/nc/-}, {(0 + \XthreeDadj,1.5 + \YthreeDadj)/e/pc/+}, {(1.5 + \XthreeDadj,0 + \YthreeDadj)/f/pc/+}, {(1.5 + \XthreeDadj,1.5 + \YthreeDadj)/g/zc/{}}}
						\node[\sign] (\name) at \pos {$\charge$}; % {$\name$};
					\foreach \edgetype/\source/ \dest in {pedge/b/a, nedge/a/c, pedge/a/d, pedge/c/f, nedge/g/f, pedge/b/g, pedge/d/e, pedge/d/f, nedge/b/e, pedge/c/g, pedge/e/g, pedge/b/c}
					\path[\edgetype] (\source) -- (\dest);
				\node at (1,-0.4) {$S_{7}$};
				\end{scope}
				\end{tikzpicture}
				\begin{tikzpicture}
					\def\XthreeDadj{0.6}
					\def\YthreeDadj{0.5}
				\begin{scope}
					\foreach \pos/\name/\sign/\charge in {{(0,0)/a/nc/-}, {(0,1.5)/b/pc/+}, {(1.5,0)/c/pc/+}, {(1.5,1.5)/d/nc/-}, {(0 + \XthreeDadj,0 + \YthreeDadj)/e/pc/+}, {(0 + \XthreeDadj,1.5 + \YthreeDadj)/f/nc/-}, {(1.5 + \XthreeDadj,0 + \YthreeDadj)/g/nc/-}, {(1.5 + \XthreeDadj,1.5 + \YthreeDadj)/h/pc/+}}
						\node[\sign] (\name) at \pos {$\charge$}; % {$\name$};
					\foreach \edgetype/\source/ \dest in {pedge/b/a, nedge/a/c, pedge/a/e, pedge/c/g, pedge/c/d, nedge/b/f, pedge/b/d, pedge/e/f, pedge/e/g, nedge/h/g, pedge/f/h, pedge/d/h}
					\path[\edgetype] (\source) -- (\dest);
				\node at (1,-0.4) {$S_{8}$};
				\end{scope}
				\end{tikzpicture}

				\begin{tikzpicture}
					\def\XthreeDadj{0.6}
					\def\YthreeDadj{0.5}
				\begin{scope}	
					\foreach \pos/\name/\sign/\charge in {{(0,0)/a/nc/-}, {(0,1.5)/b/pc/+}, {(1.5,0)/c/zc/{}}, {(1.5,1.5)/d/zc/{}}, {(0 + \XthreeDadj,0 + \YthreeDadj)/e/zc/{}}, {(0 + \XthreeDadj,1.5 + \YthreeDadj)/f/zc/{}}, {(1.5 + \XthreeDadj,0 + \YthreeDadj)/g/pc/+}, {(1.5 + \XthreeDadj,1.5 + \YthreeDadj)/h/nc/-}}
						\node[\sign] (\name) at \pos {$\charge$}; % {$\name$};
					\foreach \edgetype/\source/ \dest in {pedge/b/a, pedge/a/c, pedge/a/e, pedge/c/g, nedge/c/d, nedge/b/f, pedge/b/d, pedge/e/f, nedge/e/g, pedge/h/g, pedge/f/h, pedge/d/h, pedge/d/f, pedge/c/e}
					\path[\edgetype] (\source) -- (\dest);
				\node at (1,-0.4) {$S^\prime_{8}$};
				\end{scope}
				\end{tikzpicture}
			\begin{tikzpicture}
				\def\XthreeDadj{0.6}
				\def\YthreeDadj{0.5}
			\begin{scope}	
				\foreach \pos/\name/\sign/\charge in {{(0,0)/a/zc/{}}, {(0,1.5)/b/zc/{}}, {(1.5,0)/c/zc/{}}, {(1.5,1.5)/d/zc/{}}, {(0 + \XthreeDadj,0 + \YthreeDadj)/e/zc/{}}, {(0 + \XthreeDadj,1.5 + \YthreeDadj)/f/zc/{}}, {(1.5 + \XthreeDadj,0 + \YthreeDadj)/g/zc/{}}, {(1.5 + \XthreeDadj,1.5 + \YthreeDadj)/h/zc/{}}}
					\node[\sign] (\name) at \pos {$\charge$}; % {$\name$};
				\foreach \edgetype/\source/ \dest/\weight in {pedge/b/a/{}, pedge/a/c/{}, pedge/a/g/{}, nedge/b/h/{}, wedge/e/a/{}, wedge/g/c/{}, pedge/c/d/{}, wedge/b/f/{}, nedge/b/d/{}, pedge/e/f/{}, nedge/e/g/{}, pedge/h/g/{}, pedge/f/h/{}, wedge/d/h/{}, nedge/d/f/{}, pedge/c/e/{}}
				\path[\edgetype] (\source) -- node[weight] {$\weight$} (\dest);
				\node at (1,-0.4) {$S^{\dagger}_{8}$};
			\end{scope}
			\end{tikzpicture}
			\begin{tikzpicture}
				\def\XthreeDadj{0.6}
				\def\YthreeDadj{0.5}
			\begin{scope}	
				\foreach \pos/\name/\sign/\charge in {{(0,0)/a/zc/{}}, {(0,1.5)/b/zc/{}}, {(1.5,0)/c/zc/{}}, {(1.5,1.5)/d/zc/{}}, {(0 + \XthreeDadj,0 + \YthreeDadj)/e/zc/{}}, {(0 + \XthreeDadj,1.5 + \YthreeDadj)/f/zc/{}}, {(1.5 + \XthreeDadj,0 + \YthreeDadj)/g/zc/{}}, {(1.5 + \XthreeDadj,1.5 + \YthreeDadj)/h/zc/{}}}
					\node[\sign] (\name) at \pos {$\charge$}; % {$\name$};
				\foreach \edgetype/\source/ \dest/\weight in {pedge/b/a/{}, wwedge/a/c/{}, nedge/a/e/{}, nedge/g/c/{}, pedge/d/c/{}, pedge/b/f/{}, wedge/d/b/{}, pedge/e/f/{}, wnedge/e/g/{}, pedge/h/g/{}, wwnedge/h/f/{}, pedge/h/d/{}, pedge/d/e/{}, pedge/b/g/{}}
				\path[\edgetype] (\source) -- node[weight] {$\weight$} (\dest);
				\node at (1,-0.4) {$S^{\dagger \dagger}_{8}$};
			\end{scope}
		\end{tikzpicture}
		\begin{tikzpicture}
			\def\XthreeDadj{0.6}
			\def\YthreeDadj{0.5}
		\begin{scope}	
			\foreach \pos/\name/\sign/\charge in {{(0,0)/a/zc/{}}, {(0,1.5)/b/zc/{}}, {(1.5,0)/c/zc/{}}, {(1.5,1.5)/d/zc/{}}, {(0 + \XthreeDadj,0 + \YthreeDadj)/e/zc/{}}, {(0 + \XthreeDadj,1.5 + \YthreeDadj)/f/zc/{}}, {(1.5 + \XthreeDadj,0 + \YthreeDadj)/g/zc/{}}, {(1.5 + \XthreeDadj,1.5 + \YthreeDadj)/h/zc/{}}}
				\node[\sign] (\name) at \pos {$\charge$}; % {$\name$};
			\foreach \edgetype/\source/ \dest/\weight in {pedge/b/a/{}, wwedge/a/c/{}, pedge/a/e/{}, nedge/g/c/{}, pedge/d/c/{}, pedge/b/f/{}, wwnedge/b/d/{}, nedge/e/f/{}, wwedge/e/g/{}, pedge/h/g/{}, wwedge/f/h/{}, pedge/h/d/{}}
			\path[\edgetype] (\source) -- node[weight] {$\weight$} (\dest);
			\node at (1,-0.4) {$S^{\ddag}_{8}$};
		\end{scope}
	\end{tikzpicture}

		\caption{The sporadic maximal connected cyclotomic $\Z[i]$-graphs of orders $4$, $7$, and $8$. The $\Z$-graphs $S_7$, $S_8$, and $S_8^\prime$ are also $\Z[\omega]$-graphs.}
		\label{fig:maxcycs10}
	\end{figure}

	\begin{theorem}[$\Z[i{]}$ uncharged unit entries]
		\label{thm:classununzi}
		Let $A$ be a maximal indecomposable cyclotomic matrix that has only zeros on the diagonal and whose nonzero entries are units from the ring $\Z[i]$.
		Then $A$ is equivalent to an adjacency matrix of one of the graphs $T_{2k}$, $T_{2k}^{(i)}$, $S_8^\dag$, $S_{14}$, and $S_{16}$ in Figures~\ref{fig:maxcycs1}, \ref{fig:maxcycs5}, \ref{fig:maxcycs6}, and \ref{fig:maxcycs10}.
	\end{theorem}

	\begin{theorem}[$\Z[i{]}$ uncharged]
		\label{thm:classunzi}
		Let $A$ be a maximal indecomposable cyclotomic $\Z[i]$-matrix that has only zeros on the diagonal, and at least one entry of $A$ has norm greater than $1$.
		Then $A$ is equivalent to an adjacency matrix of one of the graphs $C_{2k}$, $S_2$, $S_8^{\dag \dag}$, and $S_8^\ddag$ in Figures~\ref{fig:maxcycs2} and \ref{fig:maxcycs10}.
	\end{theorem}
	
	\begin{theorem}[$\Z[i{]}$ charged]
		\label{thm:classchzi}
		Let $A$ be a maximal indecomposable cyclotomic $\Z[i]$-matrix that has at least one nonzero entry on the diagonal.
		Then $A$ is equivalent to an adjacency matrix of one of the graphs $C_{2k}^{++}$, $C_{2k}^{+-}$, $C_{2k +1}$, $S_1$, $S_4$, $S_4^\dag$, $S_7$, $S_8$, and $S_8^\prime$ in Figures~\ref{fig:maxcycs3}, \ref{fig:maxcycs4}, and \ref{fig:maxcycs10}.
	\end{theorem}

	The theorems above give a complete classification of cyclotomic matrices over the Gaussian integers as follows.

	\begin{theorem}[Cyclotomic matrices over $\Z[i{]}$]
		\label{thm:classzi}
		Let $A$ be a maximal indecomposable cyclotomic matrix over the ring $\Z[i]$.
		Then $A$ is equivalent to an adjacency matrix of one of the graphs from Theorems~\ref{thm:classununzi}, \ref{thm:classunzi}, or \ref{thm:classchzi}.
		
		Moreover, every indecomposable cyclotomic $\Z[i]$-matrix is contained in a maximal one.
	\end{theorem}

	\subsection{Classification of cyclotomic matrices over $\Z[\omega]$}

	As with the classification over the Gaussian integers, we split up the result to deal with uncharged graphs and charged graphs separately.

	\begin{theorem}[$\Z[\omega{]}$ uncharged]
		\label{thm:classunzw}
		Let $A$ be a maximal indecomposable cyclotomic $\Z[\omega]$-matrix that has only zeros on the diagonal.
		Then $A$ is equivalent to an adjacency matrix of one of the graphs $T_{2k}$, $T_{2k}^{(\omega)}$, $S_2$, $S_4^\ddag$, $S_{10}$, $S_{12}$, $S_{14}$, and $S_{16}$ in Figures~\ref{fig:maxcycs1}, \ref{fig:maxcycs5}, \ref{fig:maxcycs6}, \ref{fig:maxcycs8}, and \ref{fig:maxcycs10}.
	\end{theorem}
	
	\begin{theorem}[$\Z[\omega{]}$ charged]
		\label{thm:classchzw}
		Let $A$ be a maximal indecomposable cyclotomic $\Z[\omega]$-matrix that has at least one nonzero entry on the diagonal.
		Then $A$ is equivalent to an adjacency matrix of one of the graphs $S_1$, $S_2^\dag$, $C_{2k}^{++}$, $C_{2k}^{+-}$, $S_4^\dag$, $S_5$, $S_6$, $S_6^\dag$, $S_7$, $S_8$, and $S_8^\prime$ in Figures~\ref{fig:maxcycs3}, \ref{fig:maxcycs8}, and \ref{fig:maxcycs10}.
	\end{theorem}

	Again, the theorems above give a complete classification of cyclotomic matrices over the Eisenstein integers.

	\begin{theorem}[Cyclotomic matrices over $\Z[\omega {]}$]
		\label{thm:classzw}
		Let $A$ be a maximal indecomposable cyclotomic matrix over the ring $\Z[\omega]$.
		Then $A$ is equivalent to an adjacency matrix of one of the graphs from Theorems~\ref{thm:classunzw} or \ref{thm:classchzw}.
		
		Moreover, every indecomposable cyclotomic $\Z[\omega]$-matrix is contained in a maximal one.
	\end{theorem}

	As a corollary we obtain McKee and Smyth's classification \cite{McKee:IntSymCyc07} of cyclotomic integer symmetric matrices, this result is not used in our proof.
	Their work used Gram matrices and the classification of indecomposable line systems \cite{Cam:DGC91}.
	We expand on the use of Gram matrices but we do not make use of line systems.
	% Although, it would be interesting to see how generalising the line system approach relates to the results presented here.
	We also obtain as a consequence an early result of Smith \cite{Smith:CycloG}, which classifies all cyclotomic $\{0,1\}$-graphs.
	
	Following McKee and Smyth \cite{McKee:IntSymCyc07}, we remark that all the maximal connected cyclotomic graphs (with adjacency matrices $A$) of Theorems~\ref{thm:classzi} and \ref{thm:classzw} are `visibly' cyclotomic: $A^2 = 4I$, hence all their eigenvalues are $\pm 2$.
	
	\section{Excluded subgraphs and Gram matrices}
	\label{sec:exGr}
	
	In this section we introduce the main tools used in the classification.
	
	\subsection{Excluding subgraphs}
	\label{sec:exsubgraphs}

	In order to complete their classification of cyclotomic integer matrices, McKee and Smyth \cite{McKee:IntSymCyc07} used a combination of excluded subgraphs and Gram matrices.
	This proves to be a very effective method for our purposes.
	Here, we give an abstract definition of these excluded subgraphs.
	If a graph is not cyclotomic, then by Theorem~\ref{thm:interlacing}, it cannot be a subgraph of a cyclotomic graph.
	We call such a graph an \textbf{excluded subgraph of type I}.

	Certain connected cyclotomic graphs have the property that if one tries to grow them to give larger connected cyclotomic graphs then one always stays inside one of a finite number of fixed maximal connected cyclotomic graphs. 
	We call a graph with this property an \textbf{excluded subgraph of type II}.
	Given a connected cyclotomic graph $G$ and a finite list $L$ of maximal connected cyclotomic graphs containing $G$, we describe the process used to determine whether or not a graph $G$ has this property.
	Consider all possible ways of attaching a vertex to $G$ such that the resulting graph $H$ is both connected and cyclotomic.
	Check that each supergraph $H$ is equivalent to a subgraph of one of the graphs in $L$ (if not then $G$ is \emph{not} an excluded subgraph of type II with respect to the list $L$).
	Repeat this process with all supergraphs $H$.
	Since $L$ is a finite list of graphs on a finite number of vertices, this process terminates.
	The author has written PARI~\cite{PARI} code that can check if a cyclotomic graph is an excluded subgraph of type II.
	This code is available on request.

	Given a list $\mathcal L$ of graphs, we define an \textbf{$\mathcal L$-free} graph to be a connected cyclotomic graph that does not contain any graph equivalent to any graph in $\mathcal L$.
	We have included being both \emph{connected} and \emph{cyclotomic} in this definition to ease the terminology below. 
	We shall have cause to use different lists at various points in our proofs.

	\subsection{Cyclotomic matrices and Gram matrices}
	\label{sec:grammatrix}
	
	Let $S$ be a subset of $\C$ and suppose $G$ is a cyclotomic $S$-graph with adjacency matrix $A$.
	Then all of the eigenvalues of $A$ are contained in the interval $[-2,2]$.
	The matrix $M = A + 2I$ is positive semidefinite and therefore decomposes as $M = B^* B$, where the columns of $B$ are vectors in a unitary space $\C^m$ for some arbitrary $m \in \N$.
	Hence $M$ is the Gram matrix for the set of vectors forming the columns of $B$; these vectors are called \textbf{Gram vectors}.
	Each vertex $v$ of $G$ has a corresponding Gram vector $\mathbf{v}$ and the inner product of Gram vectors $\mathbf{u}$ and $\mathbf{v}$ corresponds to the adjacency of the vertices $u$ and $v$.
	By examining the diagonal of the Gram matrix, one can see that Gram vectors corresponding to uncharged vertices have squared length $2$.
	Similarly, a Gram vector corresponding to a vertex of charge $+1$ (respectively $-1$) has squared length $3$ (respectively $1$).

	\subsubsection{Gram vector constraints}

	In the proof of the classification of cyclotomic matrices we exploit the dependencies of Gram vectors that satisfy certain conditions as outlined in the next lemma.

	\begin{lemma}\label{lem:adeqvec}
		Let $V$ be a unitary space.
		Let $\mathbf{x}_1$, $\mathbf{x}_2$, $\mathbf{x}_3$, and $\mathbf{x}_4$ be pairwise orthogonal vectors from $V$ each having squared length $2$.
		Let $\mathbf{v}$ be a vector of $V$ satisfying
		\[
			|\inprod{\mathbf{v}, \mathbf{x}_1}| = \dots = |\inprod{\mathbf{v}, \mathbf{x}_4}| = 1, \text{ and } \; \inprod{\mathbf{v}, \mathbf{v}} = 2.
		\]
		Then we can write
		\[
			2 \mathbf{v} = \inprod{\mathbf{v}, \mathbf{x}_1}\mathbf{x}_1 + \inprod{\mathbf{v}, \mathbf{x}_2}\mathbf{x}_2 + \inprod{\mathbf{v}, \mathbf{x}_3}\mathbf{x}_3 + \inprod{\mathbf{v}, \mathbf{x}_4}\mathbf{x}_4.
		\]
	\end{lemma}
	\begin{proof}
		With $\lambda_j$s in $\C$, we write 
		\begin{equation}
			\label{eqn:vlem}
			\mathbf{v} = \lambda_1 \mathbf{x}_1 + \lambda_2 \mathbf{x}_2 + \lambda_3 \mathbf{x}_3 + \lambda_4 \mathbf{x}_4 + \xi,
		\end{equation}
		with $\xi$ orthogonal to each $\mathbf{x}_j$.
		Taking inner products with equation	\eqref{eqn:vlem} and each $\mathbf{x}_j$ gives
		\[
			\lambda_j = \frac{\inprod{\mathbf{v}, \mathbf{x}_j}}{2}.
		\]
		Now we write
		\begin{equation}
			\label{eqn:vlem2}
			2 \mathbf{v} = \inprod{\mathbf{v}, \mathbf{x}_1}\mathbf{x}_1 + \inprod{\mathbf{v}, \mathbf{x}_2}\mathbf{x}_2 + \inprod{\mathbf{v}, \mathbf{x}_3}\mathbf{x}_3 + \inprod{\mathbf{v}, \mathbf{x}_4}\mathbf{x}_4 + 2\xi.
		\end{equation}
		By taking the inner product of equation~\eqref{eqn:vlem2} with $\mathbf{v}$, we see that $\inprod{\xi,\mathbf{v}} = 0$ and hence, using \eqref{eqn:vlem}, we have $\xi = 0$.
	\end{proof}

	\subsubsection{Hollow vertices and saturated vertices}

	Let $H$ be a cyclotomic $S$-graph contained in some cyclotomic $S$-graph $H^\prime$.
	Given $H^\prime$ and $H$, we refer to the vertices $V(H^\prime)\backslash V(H)$ as the \textbf{hollow vertices} of $H$.
	For a graph $G$, let $N_G(v)$ denote the set of neighbours of $v$ in $G$, that is, the set of vertices $u \in V(G)$ with nonzero weight $w(u,v)$.
	We define the \textbf{hollow-degree} of a vertex $v \in V(H)$ as 
	\[
		d_{H^\prime}(v) := \sum_{u \in N_{H^\prime}(v)} |w(u,v)|^2.
	\]
	This generalises the degree of a vertex $v \in V(H)$, which is given by $d_H(v)$.
	Let $V^\prime_4(H)$ denote the subset of vertices of $H$ that have hollow-degree $4$, i.e., the set
	\[
		\left \{ v \in V(H) : d_{H^\prime}(v) = 4 \right \}.
	\]

	Since $H$ and $H^\prime$ are cyclotomic, each of their vertices $v$ has a corresponding Gram vector $\mathbf{v}$.
	Our notion of switching carries through to vectors naturally; we say that two vectors $\mathbf{u}$ and $\mathbf{v}$ are \textbf{switch-equivalent} if $\mathbf{u} = x \mathbf{v}$ for some unit $x$.
	Accordingly, the vertices $u$ and $v$ are called \textbf{switch-equivalent} if their corresponding Gram vectors are switch-equivalent.
	Let $G$ be a cyclotomic graph that contains $H$ and, for a vertex $v \in V(G)$, let $N^\prime_G(v)$ denote the set of vertices $u \in N_G(v)$ such that $u$ is switch-equivalent to some vertex in $V(H^\prime)$.
	Define $\mathcal V_G(H)$ to be the subset of $V(G)$ consisting of the vertices of $H$ and their adjacent vertices that are switch-equivalent to hollow vertices, in symbols
	\[
		\mathcal V_G(H) = \bigcup_{v \in V(H)} N^\prime_G(v).
	\]

	Let $\mathcal L$ be a list of graphs.
	A vertex $v \in V(H)$ is called $H^\prime$-\textbf{saturated in} $H$ if, for any $\mathcal L$-free $S$-graph $G$ containing $H$, each vertex in $N_G(v)$ is either in $V(H)$ or is switch-equivalent to some hollow vertex, i.e., $N_G(v) = N^\prime_G(v)$.
	Note that the definition of a vertex being $H^\prime$-saturated in $H$ depends on the set $S$, the list $\mathcal L$, and the graphs $H$ and $H^\prime$.
	Let $\Gamma$ (resp.\ $\Gamma^\prime$) be a cyclotomic $S$-supergraph of $H$ (resp.\ $H^\prime$) and suppose that $\Gamma^\prime$ contains $\Gamma$.
	Then any vertex that is $H^\prime$-saturated in $H$ is also $\Gamma^\prime$-saturated in $\Gamma$.
	We refer imprecisely to these vertices simply as `saturated vertices'.
		
	\section{Proof of Theorem~\ref{thm:classununzi}}
	\label{sec:growZi}
	
	In this section we prove Theorem~\ref{thm:classununzi} and hence we restrict our attention to the set $S = \left \{0, \pm 1, \pm i \right \}$.
	
	\subsection{Excluded subgraphs}
	\label{sec:exgraphs}
	\begin{figure}[htbp]
		\centering
		\begin{tikzpicture}
		\begin{scope}[yshift=-0.5cm]	
			\foreach \pos/\name in {{(0,0)/a}, {(0,1)/b}, {(1,1)/c}, {(1,0)/d},{(2,0)/f}}
				\node[vertex] (\name) at \pos {}; % {$\name$};
			\node at (1,-0.5) {$XA_1$};
			\foreach \edgetype/\source/ \dest /\weight in {pedge/a/b/{}, pedge/b/c/{}, pedge/c/d/{}, pedge/a/d/{}, pedge/d/f/{}}
			\path[\edgetype] (\source) -- node[weight] {$\weight$} (\dest);
		\end{scope}
		\end{tikzpicture}
		\begin{tikzpicture}
			\begin{scope}
				\newdimen\rad
				\rad=0.6cm

			    % Indicate the boundary of the regular polygons
				\foreach \x in {90,162,234,306,378}
				{
			    	\draw (\x:\rad) node[vertex] {};
					\draw[pedge] (\x:\rad) -- (\x+72:\rad);
			    }
				\draw (338:1.3cm) node[vertex] {};
			\draw[pedge] (306:\rad) -- (338:1.3cm);
			\end{scope}
			\begin{scope}[xshift=-0.2cm, yshift=-0.5cm]
				\node at (0.5,-0.5) {$XA_2$};
			\end{scope}
		\end{tikzpicture}
		\caption{some non-cyclotomic uncharged $\Z$-graphs.}
		\label{fig:xgraphsit1}
	\end{figure}
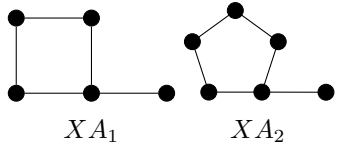

	\begin{figure}[htbp]
		\centering
		\begin{tikzpicture}
		\begin{scope}[yshift=-0.5cm]	
			\foreach \pos/\name in {{(0,0)/a}, {(0,1)/b}, {(1,1)/c}, {(1,0)/d}, {(2,1)/e}, {(2,0)/f}}
				\node[vertex] (\name) at \pos {}; % {$\name$};
			\node at (1,-0.5) {$YA_1$};
			\foreach \edgetype/\source/ \dest /\weight in {pedge/a/b/{}, pedge/b/c/{}, nedge/c/d/{}, pedge/a/d/{}, pedge/d/f/{}, pedge/c/e/{}}
			\path[\edgetype] (\source) -- node[weight] {$\weight$} (\dest);
		\end{scope}
		\end{tikzpicture}
		\begin{tikzpicture}
		\begin{scope}[xshift=2.75cm, yshift=-0.5cm]	
			\foreach \pos/\name in {{(0,0)/a}, {(0,1)/b}, {(1,1)/c}, {(1,0)/d}, {(2,1)/e}, {(2,0)/f}}
				\node[vertex] (\name) at \pos {}; % {$\name$};
			\node at (1,-0.5) {$YA_2$};
			\foreach \edgetype/\source/ \dest /\weight in {pedge/a/b/{}, pedge/b/c/{}, nedge/c/d/{}, pedge/a/d/{}, pedge/d/f/{}, pedge/c/e/{}, pedge/e/f/{}}
			\path[\edgetype] (\source) -- node[weight] {$\weight$} (\dest);
		\end{scope}
		\end{tikzpicture}
		\begin{tikzpicture}
		\begin{scope}[xshift=2.75cm, yshift=-0.5cm]	
			\foreach \pos/\name in {{(0,0)/a}, {(0,1)/b}, {(1,1)/c}, {(1,0)/d}, {(2,1)/e}, {(2,0)/f}}
				\node[vertex] (\name) at \pos {}; % {$\name$};
			\node at (1,-0.5) {$YA_3$};
			\foreach \edgetype/\source/ \dest /\weight in {pedge/a/b/{}, pedge/c/d/{}, pedge/a/d/{}, pedge/d/f/{}, pedge/e/f/{}}
			\path[\edgetype] (\source) -- node[weight] {$\weight$} (\dest);
		\end{scope}
		\end{tikzpicture}
		\begin{tikzpicture}
		\begin{scope}[xshift=7.25cm, yshift=-0.5cm]	
			\foreach \pos/\name/\sign/\charge in {{(0.5,1)/a/zc/{}}, {(0,0)/b/zc/{}}, {(1,0)/c/zc/{}}}
				\node[\sign] (\name) at \pos {$\charge$}; % {$\name$};
			\node at (0.5,-0.5) {$YA_4$};
			\foreach \edgetype/\source/ \dest /\weight in {pedge/a/b/{}, pedge/c/b/{}, pedge/a/c/{}}
			\path[\edgetype] (\source) -- node[weight] {$\weight$} (\dest);
		\end{scope}
		\end{tikzpicture}
		\begin{tikzpicture}
		\begin{scope}[xshift=9cm, yshift=-0.5cm]
			\foreach \pos/\name/\sign/\charge in {{(0.5,1)/a/zc/{}}, {(0,0)/b/zc/{}}, {(1,0)/c/zc/{}}}
				\node[\sign] (\name) at \pos {$\charge$}; % {$\name$};
			\node at (0.5,-0.5) {$YA_5$};
			\foreach \edgetype/\source/ \dest /\weight in {pedge/a/b/{}, wedge/b/c/{}, pedge/a/c/{}}
			\path[\edgetype] (\source) -- node[weight] {$\weight$} (\dest);
		\end{scope}
		\end{tikzpicture}
		\begin{tikzpicture}
		\begin{scope}[xshift=5.5cm, yshift=-0.5cm]	
			\foreach \pos/\name in {{(0,0)/a}, {(0,1)/b}, {(1,1)/c}, {(1,0)/d}}
				\node[vertex] (\name) at \pos {}; % {$\name$};
			\node at (0.5,-0.5) {$YA_6$};
			\foreach \edgetype/\source/ \dest /\weight in {pedge/a/b/{}, pedge/b/c/{}, pedge/c/d/{}, wedge/a/d/{}}
			\path[\edgetype] (\source) -- node[weight] {$\weight$} (\dest);
		\end{scope}
		\end{tikzpicture}
		\begin{tikzpicture}
			\begin{scope}
				\newdimen\rad
				\rad=0.6cm

			    % Indicate the boundary of the regular polygons
				\foreach \x in {90,162,234,306,378}
				{
			    	\draw (\x:\rad) node[vertex] {};
			    }
				\foreach \x in {90,162,306,378}
				{
					\draw[pedge] (\x:\rad) -- (\x+72:\rad);
			    }
			\draw[wedge] (234:\rad) -- (306:\rad);
			\end{scope}
			\begin{scope}[xshift=-0.5cm, yshift=-0.5cm]
				\node at (0.5,-0.5) {$YA_7$};
			\end{scope}
		\end{tikzpicture}
		\caption{some cyclotomic $\Z[i]$-graphs that are contained as subgraphs of fixed maximal connected cyclotomic $\Z[i]$-graphs.}
		\label{fig:xgraphsi}
	\end{figure}

	\begin{table}[htbp]
		\begin{center}
		\begin{tabular}{c|c}
			Excluded subgraph & Maximal cyclotomics \\
			\hline
			$YA_1$ & $S_{14}$ and $S_{16}$ \\
			$YA_2$ & $S_{14}$ and $S_{16}$ \\
			$YA_3$ & $S_{14}$ and $S_{16}$ \\
			$YA_4$ & $T_{6}$ and $S_7$ \\
			$YA_5$ & $T^{(i)}_{6}$ and $S^\dagger_8$ \\
			$YA_6$ & $T^{(i)}_{8}$ and $S^{\dag \dag}_8$ \\
			$YA_7$ & $T^{(i)}_{10}$
		\end{tabular}
		\end{center}
		\caption{Excluded subgraphs from Figure~\ref{fig:xgraphsi} and (up to equivalence) their containing maximal connected cyclotomic $\Z[i]$-graphs.}
		\label{tab:exgraphsi}
	\end{table}

	In Table~\ref{tab:exgraphsi} we list each excluded subgraph of type II in Figure~\ref{fig:xgraphsi} along with every maximal connected cyclotomic $\Z[i]$-graph that contains it.
	Let $\mathcal L_1$ consist of vertices of charge $\pm 1$ and the graphs from Figure~\ref{fig:xgraphsi}.
	Hence, all $\mathcal L_1$-free $S$-graphs are uncharged and, since $YA_4$ and $YA_5$ are excluded, no $\mathcal L_1$-free $S$-graph can contain a subgraph whose underlying graph is a triangle.
	We refer to this fact as the `exclusion of triangles'.
	For this section, the notion of a saturated vertex will depend on the list $\mathcal L_1$.
	
		\subsection{Inductive Lemmata} % (fold)
		\label{sub:inductive_lemmata}

		Define $P_{l,r}$ (solid vertices) and $P^\prime_{l,r}$ (solid and hollow vertices) with the following $\Z$-graph
	
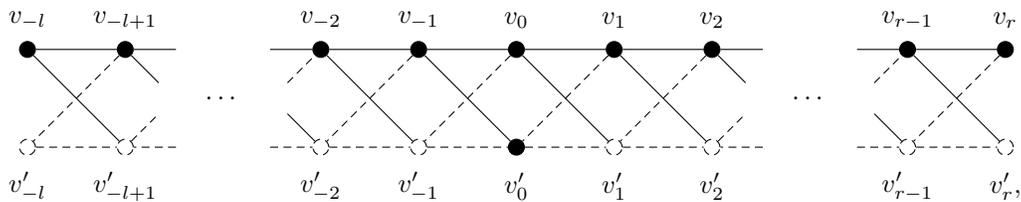
\begin{center}	
		\begin{tikzpicture}
			\begin{scope}[auto, scale=1.3]
				\foreach \type/\pos/\name in {{ghost/(0,0)/a2}, {vertex/(0,1)/a1}, {vertex/(1,1)/b1}, {ghost/(1,0)/b2}, {empty/(1.6,1)/b11}, {empty/(1.6,0)/b21}, {empty/(1.4,0.6)/b12}, {empty/(1.4,0.4)/b22}, {empty/(2.4,1)/c11}, {empty/(2.4,0)/c21}, {empty/(2.6,0.6)/c12}, {empty/(2.6,0.4)/c22}, {vertex/(3,1)/c1}, {ghost/(3,0)/c2}, {vertex/(4,1)/d1}, {ghost/(4,0)/d2}, {vertex/(5,1)/e1}, {vertex/(5,0)/e2}, {vertex/(6,1)/ee1}, {ghost/(6,0)/ee2}, {vertex/(7,1)/f1}, {ghost/(7,0)/f2}, {empty/(7.6,1)/f11}, {empty/(7.6,0)/f21}, {empty/(7.4,0.6)/f12}, {empty/(7.4,0.4)/f22}, {empty/(8.4,1)/g11}, {empty/(8.4,0)/g21}, {empty/(8.6,0.6)/g12}, {empty/(8.6,0.4)/g22}, {vertex/(9,1)/g1}, {ghost/(9,0)/g2}, {vertex/(10,1)/h1}, {ghost/(10,0)/h2}}
					\node[\type] (\name) at \pos {};
				\foreach \pos/\name in {{(8,0.5)/\dots}, {(2,0.5)/\dots}, {(0,1.3)/v_{-l}}, {(0,-0.4)/v_{-l}^\prime}, {(1,1.3)/v_{-l+1}}, {(1,-0.4)/v_{-l+1}^\prime}, {(3,1.3)/v_{-2}}, {(3,-0.4)/v_{-2}^\prime}, {(4,1.3)/v_{-1}}, {(4,-0.4)/v_{-1}^\prime}, {(5,1.3)/v_0}, {(5,-0.4)/v_0^\prime}, {(6,1.3)/v_{1}}, {(6,-0.4)/v_{1}^\prime}, {(7,1.3)/v_2}, {(7,-0.4)/v_2^\prime}, {(9,1.3)/v_{r-1}}, {(9,-0.4)/v_{r-1}^\prime}, {(10,1.3)/v_r}, {(10,-0.4)/v_r^\prime,}}
					\node at \pos {$\name$};
				\foreach \edgetype/\source/ \dest in {nedge/b1/a2, pedge/a1/b1, pedge/a1/b2, nedge/a2/b2, nedge/b21/b2, pedge/b1/b11, pedge/b1/b12, nedge/b2/b22, pedge/c11/c1, nedge/c12/c1, pedge/c22/c2, nedge/c21/c2, nedge/d1/c2, pedge/c1/d1, pedge/c1/d2, nedge/c2/d2, nedge/f21/f2, pedge/f1/f11, pedge/f1/f12, nedge/f2/f22, pedge/g11/g1, nedge/g12/g1, pedge/g22/g2, nedge/g21/g2, nedge/e1/d2, pedge/d1/e1, pedge/d1/e2, nedge/d2/e2, nedge/h1/g2, pedge/g1/h1, pedge/g1/h2, nedge/g2/h2, nedge/ee1/e2, pedge/e1/ee1, pedge/e1/ee2, nedge/e2/ee2, nedge/f1/ee2, pedge/ee1/f1, pedge/ee1/f2, nedge/ee2/f2}
					\path[\edgetype] (\source) -- (\dest);
			\end{scope}
		\end{tikzpicture}	
	\end{center}
		where $l \geqslant 0$ and $r \geqslant 0$.
		Here, the set of hollow vertices of $P_{l,r}$ is the set $V(P^\prime_{l,r}) \backslash V(P_{l,r})$.
		Clearly both $P_{l,r}$ and $P^\prime_{l,r}$ are cyclotomic since they are contained in $T_{2(l+r+2)}$.
		Note that $P_{l,r}$ has $l+r+2$ vertices and $P_{l,r}^\prime$ has $2(l+r+1)$ vertices.
		The set $V^\prime_4(P_{l,r})$ of vertices of $P_{l,r}$ having hollow-degree $4$ is the set $\left \{ v_j : -l < j < r \right \} \cup \left \{ v_0^\prime \right \}$.
		% Here we define $r_0 = l_0 = c$ and $r_0^\prime = l_0^\prime = c^\prime$.
		\begin{lemma}\label{lem:gramvectsind}
			In $P_{l,r}$ for $l \geqslant 2 \text{ or } r \geqslant 2$, we can write the Gram vector for each hollow vertex in terms of Gram vectors of the vertices as follows:
			\begin{align*}
				\mathbf{v}^\prime_{-t} &= \mathbf{v}_{-t} + 2\sum_{j=1}^{t-1} (-1)^{t+j} \mathbf{v}_{-j} + (-1)^t \big ( \mathbf{v}_0 + \mathbf{v}_0^\prime \big ), & \text{ for } t \in \left \{1,\dots, l\right \}. \\
				\mathbf{v}^\prime_t &= -\mathbf{v}_t - 2\sum_{j=1}^{t-1} (-1)^{t+j} \mathbf{v}_j - (-1)^t \big ( \mathbf{v}_0 - \mathbf{v}_0^\prime \big ), & \text{ for } t \in \left \{1,\dots, r\right \}.
			\end{align*}
		\end{lemma}

		\begin{proof}
			By induction using Lemma~\ref{lem:adeqvec}.
		\end{proof}

		\begin{lemma}[Saturated vertices]\label{lem:saturation}
			Let $G$ be an $\mathcal L_1$-free $S$-graph containing $P_{l,r}$ with $l + r > 2$.
			Then, for each vertex $v \in V^\prime_4(P_{l,r})$, we have $N_G(v) = N^\prime_G(v)$.
			Hence, each vertex in $V^\prime_4(P_{l,r})$ is $P_{l,r}^\prime$-saturated in $P_{l,r}$.
		\end{lemma}
		\begin{proof}
			Fix Gram vectors for $P_{l,r}^\prime$.
			We want to show that, for all vertices $v \in V^\prime_4(P_{l,r})$, we have $N_G(v) = N^\prime_G(v)$.
			Since $P_{l,r}^\prime$ contains $P_{l,r}$, we have $N_G(v) \cap V(P_{l,r}) =  N^\prime_G(v)\cap V(P_{l,r})$ for all vertices $v \in V(G)$. 
			Hence, we consider a vertex $v \in V(G)\backslash V(P_{l,r})$ adjacent to some vertex $w \in V^\prime_4(P_{l,r})$ and show that $v$ is switch-equivalent to some hollow vertex, i.e., a vertex in $V(P_{l,r}^\prime)\backslash V(P_{l,r})$.
			
			Split into two cases depending on the hollow-degree of $v_0$.
			
			\paragraph{Case 1} % (fold)
			$v_0$ has hollow-degree $4$.
			In this case both $r$ and $l$ are nonzero.
			We can assume that $r \geqslant 2$ (and $l \geqslant 1$).
			We consider vertices $v_0^\prime, v_j \in V^\prime_4(P_{l,r})$ where $j \geqslant 0$.
			(The arguments are similar for $j \leqslant 0$.)
			Suppose that $v$ is adjacent to $v_j$ for some $j \geqslant 0$.
			Working up to a switching of $v$, we can assume that $\inprod{\mathbf{v},\mathbf{v}_j} = 1$.

			First suppose $j=0$.
			Lemma~\ref{lem:gramvectsind}, provides the following equalities:
			\begin{align}
				\mathbf{v}_{-1}^\prime &= \mathbf{v}_{-1} - \mathbf{v}_0 - \mathbf{v}_0^\prime; \label{eqn:v-1dash} \\
				\mathbf{v}_{1}^\prime &= -\mathbf{v}_{1} + \mathbf{v}_0 - \mathbf{v}_0^\prime. \label{eqn:v1dash}
			\end{align}
			Since we have excluded triangles, $\mathbf{v}$ is orthogonal to both $\mathbf{v}_{-1}$ and $\mathbf{v}_1$.
			We have assumed that $r \geqslant 2$ and $l \geqslant 1$, and since we have excluded $YA_1$, the vertex $v$ must be adjacent to at least one of the vertices $v_0^\prime$ and $v_2$.
			But the exclusion of subgraphs $XA_1$, $YA_2$, and $YA_6$ imply that $v$ must be adjacent to $v_0^\prime$.
			Moreover, $\inprod{\mathbf{v}, \mathbf{v}_0^\prime} = \pm 1$ for otherwise $G$ would contain $YA_6$.  
			In either case, using equations~\eqref{eqn:v-1dash} and \eqref{eqn:v1dash}, we obtain that $\mathbf{v}$ is switch-equivalent to either $\mathbf{v}^\prime_1$ or $\mathbf{v}^\prime_{-1}$.
			Thus $v_0$ is $P_{l,r}^\prime$-saturated in $P_{l,r}$.
			Similarly, $v_0^\prime$ is also $P_{l,r}^\prime$-saturated in $P_{l,r}$.

			Second, suppose $j=1$.
			The exclusion of triangles implies that $\mathbf{v}$ is orthogonal to all of $\mathbf{v}_0$, $\mathbf{v}_0^\prime$, and $\mathbf{v}_2$.
			By Lemma~\ref{lem:adeqvec} we have
			\begin{equation}
				\label{eqn:v1ind}
				2 \mathbf{v}_1 = \mathbf{v}_0 - \mathbf{v}_0^\prime + \mathbf{v}_2 + \mathbf{v}_2^\prime.
			\end{equation}
			Now, by taking the inner product of $\mathbf{v}$ and equation~\eqref{eqn:v1ind} we find that $\mathbf{v} = \mathbf{v}_2^\prime$.
			Hence $v_1$ is $P_{l,r}^\prime$-saturated in $P_{l,r}$.

			If $r = 2$ we are done, so we assume that $r > 2$.
			For our final basic case we suppose that $j=2$.
			Exclusion of triangles implies that $\mathbf{v}$ is orthogonal to both $\mathbf{v}_1$ and $\mathbf{v}_3$.
			If $v$ is adjacent to either $v_0$ or $v_0^\prime$, then since they are $P_{l,r}^\prime$-saturated in $P_{l,r}$, $\mathbf{v}$ must be switch-equivalent to either $\mathbf{v}_1^\prime$ or $\mathbf{v}_{-1}^\prime$, and since $\mathbf{v}_2$ is orthogonal to $\mathbf{v}_{-1}^\prime$, $\mathbf{v}$ must be switch-equivalent to $\mathbf{v}_1^\prime$. 
			Otherwise, if $v$ is adjacent to neither $v_0$ nor $v_0^\prime$ then, from equation~\eqref{eqn:v1dash}, we have $\inprod{\mathbf{v}_{1}^\prime, \mathbf{v}} = 0$.
			By Lemma~\ref{lem:adeqvec}, we have the equality
			\begin{equation}
				\label{eqn:v2adeq}
				2 \mathbf{v}_2 = \mathbf{v}_1 - \mathbf{v}_1^\prime + \mathbf{v}_3 + \mathbf{v}_3^\prime.
			\end{equation}
			From taking the inner product of $\mathbf{v}$ with equation~\eqref{eqn:v2adeq} it follows that $\mathbf{v} = \mathbf{v}_3^\prime$.

			Thus the vertices $v_0$, $v_1$, and $v_2$ are $P_{l,r}^\prime$-saturated in $P_{l,r}$.
			If $r=3$ then we are done.
			Otherwise we assume that $2 < t < r $ and that each vertex $v_j$ with $0 \leqslant j < t$ is $P_{l,r}^\prime$-saturated in $P_{l,r}$.
			It suffices now to show that $v_t$ is $P_{l,r}^\prime$-saturated in $P_{l,r}$.
			Suppose that $v \in V(G)\backslash V(P_{l,r})$ is adjacent to $v_t$.
			We split into cases.
			\subparagraph{Case 1.1.} % (fold)
			$v$ is adjacent to $v_{t-2}$.
			By our inductive hypothesis, $v_{t-2}$ is $P_{l,r}^\prime$-saturated in $P_{l,r}$ and thus $\mathbf{v}$ is switch-equivalent to the Gram vector of some hollow vertex. 
			Moreover, the hollow vertex in question must be adjacent to both $v_t$ and $v_{t-2}$.
			Hence $\mathbf{v}$ is switch-equivalent to $\mathbf{v}_{t-1}^\prime$.		
			% subparagraph subcase_1 (end)

			\subparagraph{Case 1.2.} % (fold)
			$v$ is not adjacent to $v_{t-2}$.
			Hence $\mathbf{v}$ is orthogonal to $\mathbf{v}_{t-2}$.
			The exclusion of triangles implies that $\mathbf{v}$ is also orthogonal to both $\mathbf{v}_{t-1}$ and $\mathbf{v}_{t+1}$.
			Now, our inductive hypothesis says that if $v$ is adjacent to a vertex $v_j \in V^\prime_4(P_{l,r})$ then $v$ is switch-equivalent to some hollow vertex. 
			But for $0 \leqslant k \leqslant t - 3$ there are no hollow vertices adjacent to both $v_k$ and $v_t$.
			Therefore $\mathbf{v}$ must be orthogonal to all of $\mathbf{v}_0, \mathbf{v}_0^\prime, \mathbf{v}_1, \ldots, \mathbf{v}_{t-3}$.
			By Lemma~\ref{lem:gramvectsind}, the vector $\mathbf{v}_{t-1}^\prime$ is a linear combination of the Gram vectors $\mathbf{v}_0, \mathbf{v}_0^\prime, \mathbf{v}_1, \ldots, \mathbf{v}_{t-1}$, and hence $\inprod{\mathbf{v}_{t-1}^\prime, \mathbf{v}} = 0$.
			By Lemma~\ref{lem:adeqvec} we can write
			\begin{equation}
				\label{eqn:vkadeq}
				2 \mathbf{v}_t = \mathbf{v}_{t-1} - \mathbf{v}_{t-1}^\prime + \mathbf{v}_{t+1} + \mathbf{v}_{t+1}^\prime.
			\end{equation}
			The inner product of $\mathbf{v}$ and equation~\eqref{eqn:vkadeq} gives $\inprod{\mathbf{v}, \mathbf{v}_{t+1}^\prime} = 2$.
			Hence $\mathbf{v} = \mathbf{v}_{t+1}^\prime$ as required.
			% subparagraph subcase_2 (end)
			% paragraph case_1 (end)
			\paragraph{Case 2} % (fold)
			$v_0$ does not have hollow-degree $4$.
			Up to equivalence, we can assume that $l = 0$ and $r \geqslant 3$.
			We consider vertices $v_j \in V^\prime_4(P_{l,r})$ where $j \geqslant 1$.
			Suppose that $v$ is adjacent to $v_j$.
			We can assume that $\inprod{\mathbf{v},\mathbf{v}_j} = 1$.
			One can show that $v_1$ is $P_{l,r}^\prime$-saturated in $P_{l,r}$ just as in Case~1.

			Suppose $j = 2$.
			Since triangles are excluded, $v$ is adjacent to neither $v_1$ nor $v_3$.
			If $v$ is adjacent to either $v_0$ or $v_0^\prime$ then the exclusion of $XA_1$, $YA_1$, and $YA_6$ forces $\inprod{\mathbf{v}, \mathbf{v}_0} = -1$ and $\inprod{\mathbf{v}, \mathbf{v}_0^\prime} = 1$.
			And taking the inner product of $\mathbf{v}$ with equation~\eqref{eqn:v1dash} gives $\mathbf{v} = - \mathbf{v}_1^\prime$.
			Otherwise, if $v$ is adjacent to neither $v_0$ nor $v_0^\prime$ then by equation~\eqref{eqn:v1dash}, $\mathbf{v}$ is orthogonal to $\mathbf{v}_1^\prime$.
			Hence, taking the inner product of $\mathbf{v}$ with the equation
			\[
				2 \mathbf{v}_2 = \mathbf{v}_{1} - \mathbf{v}_{1}^\prime + \mathbf{v}_{3} + \mathbf{v}_{3}^\prime
			\]
			gives $\mathbf{v} = \mathbf{v}_{3}^\prime$.
			Therefore, the vertex $v_2$ is $P_{l,r}^\prime$-saturated in $P_{l,r}$.

			If $r=3$ then we are done.
			Otherwise suppose $r \geqslant 4$ and assume that $2 < t < r $ and that each vertex $v_j$ is $P_{l,r}^\prime$-saturated in $P_{l,r}$ where $1 \leqslant j < t$.
			To show that $v_t$ is $P_{l,r}^\prime$-saturated in $P_{l,r}$ we split into cases.
			\subparagraph{Case 2.1.} % (fold)
			$v$ is adjacent to $v_{t-2}$.
			This is the same as in Case~1.1.
			\subparagraph{Case 2.2.} % (fold)
			$v$ is not adjacent to $v_{t-2}$.
			The possibility of $v$ being adjacent to $v_0$ or $v_0^\prime$ is ruled out by the excluded subgraphs $XA_2$ and $YA_7$ if $t = 3$ and by $YA_3$ if $t > 3$.
			Hence $\mathbf{v}$ is orthogonal to both $\mathbf{v}_0$ and $\mathbf{v}_0^\prime$.
			Now the argument is the same as in Case~1.2.
			% paragraph case_2_ (end)
		\end{proof}

		Let $G$ be an $\mathcal L_1$-free $S$-graph containing $P_{l,r}$ with $l+r>2$.
		By the symmetry of the graph $P_{l,r}^\prime$, it follows from Lemma~\ref{lem:saturation} that each vertex in $V^\prime_4(\mathcal V_G(P_{l,r}))$ is $P_{l,r}^\prime$-saturated in $\mathcal V_G(P_{l,r})$.

		\begin{lemma}[Left adjacency]\label{lem:adjacentleft}
			Let $G$ be an $\mathcal L_1$-free $S$-graph containing $P_{l,r}$ with $l+r>2$, where a vertex $v \in V(G)\backslash \mathcal V_G(P_{l,r})$ is adjacent to $v_{-l}$ but not to $v_{r}$.
			Then $\mathbf{v}$ is orthogonal to all of the vectors $\mathbf{v}_j$ and $\mathbf{v}_j^\prime$, for $j \in \left \{1-l,\ldots, r\right \}$.
			Hence $G$ contains a subgraph equivalent to $P_{l+1,r}$.
		\end{lemma}

		\begin{proof}
			By Lemma~\ref{lem:saturation}, the vertices ${v}_{1-l}, \ldots, {v}_{r-1}$ are $P_{l,r}^\prime$-saturated in $P_{l,r}$, and so all of their neighbours are in $\mathcal V_G(P_{l,r})$, hence $\mathbf{v}$ is orthogonal to their Gram vectors.
			And by assumption, we have $\inprod{ \mathbf{v}, \mathbf{v}_r} = 0$.

			First suppose that $l>0$.
			By Lemma~\ref{lem:gramvectsind}, for each $j \in \left \{1-l,\ldots, r \right \}$, we can write $\mathbf{v}_j^\prime$ as a linear combination of the Gram vectors $\mathbf{v}_{1-l}, \ldots, \mathbf{v}_{r}$.
			Therefore $\mathbf{v}$ is orthogonal to each $\mathbf{v}_j^\prime$ as required.
			We can assume that $\inprod{ \mathbf{v}, \mathbf{v}_{-l}} = 1$ and hence $G$ contains a subgraph equivalent to $P_{l+1,r}$.

			Finally suppose that $l=0$ then, by assumption, $r \geqslant 3$.
			We can assume that $\inprod{\mathbf{v}, \mathbf{v}_0} = 1$.
			Now, $v_0^\prime$ is a vertex of $G$ so $\inprod{ \mathbf{v}, \mathbf{v}_0^\prime}$ is in $S$.
			The exclusion of $YA_3$ causes $\inprod{ \mathbf{v}, \mathbf{v}_0^\prime} \ne 0$.
			And the excluded subgraphs $XA_1$ and $YA_6$ force the inner product $\inprod{ \mathbf{v}, \mathbf{v}_0^\prime} = 1$.
			Therefore, we have $\inprod{\mathbf{v}, \mathbf{v}_j} = 0$ for all $j \in \left \{ 1, \ldots, r \right \}$ and $\inprod{\mathbf{v}, \mathbf{v}_0- \mathbf{v}_0^\prime} = 0$.
			Applying Lemma~\ref{lem:gramvectsind} completes the proof.
			\end{proof}

		\begin{lemma}[Right adjacency]\label{lem:adjacentright}
			Let $G$ be an $\mathcal L_1$-free $S$-graph containing $P_{l,r}$ with $l+r>2$, where a vertex $v \in V(G)\backslash \mathcal V_G(P_{l,r})$ is adjacent  to $v_{r}$ but not to $v_{-l}$.
			Then $\mathbf{v}$ is orthogonal to all of the vectors $\mathbf{v}_j$ and $\mathbf{v}_j^\prime$, for $j \in \left \{-l,\ldots, r-1\right \}$.
			Hence $G$ contains a subgraph equivalent to $P_{l,r+1}$.
		\end{lemma}
		\begin{proof}
			Similar to the proof of Lemma~\ref{lem:adjacentleft}.
		\end{proof}

		\begin{lemma}[Left/Right orthogonality]\label{lem:orthogtohollowsboth}
			Let $G$ be an $\mathcal L_1$-free $S$-graph containing $P_{l,r}$ with $l+r>2$, where a vertex $v \in V(G)\backslash \mathcal V_G(P_{l,r})$ is adjacent to $v_{-l}$ and $v_r$.
			Then $\mathbf{v}$ is orthogonal to all of the vectors $\mathbf{v}_j$ and $\mathbf{v}_j^\prime$, for $j \in \left \{1-l,\ldots, r-1\right \}$.
		\end{lemma}
		\begin{proof}
			By Lemma~\ref{lem:saturation}, the vertices $v_j$ are $P_{l,r}^\prime$-saturated in $P_{l,r}$ for all $j \in \left \{1-l,\ldots, r-1\right \}$ and hence all of the neighbours of these vertices are in $\mathcal V_G(P_{l,r})$.
			Therefore $\mathbf{v}$ is orthogonal to $\mathbf{v}_j$ for all $j \in \left \{1-l,\ldots, r-1\right \}$.
			If $l > 0$ and $r >0$, then, in particular, the vertices $v_0$ and $v_0^\prime$ are $P_{l,r}^\prime$-saturated in $P_{l,r}$.
			And Lemma~\ref{lem:gramvectsind} gives that $\mathbf{v}$ is orthogonal to all of the vectors $\mathbf{v}_j^\prime$ for $j \in \left \{1-l,\ldots, r-1\right \}$.

			Suppose $l = 0$.
			We can assume $\inprod{\mathbf{v},\mathbf{v}_0} = 1$.
			We must have that $v$ is adjacent to $v_0^\prime$ otherwise $G$ would contain a subgraph equivalent to $YA_3$.
			Moreover, the exclusion of $XA_1$ and $YA_6$ forces the inner product $\inprod{\mathbf{v}, \mathbf{v}_0^\prime} = 1$.
			Thus, $\inprod{\mathbf{v}, \mathbf{v}_0 - \mathbf{v}_0^\prime} = 0$.
			Applying Lemma~\ref{lem:gramvectsind} gives us that $\mathbf{v}$ is also orthogonal to all of the vectors $\mathbf{v}_j^\prime$, for $j \in \left \{1-l,\ldots, r-1\right \}$.
			The argument is similar when $r=0$.
		\end{proof}

		\begin{lemma}[Left/Right adjacency]\label{lem:adjacentboth}
			Let $G$ be an $\mathcal L_1$-free $S$-graph containing $P_{l,r}$ with $l+r>2$, where a vertex $v \in V(G)\backslash \mathcal V_G(P_{l,r})$ is adjacent to $v_{-l}$ and $v_r$.
			Then $G$ is contained in a graph equivalent to either $T_{2(l+r+2)}$ or $T^{(i)}_{2(l+r+2)}$.
		\end{lemma}
		\begin{proof}
			We can assume that $\inprod{\mathbf{v},\mathbf{v}_{-l}} = 1$ and $\inprod{\mathbf{v},\mathbf{v}_r} = s$ for some nonzero $s \in S$.
			By Lemma~\ref{lem:adeqvec}, we can write
			\begin{equation}
				\label{eqn:forv1l}
				2 \mathbf{v}_{1-l} = \mathbf{v}_{-l} - \mathbf{v}_{-l}^\prime + \mathbf{v}_{2-l} + \mathbf{v}_{2-l}^\prime.
			\end{equation}
			By assumption, $r > 2-l$, and so, according to Lemma~\ref{lem:orthogtohollowsboth}, the Gram vector $\mathbf{v}$ is orthogonal to $\mathbf{v}_{1-l}$, $\mathbf{v}_{2-l}$, and $\mathbf{v}_{2-l}^\prime$.
			Taking the inner product of $\mathbf{v}$ and equation~\eqref{eqn:forv1l} gives $\inprod{\mathbf{v},\mathbf{v}_{-l}^\prime} = \inprod{\mathbf{v},\mathbf{v}_{-l}} = 1$.
			Again, by Lemma~\ref{lem:adeqvec}, we can write
			\begin{equation}
				\label{eqn:forvr1}
				2 \mathbf{v}_{r-1} = \mathbf{v}_{r-2} - \mathbf{v}_{r-2}^\prime + \mathbf{v}_{r} + \mathbf{v}_{r}^\prime.
			\end{equation}
			Similarly, $\mathbf{v}$ is orthogonal to $\mathbf{v}_{r-1}$, $\mathbf{v}_{r-2}$, and $\mathbf{v}^\prime_{r-2}$ and, from the inner product of $\mathbf{v}$ and equation~\eqref{eqn:forvr1}, we obtain $\inprod{\mathbf{v},\mathbf{v}_r^\prime} = -\inprod{\mathbf{v},\mathbf{v}_r} = -s$.
			By Lemma~\ref{lem:adeqvec}, we write
			\begin{equation*}
				2 \mathbf{v} = s \mathbf{v}_{r} - s \mathbf{v}_{r}^\prime + \mathbf{v}_{-l} + \mathbf{v}_{-l}^\prime.
			\end{equation*}
			Define the vector $\mathbf{v}^\prime$ by the equation
			\begin{equation}
				\label{eqn:forvp}
				2 \mathbf{v}^\prime = s \mathbf{v}_{r} - s \mathbf{v}_{r}^\prime - \mathbf{v}_{-l} - \mathbf{v}_{-l}^\prime.
			\end{equation}
			Let $v^\prime$ be a hollow vertex of $P_{l,r} \cup \left \{v\right \}$ with Gram vector $\mathbf{v}^\prime$.
			The graph $P^\prime_{l,r} \cup \left \{v,v^\prime \right \}$ is equivalent to one of the $\Z[i]$-graphs $T_{2k}$ or $T^{(i)}_{2k}$ (for $k = l+r+ 2$), and hence it too is cyclotomic.

			Now we show that every vertex in $V(P_{l,r}) \cup \left \{v\right \}$ is $( P_{l,r}^\prime \cup \left \{v, v^\prime \right \} )$-saturated in $P_{l,r} \cup \left \{ v \right \}$.
			By Lemma~\ref{lem:saturation}, this immediately reduces to showing that $v_{-l}$, $v_r$, and $v$ are $( P_{l,r}^\prime \cup \left \{v, v^\prime \right \} )$-saturated in $P_{l,r} \cup \left \{ v \right \}$.

			First we show that $v_{-l}$ is $( P_{l,r}^\prime \cup \left \{v, v^\prime \right \} )$-saturated in $P_{l,r} \cup \left \{ v \right \}$.
			Suppose that a vertex $x \in V(G) \backslash V(P_{l,r} \cup \left \{v\right \})$ is adjacent to $v_{-l}$.
			We can assume that $\inprod{\mathbf{x}, \mathbf{v}_{-l}} = -1$.
			We must have that $x$ is adjacent to at least one of the vertices $v_r$ and $v_{2-l}$, otherwise $G$ would contain a subgraph equivalent to $YA_3$.
			The exclusion of $XA_1$, $YA_2$, and $YA_6$ forces $x$ to be adjacent to either $v_r$ or $v_{2-l}$.
			If $x$ is adjacent to $v_{2-l}$ then, since $v_{2-l}$ is $( P_{l,r}^\prime \cup \left \{v, v^\prime \right \} )$-saturated in $P_{l,r} \cup \left \{ v \right \}$, the Gram vector $\mathbf{x}$ must be switch-equivalent to $\mathbf{v}_{1-l}^\prime$.
			Otherwise, we assume $x$ is adjacent to $v_r$.
			If $\inprod{\mathbf{x}, \mathbf{v}_{r}} = -s$ then $G$ would contain a subgraph equivalent $XA_1$ and if $\inprod{\mathbf{x}, \mathbf{v}_{r}} = \pm is$ then $G$ would contain a subgraph equivalent to $YA_6$.
			We have, therefore, that $\inprod{\mathbf{x}, \mathbf{v}_{r}} = s$.
			Apply Lemma~\ref{lem:orthogtohollowsboth}, to give that, in particular, $\mathbf{x}$ is orthogonal to the vectors $\mathbf{v}_{1-l}$, $\mathbf{v}_{2-l}$, $\mathbf{v}_{2-l}^\prime$, $\mathbf{v}_{r-1}$, $\mathbf{v}_{r-2}$, and $\mathbf{v}_{r-2}^\prime$.
			The inner product of $\mathbf{x}$ with equation~\eqref{eqn:forvr1} and the inner product of $\mathbf{x}$ with equation~\eqref{eqn:forv1l} yield $\inprod{\mathbf{x}, \mathbf{v}_r^\prime} = -s$ and $\inprod{\mathbf{x}, \mathbf{v}_{-l}^\prime} = -1$ respectively.
			Now, taking the inner product of $\mathbf{x}$ with equation~\eqref{eqn:forvp} gives $\mathbf{x} = \mathbf{v}^\prime$.
			Hence, the vertex $v_{-l}$ is $( P_{l,r}^\prime \cup \left \{v, v^\prime \right \} )$-saturated in $P_{l,r} \cup \left \{ v \right \}$.
			Similar arguments show that $v_r$ is also $( P_{l,r}^\prime \cup \left \{v, v^\prime \right \} )$-saturated in $P_{l,r} \cup \left \{ v \right \}$.

			It remains to show that $v$ is $( P_{l,r}^\prime \cup \left \{v, v^\prime \right \} )$-saturated in $P_{l,r} \cup \left \{ v \right \}$.
			Suppose that a vertex $x \in V(G) \backslash ( V(P_{l,r})\cup \left \{v\right \} )$ is adjacent to $v$.
			Since triangles have been excluded, $x$ is adjacent to neither $v_{-l}$ nor $v_r$.
			In fact, we must have that $x$ is adjacent to either $v_{1-l}$ or $v_{r-1}$ otherwise $G$ would contain a subgraph equivalent to $YA_3$.
			Both $v_{1-l}$ and $v_{r-1}$ are $( P_{l,r}^\prime \cup \left \{v, v^\prime \right \} )$-saturated in $P_{l,r} \cup \left \{ v \right \}$, so we are done.

			Thus we have that each vertex in $\mathcal V_G(P_{l,r} \cup \left \{ v \right \}))$ is $( P_{l,r}^\prime \cup \left \{v, v^\prime \right \} )$-saturated in $P_{l,r} \cup \left \{ v \right \}$.
			Since $G$ is an $\mathcal L_1$-free $S$-graph containing $\mathcal V_G(P_{l,r} \cup \left \{ v \right \}))$, each vertex of $G$ corresponds to a vertex of $P_{l,r}^\prime \cup \left \{v, v^\prime \right \}$.
			This correspondence is one to one, since otherwise, if two vertices $x$ and $y$ of $G$ were both switch-equivalent to the same vertex $z$, then $\abs{w(x,y)} = 2$, which is not in $S$.
			Depending on the value of $s$, the $S$-graph $P_{l,r}^\prime \cup \left \{v, v^\prime \right \}$ is equivalent to either $T_{2(l+r+2)}$ or $T^{(i)}_{2(l+r+2)}$.
			Hence $G$ is contained in a graph equivalent to either $T_{2(l+r+2)}$ or $T^{(i)}_{2(l+r+2)}$.
		\end{proof}

		% subsection inductive_lemmata2 (end)
		
		\subsection{$\mathcal L_1$-free $S$-graphs on up to 9 vertices}
		\label{sec:unchargedS9}

		Consider the infinite family of $n$-vertex $(n \geqslant 3)$ $S$-cycles $O^{(s)}_n$ illustrated below.
		
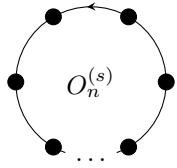
\begin{center}
			\begin{tikzpicture}[scale=1, auto]
				\newdimen\rad
				\rad=1cm
				\newdimen\radi
				\radi=1.04cm
				% \def\shift{0}
			    % Indicate the boundary of the regular polygons
				\draw (271.5:\radi) node[empty] {$\dots$};
				\foreach \y in {60,120,180,240,300}
				{
					\def\x{\y - 120}
					\draw (\x:\rad) node[vertex] {};
					% \draw[pedge] (\x:\radi) -- (\x+225:\radi);
			    }
				\foreach \y in {60,120,240,300}
				{
					\def\x{\y - 120}
					% \draw[pedge] (\x:\radi) -- (\x+225:\radi);
					\draw[pedge] (\x:\rad) arc (\x:\x+60:\rad);
			    }
				\draw[wedge] (60:\rad) arc (60:120:\rad);
				\draw (240:\rad) node[vertex] {};
				\draw[pedge] (240:\rad) arc (240:250:\rad);
				\draw[pedge] (290:\rad) arc (290:300:\rad);
				\node at (0,0) {$O_n^{(s)}$};
			\end{tikzpicture}
		\end{center}
		The edge of $O_n^{(s)}$ marked with an arrow corresponds to the edge of weight $s$.
		The $S$-cycles $O^{(s)}_n$ can be defined on vertices $v_1,\dots,v_n$ by setting $w(v_1, v_n) = s$ for some $s \in S$ and $w(v_j,v_{j+1}) = 1$ for $j \in \left \{1,\ldots,n-1\right \}$.

		\begin{lemma}\label{lem:CnSmallSpan}
			The $S$-graph $O^{(s)}_n$ is cyclotomic for all $n \geqslant 3$.
		\end{lemma}
		\begin{proof}
			Since $O^{(s)}_n$ is contained in either $T_{2k}$ or $T_{2k}^{(i)}$, which are both cyclotomic, the lemma follows by Theorem~\ref{thm:interlacing}.
		\end{proof}

		\begin{lemma}\label{lem:cycleequiv}
			Let $G$ be an uncharged $S$-cycle.
			Then $G$ is strongly equivalent to $O_n^{(s)}$ for some $s \in \left \{\pm 1, \pm i \right \}$ and some $n \in \N$.
		\end{lemma}
		\begin{proof}
			Suppose $G$ is an $S$-cycle on $n$ vertices.
			Label the vertices $v_1, \ldots, v_n$ so that $v_1$ is adjacent to $v_n$ and $v_j$ is adjacent to $v_{j+1}$ for all $j \in \left \{1,\ldots,n-1\right \}$.
			We can inductively switch the vertices of $G$ so that $w(v_j,v_{j+1}) = 1$ for all $j \in \left \{1,\ldots,n-1\right \}$, and $w(v_1,v_n) = s$ for some $s \in \left \{\pm 1, \pm i \right \}$.
		\end{proof}

		Let $G$ be an $\mathcal L_1$-free $S$-graph.
		If the maximum degree of $G$ is $1$ then $G$ is just an edge.
		If the maximum degree of $G$ is $2$, then $G$ is either an $S$-cycle or an $S$-path.
		If $G$ is an $S$-path then by inductively switching the vertices, we obtain an equivalent $\left \{0,1\right \}$-path which is contained in the visibly cyclotomic $\Z$-graph $T_{2k}$ for some $k$.
		If $G$ is an $S$-cycle then, by Lemma~\ref{lem:cycleequiv}, $G$ is equivalent to the $S$-cycle $O^{(s)}_n$ in Lemma~\ref{lem:CnSmallSpan} for some $s \in \left \{\pm 1, \pm i \right \}$.
		The problem, therefore, reduces to assuming that the maximum degree of $G$ is at least $3$.

		Below we describe the process of computing $\mathcal L_1$-free $S$-graphs on a given number of vertices.

		\paragraph{Growing process} % (fold)
		\label{par:growing_process}
			Start with a single vertex $H$.
			Consider all possible ways of adding a vertex to $H$ such that the resulting graph $H^\prime$ is $\mathcal L_1$-free.
			Repeat this process with all supergraphs $H^\prime$ until all $\mathcal L_1$-free $S$-graphs on the desired number of vertices have been obtained.
		% paragraph growing_process (end)

		We have exhaustively computed (up to equivalence) all $\mathcal L_1$-free $S$-graphs on up to $9$ vertices having maximal degree at least $3$.
		Out of these graphs, the ones on $9$ vertices contained a subgraph equivalent to either $P_{0,3}$ or $P_{1,2}$.
		% On $8$ vertices, in particular, we obtain $T_8$ and $T_8^{(i)}$ which are maximal since none of the degree-$9$ uncharged cyclotomic $S$-graphs contained either of them.
		It should be noted that this computation can be done by hand.
		One considers all $\mathcal L_1$-free $S$-supergraphs of the complete bipartite graph $K_{1,3}$ that do not contain a graph equivalent to $P_{l,r}$, with $l + r > 2$, to find that there do not exist any such graphs on more than $8$ vertices.
		For the sake of succinctness we  have omitted the details.

		Now, from the above computation and by iteratively applying Lemmata~\ref{lem:adjacentleft}, \ref{lem:adjacentright}, and \ref{lem:adjacentboth}, we have the following lemma.

		\begin{lemma}\label{lem:unL1free}
			Let $G$ be an $\mathcal L_1$-free $S$-graph.
			Then $G$ is contained in either $T_{2k}$ or $T_{2k}^{(i)}$ for $k \geqslant 3$.
		\end{lemma}

		Together with the computation of the maximal connected cyclotomic $\Z[i]$-graphs containing the excluded subgraphs of type II from the list $\mathcal L_1$ (see Figure~\ref{fig:xgraphsi}), we have proved Theorem~\ref{thm:classununzi}.

	\section{Proof of Theorem~\ref{thm:classunzi}}
	\label{sec:proof2}

	In this section we prove Theorem~\ref{thm:classunzi}.
	Let $G$ be an uncharged cyclotomic $\Z[i]$-graph.
	By Lemma~\ref{lem:maxDeg4}, we know that $G$ cannot be equivalent to a graph containing any weight-$\alpha$ edge where the norm of $\alpha$ is greater than $4$.
	Therefore $G$ can have edge-weights coming only from the subset	$\left \{0,\pm 1, \pm i, \pm 1 \pm i, \pm 2, \pm 2i \right \}$.

% \newpage

	\subsection{Excluded subgraphs}

	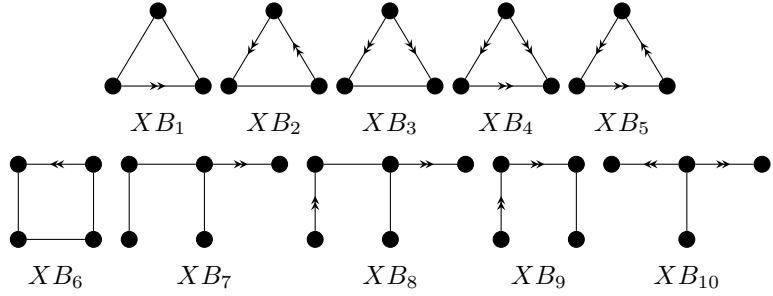
\begin{figure}[htbp]
		\centering
		\begin{tikzpicture}[scale=1, auto]
		\begin{scope}[yshift=-0.5cm]	
			\foreach \pos/\name in {{(0,0)/a}, {(1.2,0)/b}, {(0.6,1)/c}}
				\node[vertex] (\name) at \pos {}; % {$\name$};
			\node at (0.6,-0.5) {$XB_{1}$};
			\foreach \edgetype/\source/ \dest /\weight in {wwedge/a/b/{}, pedge/b/c/{}, pedge/c/a/{}}
			\path[\edgetype] (\source) -- node[weight] {$\weight$} (\dest);
		\end{scope}
		\end{tikzpicture}
		\begin{tikzpicture}[scale=1, auto]
		\begin{scope}[yshift=-0.5cm]	
			\foreach \pos/\name in {{(0,0)/a}, {(1.2,0)/b}, {(0.6,1)/c}}
				\node[vertex] (\name) at \pos {}; % {$\name$};
			\node at (0.6,-0.5) {$XB_2$};
			\foreach \edgetype/\source/ \dest /\weight in {pedge/a/b/{}, wwedge/b/c/{}, wwedge/c/a/{}}
			\path[\edgetype] (\source) -- node[weight] {$\weight$} (\dest);
		\end{scope}
		\end{tikzpicture}
		\begin{tikzpicture}[scale=1, auto]
		\begin{scope}[yshift=-0.5cm]	
			\foreach \pos/\name in {{(0,0)/a}, {(1.2,0)/b}, {(0.6,1)/c}}
				\node[vertex] (\name) at \pos {}; % {$\name$};
			\node at (0.6,-0.5) {$XB_{3}$};
			\foreach \edgetype/\source/ \dest /\weight in {pedge/a/b/{}, wwedge/c/b/{}, wwedge/c/a/{}}
			\path[\edgetype] (\source) -- node[weight] {$\weight$} (\dest);
		\end{scope}
		\end{tikzpicture}
		\begin{tikzpicture}[scale=1, auto]
		\begin{scope}[yshift=-0.5cm]	
			\foreach \pos/\name in {{(0,0)/a}, {(1.2,0)/b}, {(0.6,1)/c}}
				\node[vertex] (\name) at \pos {}; % {$\name$};
			\node at (0.6,-0.5) {$XB_{4}$};
			\foreach \edgetype/\source/ \dest /\weight in {wwedge/a/b/{}, wwedge/c/b/{}, wwedge/c/a/{}}
			\path[\edgetype] (\source) -- node[weight] {$\weight$} (\dest);
		\end{scope}
		\end{tikzpicture}
		\begin{tikzpicture}[scale=1, auto]
		\begin{scope}[yshift=-0.5cm]	
			\foreach \pos/\name in {{(0,0)/a}, {(1.2,0)/b}, {(0.6,1)/c}}
				\node[vertex] (\name) at \pos {}; % {$\name$};
			\node at (0.6,-0.5) {$XB_{5}$};
			\foreach \edgetype/\source/ \dest /\weight in {wwedge/a/b/{}, wwedge/b/c/{}, wwedge/c/a/{}}
			\path[\edgetype] (\source) -- node[weight] {$\weight$} (\dest);
		\end{scope}
		\end{tikzpicture}

		\begin{tikzpicture}
		\begin{scope}[xshift=12cm, yshift=-0.5cm]	
			\foreach \pos/\name in {{(0,0)/a}, {(0,1)/b}, {(1,1)/c}, {(1,0)/d}}
				\node[vertex] (\name) at \pos {}; % {$\name$};
			\node at (0.5,-0.5) {$XB_6$};
			\foreach \edgetype/\source/ \dest /\weight in {pedge/a/b/{}, wwedge/c/b/{}, pedge/c/d/{}, pedge/a/d/{}}
			\path[\edgetype] (\source) -- node[weight] {$\weight$} (\dest);
		\end{scope}
		\end{tikzpicture}
		\begin{tikzpicture}[scale=1, auto]
		\begin{scope}[yshift=-0.5cm]	
			\foreach \pos/\name in {{(0,0)/a}, {(0,1)/b}, {(1,1)/c}, {(1,0)/d}, {(2,1)/e}}
				\node[vertex] (\name) at \pos {}; % {$\name$};
			\node at (1,-0.5) {$XB_7$};
			\foreach \edgetype/\source/ \dest /\weight in {pedge/a/b/{}, pedge/c/b/{}, pedge/c/d/{}, wwedge/c/e/{}}
			\path[\edgetype] (\source) -- node[weight] {$\weight$} (\dest);
		\end{scope}
		\end{tikzpicture}
		\begin{tikzpicture}[scale=1, auto]
		\begin{scope}[yshift=-0.5cm]	
			\foreach \pos/\name in {{(0,0)/a}, {(0,1)/b}, {(1,1)/c}, {(1,0)/d}, {(2,1)/e}}
				\node[vertex] (\name) at \pos {}; % {$\name$};
			\node at (1,-0.5) {$XB_8$};
			\foreach \edgetype/\source/ \dest /\weight in {wwedge/a/b/{}, pedge/c/b/{}, pedge/c/d/{}, wwedge/c/e/{}}
			\path[\edgetype] (\source) -- node[weight] {$\weight$} (\dest);
		\end{scope}
		\end{tikzpicture}
		\begin{tikzpicture}[scale=1, auto]
		\begin{scope}[yshift=-0.5cm]	
			\foreach \pos/\name in {{(0,0)/a}, {(0,1)/b}, {(1,1)/c}, {(1,0)/d}}
				\node[vertex] (\name) at \pos {}; % {$\name$};
			\node at (0.5,-0.5) {$XB_9$};
			\foreach \edgetype/\source/ \dest /\weight in {wwedge/a/b/{}, wwedge/b/c/{}, pedge/c/d/{}}
			\path[\edgetype] (\source) -- node[weight] {$\weight$} (\dest);
		\end{scope}
		\end{tikzpicture}
		\begin{tikzpicture}[scale=1, auto]
		\begin{scope}[yshift=-0.5cm]	
			\foreach \pos/\name in {{(0,1)/b}, {(1,1)/c}, {(1,0)/d}, {(2,1)/e}}
				\node[vertex] (\name) at \pos {}; % {$\name$};
			\node at (1,-0.5) {$XB_{10}$};
			\foreach \edgetype/\source/ \dest /\weight in {wwedge/c/b/{}, pedge/c/d/{}, wwedge/c/e/{}}
			\path[\edgetype] (\source) -- node[weight] {$\weight$} (\dest);
		\end{scope}
		\end{tikzpicture}
		\caption{some non-cyclotomic uncharged $\Z[i]$-graphs.}
		\label{fig:xgraphs1}
	\end{figure}

	\begin{figure}[htbp]
		\centering
		\begin{tikzpicture}[scale=1, auto]
		\begin{scope}[yshift=-0.5cm]	
			\foreach \pos/\name in {{(0,0)/a}, {(0,1)/b}, {(1,1)/c}, {(1,0)/d}}
				\node[vertex] (\name) at \pos {}; % {$\name$};
			\node at (0.5,-0.5) {$YB_1$};
			\foreach \edgetype/\source/ \dest /\weight in {pedge/a/b/{}, wwedge/b/c/{}, pedge/c/d/{}}
			\path[\edgetype] (\source) -- node[weight] {$\weight$} (\dest);
		\end{scope}
		\end{tikzpicture}
		\begin{tikzpicture}
		\begin{scope}[xshift=2cm, yshift=-0.5cm]	
			\foreach \pos/\name in {{(0,0)/a}, {(0,1)/b}, {(1,1)/c}, {(1,0)/d}}
				\node[vertex] (\name) at \pos {}; % {$\name$};
			\node at (0.5,-0.5) {$YB_2$};
			\foreach \edgetype/\source/ \dest /\weight in {pedge/a/b/{}, wwedge/b/c/{}, pedge/c/d/{}, wwnedge/a/d/{}}
			\path[\edgetype] (\source) -- node[weight] {$\weight$} (\dest);
		\end{scope}
		\end{tikzpicture}
		\begin{tikzpicture}
		\begin{scope}[xshift=12cm, yshift=-0.5cm]	
			\foreach \pos/\name in {{(0,0)/a}, {(0,1)/b}, {(1,1)/c}, {(1,0)/d}}
				\node[vertex] (\name) at \pos {}; % {$\name$};
			\node at (0.5,-0.5) {$YB_3$};
			\foreach \edgetype/\source/ \dest /\weight in {pedge/a/b/{}, wwedge/c/b/{}, pedge/c/d/{}, nedge/a/d/{}}
			\path[\edgetype] (\source) -- node[weight] {$\weight$} (\dest);
		\end{scope}
		\end{tikzpicture}
		\begin{tikzpicture}
		\begin{scope}	
			\foreach \pos/\name/\type/\charge in {{(0,0)/a/zc/{}}, {(1,0)/b/zc/{}}}
				\node[\type] (\name) at \pos {$\charge$}; % {$\name$};
			\foreach \edgetype/\source/ \dest / \weight in {pedge/a/b/2}
			\path[\edgetype] (\source) -- node[weight2] {$\weight$} (\dest);
			\node at (0.5,-1) {$YB_4$};
		\end{scope}
		\end{tikzpicture}
		\caption{some cyclotomic $\Z[i]$-graphs that are contained as subgraphs of fixed maximal connected cyclotomic $\Z[i]$-graphs.}
		\label{fig:xgraphs2}
	\end{figure}

	\begin{table}[htbp]
		\begin{center}
		\begin{tabular}{c|c}
			Excluded subgraph & Maximal cyclotomics \\
			\hline
			$YB_1$ & $S_8^\ddag$ \\
			$YB_2$ & $S_8^\ddag$ \\
			$YB_3$ & $S_{8}^{\dag \dag}$ \\ 
			$YB_4$ & $S_2$ \\
		\end{tabular}
		\end{center}
		\caption{Excluded subgraphs from Figure~\ref{fig:xgraphs2} and (up to equivalence) their containing maximal connected cyclotomic $\Z[i]$-graphs.}
		\label{tab:exgraphsiweight}
	\end{table}

	In Table~\ref{tab:exgraphsiweight} we list each excluded subgraph of type II in Figure~\ref{fig:xgraphs2} along with every maximal connected cyclotomic $\Z[i]$-graph that contains it.
	Let $\mathcal L_2$ consist of all cyclotomic charged vertices and the graphs in Figures~\ref{fig:xgraphsi} and \ref{fig:xgraphs2}.
	Hence, all $\mathcal L_2$-free $\Z[i]$-graphs are uncharged and, since we have excluded $XB_1, XB_2, XB_3, XB_4, XB_5$ together with $YA_4$ and $YA_5$, we have that no $\mathcal L_2$-free $\Z[i]$-graph can contain a subgraph whose underlying subgraph is a triangle.
	As in Section~\ref{sec:growZi}, we may refer to this fact as the `exclusion of triangles'.
	For this section, the notion of a saturated vertex will depend on the list $\mathcal L_2$.
	
\subsection{Inductive lemmata} % (fold)
\label{sub:inductive_lemmata2}

Define $P_{2r+1}$ (solid vertices) and $P^\prime_{2r+1}$ (solid vertices and hollow vertices) with the following $\Z[i]$-graph
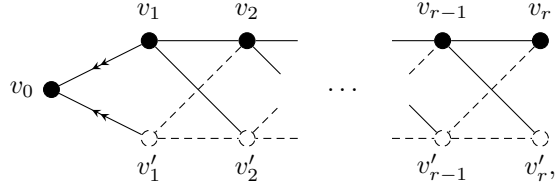
\[
\begin{tikzpicture}
	\begin{scope}[auto, scale=1.3]
		\foreach \type/\pos/\name in {{vertex/(-1,0.5)/bgn}, {ghost/(0,0)/a2}, {vertex/(0,1)/a1}, {vertex/(1,1)/b1}, {ghost/(1,0)/b2}, {empty/(1.6,1)/b11}, {empty/(1.6,0)/b21}, {empty/(1.4,0.6)/b12}, {empty/(1.4,0.4)/b22}, {empty/(2.4,1)/c11}, {empty/(2.4,0)/c21}, {empty/(2.6,0.6)/c12}, {empty/(2.6,0.4)/c22}, {vertex/(3,1)/c1}, {ghost/(3,0)/c2}, {vertex/(4,1)/d1}, {ghost/(4,0)/d2}}
			\node[\type] (\name) at \pos {};
		\foreach \pos/\name in {{(2,0.5)/\dots},{(-1.3,0.5)/v_0}, {(0,1.3)/v_1}, {(0,-0.3)/v_1^\prime}, {(1,1.3)/v_2}, {(1,-0.3)/v_2^\prime}, {(3,1.3)/v_{r-1}}, {(3,-0.3)/v_{r-1}^\prime}, {(4,1.3)/v_r}, {(4,-0.3)/v_r^\prime,}}
			\node at \pos {$\name$};
		\foreach \edgetype/\source/ \dest in {wwedge/a1/bgn, wwedge/a2/bgn, nedge/b1/a2, pedge/a1/b1, pedge/a1/b2, nedge/a2/b2, nedge/b21/b2, pedge/b1/b11, pedge/b1/b12, nedge/b2/b22, pedge/c11/c1, nedge/c12/c1, pedge/c22/c2, nedge/c21/c2, nedge/d1/c2, pedge/c1/d1, pedge/c1/d2, nedge/c2/d2}
			\path[\edgetype] (\source) -- (\dest);
	\end{scope}
\end{tikzpicture}	
\]
where $r \geqslant 1$.
The set of hollow vertices of $P_{2r+1}$ is the set $V(P^\prime_{2r+1}) \backslash V(P_{2r+1})$.
Clearly both $P_{2r+1}$ and $P^\prime_{2r+1}$ are cyclotomic since they are contained in $C_{2(r+1)}$.
Note that $P_{2r+1}$ has $r+1$ vertices and $P_{2r+1}^\prime$ has $2r+1$ vertices.
Having chosen Gram vectors $\mathbf{v}_0, \ldots, \mathbf{v}_r$, by an argument similar to the proof of Lemma~\ref{lem:adeqvec}, we can write
\begin{equation}
	\label{eqn:forv0weight}
	\mathbf{v}_1^\prime = -\mathbf{v}_1 + (1 + i) \mathbf{v}_0
\end{equation}
and
\begin{equation}
	\label{eqn:forv1weight}
	\mathbf{v}_2^\prime = -\mathbf{v}_2 + 2 \mathbf{v}_1 -(1 + i)\mathbf{v}_0.
\end{equation}
\begin{lemma}\label{lem:gramvectsind2}
	In $P_{2r+1}$ for $r \geqslant 2$, we can write the Gram vector for each hollow vertex in terms of Gram vectors of the vertices as follows:
	\begin{equation*}
		\mathbf{v}^\prime_{t} = -\mathbf{v}_{t} - 2\sum_{j=1}^{t-1} (-1)^{t+j} \mathbf{v}_{j} - (-1)^t (1+i) \mathbf{v}_0, \quad \text{ for } t \in \left \{1,\dots, r\right \}.
	\end{equation*}
\end{lemma}
\begin{proof}
	By induction using equations \eqref{eqn:forv0weight}, \eqref{eqn:forv1weight}, and Lemma~\ref{lem:maxDeg4}.
\end{proof}

\begin{lemma}[Saturated vertices]\label{lem:saturationweighted}
	Let $G$ be an $\mathcal L_2$-free $\Z[i]$-graph containing $P_{2r+1}$ with $r \geqslant 3$.
	Then, for each vertex $v \in V^\prime_4(P_{2r+1})$, we have $N_G(v) = N^\prime_G(v)$.
	Hence, each vertex in $V^\prime_4(P_{2r+1})$ is $P_{2r+1}^\prime$-saturated in $P_{2r+1}$.
\end{lemma}
\begin{proof}
	Fix Gram vectors for $P_{2r+1}^\prime$.
	Consider a vertex $v \in V(G)\backslash V(P_{2r+1})$.
	Suppose that $v$ is adjacent to the vertex $v_j \in V^\prime_4(P_{2r+1})$ for some $j \in \left \{0, \ldots, r-1 \right \}$.
	Without loss of generality, assume either $\inprod{\mathbf{v}, \mathbf{v}_j} = 1$ or $\inprod{\mathbf{v}, \mathbf{v}_j} = 1 + i$.
	
	Suppose first that $j = 0$.
	The exclusion of triangles implies that $\inprod{\mathbf{v}, \mathbf{v}_1} = 0$.
	Since $XB_9$ and $YB_1$ are excluded, $v$ must be adjacent to $v_2$, moreover, the excluded subgraphs $XB_6$ and $YB_3$ preclude the possibility of $\inprod{\mathbf{v}, \mathbf{v}_0} = 1$ while $\inprod{\mathbf{v}, \mathbf{v}_2} = \pm1$.
	And since $r \geqslant 3$, if $\inprod{\mathbf{v}, \mathbf{v}_0} = 1$ while $\inprod{\mathbf{v}, \mathbf{v}_2} = \pm1 \pm i$ then $G$ would contain a subgraph equivalent to $YB_1$.
	Therefore we must have $\inprod{\mathbf{v}, \mathbf{v}_0} = 1+i$.
	By taking the inner product of $\mathbf{v}$ and equation~\eqref{eqn:forv0weight}, we obtain that $\mathbf{v} = \mathbf{v}_1^\prime$.
	
	Second, suppose that $j = 1$.
	Since we have excluded triangles, $\mathbf{v}$ must be orthogonal to both $\mathbf{v}_0$ and $\mathbf{v}_2$ and we must have $\inprod{\mathbf{v}, \mathbf{v}_1} = 1$, otherwise the degree of $v_1$ is greater than $4$.
	Using equation~\eqref{eqn:forv1weight}, we find that $\inprod{\mathbf{v} - \mathbf{v}_2^\prime, \mathbf{v} - \mathbf{v}_2^\prime} = 0$.
	Hence $\mathbf{v} = \mathbf{v}_2^\prime$.
	
	We have that the vertices $v_0$ and $v_1$ are $P_{2r+1}^\prime$-saturated in $P_{2r+1}$.
	We assume that, for $1 < t < r$, each vertex $v_j \in V^\prime_4(P_{2r+1})$ with $0 \leqslant j < t$ is $P_{2r+1}^\prime$-saturated in $P_{2r+1}$.
	It suffices now to show that $v_t$ is $P_{2r+1}^\prime$-saturated in $P_{2r+1}$.
	Suppose a vertex $v \in V(G)\backslash V(P_{2r+1})$ is adjacent to $v_t$.
	We split into cases.
	\paragraph{Case 1} % (fold)
	$v$ is adjacent to $v_{t-2}$.
	By our inductive hypothesis, $v_{t-2}$ is $P_{2r+1}^\prime$-saturated in $P_{2r+1}$ and thus $\mathbf{v}$ is switch-equivalent to the Gram vector of some hollow vertex. 
	Moreover, the hollow vertex in question must be adjacent to both $v_t$ and $v_{t-2}$.
	Hence $\mathbf{v}$ is switch-equivalent to $\mathbf{v}_{t-1}^\prime$.		
	% subparagraph subcase_1 (end)
	
	\paragraph{Case 2} % (fold)
	$v$ is not adjacent to $v_{t-2}$.
	Then $\inprod{\mathbf{v}, \mathbf{v}_{t-2}} = 0$.
	Since triangles are excluded, $\mathbf{v}$ is orthogonal to $\mathbf{v}_{t-1}$ and $\mathbf{v}_{t+1}$.
	And we must have $\inprod{\mathbf{v}, \mathbf{v}_t} = 1$ since we have excluded $XB_7$ and $XB_8$.
	Now, our inductive hypothesis says that if $v$ is adjacent to a vertex $v_k \in V^\prime_4(P_{2r+1})$ where $0 \leqslant k < t$ then $\mathbf{v}$ is switch-equivalent to the Gram vector of some hollow vertex. 
	But for $0 \leqslant k \leqslant t - 3$ there is no hollow vertex adjacent to both $v_k$ and $v_t$.
	Therefore $\mathbf{v}$ must be orthogonal to all of $\mathbf{v}_0, \mathbf{v}_1, \ldots, \mathbf{v}_{t-3}$.
	By Lemma~\ref{lem:gramvectsind}, the vector $\mathbf{v}_{t-1}^\prime$ is a linear combination of the Gram vectors $\mathbf{v}_0, \mathbf{v}_1, \ldots, \mathbf{v}_{t-1}$, and hence $\inprod{\mathbf{v}_{t-1}^\prime, \mathbf{v}} = 0$.
	By Lemma~\ref{lem:adeqvec} we can write
	\begin{equation}
		\label{eqn:vkadeqweight}
		2 \mathbf{v}_t = \mathbf{v}_{t-1} - \mathbf{v}_{t-1}^\prime + \mathbf{v}_{t+1} + \mathbf{v}_{t+1}^\prime.
	\end{equation}
	From the inner product of $\mathbf{v}$ and equation~\eqref{eqn:vkadeqweight}, it follows that $\mathbf{v} = \mathbf{v}_{t+1}^\prime$ as required.
\end{proof}

	Let $G$ be an $\mathcal L_2$-free $\Z[i]$-graph containing $P_{2r+1}$ with $r \geqslant 3$.
	By the symmetry of $P_{2r+1}^\prime$, it follows from Lemma~\ref{lem:saturationweighted} that each vertex in $V_4^\prime(\mathcal V_G(P_{2r+1}))$ is $P_{2r+1}^\prime$-saturated in $\mathcal V_G(P_{2r+1})$.

\begin{lemma}\label{lem:unchargedinduct}
	Let $G$ be an $\mathcal L_2$-free $\Z[i]$-graph containing $P_{2r+1}$ with $r \geqslant 3$, where $v_r$ is adjacent to a vertex $v \in V(G) \backslash \mathcal V_G(P_{2r+1})$.
	Then either $G$ is contained in $C_{2(r+1)}$ or $G$ contains $P_{2(r+1)+1}$.
\end{lemma}
\begin{proof}
	Without loss of generality, we have either $\inprod{\mathbf{v}, \mathbf{v}_r} = 1$ or $\inprod{\mathbf{v}, \mathbf{v}_r} = 1+i$.
	By Lemma~\ref{lem:gramvectsind2}, for $j \in \left \{1,\ldots, r-1 \right \}$, we can write $\mathbf{v}_j^\prime$ as a linear combination of the Gram vectors $\mathbf{v}_0, \ldots, \mathbf{v}_j$.
	According to Lemma~\ref{lem:saturationweighted}, the vertices $v_0$, \ldots, $v_{r-1}$ are $P_{2r+1}^\prime$-saturated in $P_{2r+1}$. 
	Since $v \not \in \mathcal V_G(P_{2r+1})$, we have $\inprod{\mathbf{v}, \mathbf{v}_j} = 0$ for $j \in \left \{0,\ldots, r-1 \right \}$.
	Therefore, $\mathbf{v}$ is orthogonal to $\mathbf{v}_j^\prime$ for all $j \in \left \{ 1, \ldots, r-1 \right \}$.
	Hence, in particular, $\mathbf{v}$ is orthogonal to $\mathbf{v}_{r-1}$, $\mathbf{v}_{r-2}$, and $\mathbf{v}_{r-2}^\prime$.
	By Lemma~\ref{lem:adeqvec} we have
	\begin{equation}
		\label{eqn:forakk}
		2 \mathbf{v}_{r-1} = \mathbf{v}_{r-2} - \mathbf{v}_{r-2}^\prime + \mathbf{v}_r + \mathbf{v}_r^\prime.
	\end{equation}
	Take the inner product of $\mathbf{v}$ and equation~\eqref{eqn:forakk} to give $\inprod{\mathbf{v}, \mathbf{v}_r} = -\inprod{\mathbf{v}, \mathbf{v}_r^\prime}$.
	
	\paragraph{Case 1} % (fold)
	\label{par:case_w_c_b_1_i_}
		$\inprod{\mathbf{v}, \mathbf{v}_{r}} = 1+i$.
		By above, we have $\inprod{\mathbf{v}, \mathbf{v}_{r}^\prime} = -1-i$.
		Hence, the graph $P_{2r+1}^\prime \cup \left \{v\right \}$ is equal to the visibly cyclotomic $\Z[i]$-graph $C_{2(r+1)}$, and hence it too is cyclotomic.
		
		It remains to show that every vertex of $V(P_{2r+1}) \cup \left \{v\right \}$ is $( P_{2r+1}^\prime \cup \left \{v\right \} )$-saturated in $P_{2r+1} \cup \left \{v\right \}$.
		By Lemma~\ref{lem:saturationweighted}, this immediately reduces to showing that both $v$ and $v_r$ are $( P_{2r+1}^\prime \cup \left \{v\right \} )$-saturated in $P_{2r+1} \cup \left \{v\right \}$.
		
		First we treat $v$.
		Suppose a vertex $x \in V(G) \backslash V(P_{2r+1} \cup \left \{ v \right \})$ is adjacent to $v$.
		The exclusion of triangles and the excluded subgraphs $XB_9$ and $YB_1$ force $x$ to be adjacent to the vertex $v_{r-1}$ which is $( P_{2r+1}^\prime \cup \left \{v\right \} )$-saturated in $P_{2r+1} \cup \left \{v\right \}$.
		Therefore $\mathbf{x}$ is switch-equivalent to $\mathbf{v}_{r}^\prime$.
		
		It remains to show that $v_r$ is $( P_{2r+1}^\prime \cup \left \{v\right \} )$-saturated in $P_{2r+1} \cup \left \{v\right \}$.
		Suppose that $x \in V(G) \backslash V(P_{2r+1} \cup \left \{ v \right \})$ is adjacent to $v_r$.
		Since all possible uncharged triangles have been excluded, we have that $\mathbf{x}$ is orthogonal to both $\mathbf{v}_{r-1}$ and $\mathbf{v}$.
		And the excluded subgraphs $XB_7$ and $XB_{10}$ force $x$ to be adjacent to the vertex $v_{r-2}$ which is $( P_{2r+1}^\prime \cup \left \{v\right \} )$-saturated in $P_{2r+1} \cup \left \{v\right \}$.
		Therefore $\mathbf{x}$ is switch-equivalent to $\mathbf{v}_{r-1}^\prime$.
		We have shown that both $v$ and $v_r$ are $( P_{2r+1}^\prime \cup \left \{v\right \} )$-saturated in $P_{2r+1} \cup \left \{v\right \}$.
		
		Since each vertex of $\mathcal V_G(P_{2r+1} \cup \left \{v\right \})$ is $( P_{2r+1}^\prime \cup \left \{v\right \} )$-saturated in $P_{2r+1} \cup \left \{v\right \}$, the vertices of $G$ correspond to vertices of $P_{2r+1}^\prime \cup \left \{v\right \}$.
		This correspondence is one to one, since otherwise, if two vertices $x$ and $y$ of $G$ were both switch-equivalent to the same vertex $z$, then $\abs{\inprod{\mathbf{x}, \mathbf{y}}} = 2$, and $YB_4$ has been excluded.
		Since $P_{2r+1}^\prime \cup \left \{v\right \}$ is equal to $C_{2(r+1)}$, $G$ is equivalent to a subgraph of $C_{2(r+1)}$.
	% paragraph case_w_c_b_1_i_ (end)
	\paragraph{Case 2} % (fold)
	\label{par:case_2}
		$\inprod{\mathbf{v}, \mathbf{v}_r} = 1$.
		By above, we have $\inprod{\mathbf{v}, \mathbf{v}_r^\prime} = -1$.
		We have established a subgraph of $G$ equivalent to $P_{2(r+1)+1}$.
	% paragraph case_2 (end)
\end{proof}
% subsection inductive_lemmata (end)

\subsection{$\mathcal L_2$-free $\Z[i]$-graphs on up to 7 vertices}

Let $G$ be an $\mathcal L_2$-free $\Z[i]$-graph.
If $G$ does not contain an edge with a weight of norm at least $2$ then $G$ has been classified in Theorem~\ref{thm:classununzi}.
Since $G$ is cyclotomic, it cannot be equivalent to a graph containing an edge of norm greater than $4$.
We have excluded $YB_4$ and no element of $\Z[i]$ has norm $3$, so we can assume that $G$ contains an edge of norm $2$.
The growing process is similar to that described in Section~\ref{sec:unchargedS9}, but in this case we can start the process with a weight-$(1+i)$ edge.
From this process, we have exhaustively computed (up to equivalence) all $\mathcal L_2$-free $\Z[i]$-graphs on up to $7$ vertices.
Out of these graphs, each one on $7$ vertices contains a subgraph equivalent to $P_{7}$ ($4$ vertices).
Again, we note that this computation can be done by hand.

From the above computation and by iteratively applying Lemma~\ref{lem:unchargedinduct} we have the following lemma.

\begin{lemma}\label{lem:L2free}
	Let $G$ be an $\mathcal L_2$-free $\Z[i]$-graph having at least one edge-weight of norm $2$.
	Then $G$ is equivalent to a subgraph of $C_{2k}$ for some $k \geqslant 2$.
\end{lemma}

Together with the computation of the maximal connected cyclotomic $\Z[i]$-graphs containing the excluded subgraphs of type II from the list $\mathcal L_2$ (see Figure~\ref{fig:xgraphs2}), we have proved Theorem~\ref{thm:classunzi}.

	\section{Proof of Theorem~\ref{thm:classchzi}}
	\label{sec:growZicharge}
	
	In this section we prove Theorem~\ref{thm:classchzi}.
	Let $G$ be a cyclotomic $\Z[i]$-graph.
	As in Section~\ref{sec:proof2}, $G$ can have edge-weights coming only from the set $\left \{0,\pm 1, \pm i, \pm 1 \pm i, \pm 2, \pm 2i \right \}$.
	Moreover, since we are actually studying Hermitian matrices, we allow $G$ to contain only rational integer charges, and by Lemma~\ref{lem:maxDeg4}, this immediately restricts the charges to coming from the set $\left \{0,\pm 1, \pm 2 \right \}$. 

	\subsection{Excluded subgraphs}
	
	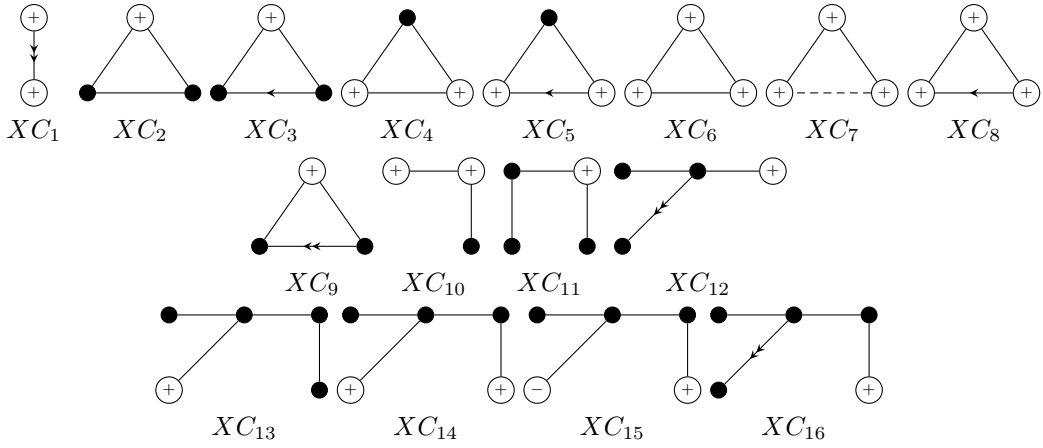
\begin{figure}[h!tbp]
		\centering
		\begin{tikzpicture}
		\begin{scope}[scale=1, auto]
			\foreach \pos/\name/\sign/\charge in {{(0,1)/a/pc/+}, {(0,0)/b/pc/+}}
				\node[\sign] (\name) at \pos {$\charge$}; % {$\name$};
			\node at (0,-0.5) {$XC_1$};
			\foreach \edgetype/\source/ \dest /\weight in {wwedge2/a/b/{}}
			\path[\edgetype] (\source) -- node[weight] {$\weight$} (\dest);
		\end{scope}
		\end{tikzpicture}
		\begin{tikzpicture}
		\begin{scope}[scale=1, auto]
			\foreach \pos/\name/\sign/\charge in {{(0.7,1)/a/pc/+}, {(0,0)/b/zc/{}}, {(1.4,0)/c/zc/{}}}
				\node[\sign] (\name) at \pos {$\charge$}; % {$\name$};
			\node at (0.7,-0.5) {$XC_2$};
			\foreach \edgetype/\source/ \dest /\weight in {pedge/a/b/{}, pedge/c/b/{}, pedge/a/c/{}}
			\path[\edgetype] (\source) -- node[weight] {$\weight$} (\dest);
		\end{scope}
		\end{tikzpicture}
		\begin{tikzpicture}[scale=1, auto]
		\begin{scope}[xshift=6cm, yshift=-0.5cm]	
			\foreach \pos/\name/\sign/\charge in {{(0.7,1)/a/pc/+}, {(0,0)/b/zc/{}}, {(1.4,0)/c/zc/{}}}
				\node[\sign] (\name) at \pos {$\charge$}; % {$\name$};
			\node at (0.7,-0.5) {$XC_3$};
			\foreach \edgetype/\source/ \dest /\weight in {pedge/a/b/{}, wedge/c/b/{}, pedge/a/c/{}}
			\path[\edgetype] (\source) -- node[weight] {$\weight$} (\dest);
		\end{scope}
		\end{tikzpicture}
		\begin{tikzpicture}
		\begin{scope}[scale=1, auto]	
			\foreach \pos/\name/\sign/\charge in {{(0.7,1)/a/zc/{}}, {(0,0)/b/pc/+}, {(1.4,0)/c/pc/+}}
				\node[\sign] (\name) at \pos {$\charge$}; % {$\name$};
			\node at (0.7,-0.5) {$XC_4$};
			\foreach \edgetype/\source/ \dest /\weight in {pedge/a/b/{}, pedge/c/b/{}, pedge/a/c/{}}
			\path[\edgetype] (\source) -- node[weight] {$\weight$} (\dest);
		\end{scope}
		\end{tikzpicture}
		\begin{tikzpicture}[scale=1, auto]
		\begin{scope}[xshift=9cm, yshift=-0.5cm]	
			\foreach \pos/\name/\sign/\charge in {{(0.7,1)/a/zc/{}}, {(0,0)/b/pc/+}, {(1.4,0)/c/pc/+}}
				\node[\sign] (\name) at \pos {$\charge$}; % {$\name$};
			\node at (0.7,-0.5) {$XC_5$};
			\foreach \edgetype/\source/ \dest /\weight in {pedge/a/b/{}, wedge/c/b/{}, pedge/a/c/{}}
			\path[\edgetype] (\source) -- node[weight] {$\weight$} (\dest);
		\end{scope}
		\end{tikzpicture}
		\begin{tikzpicture}
		\begin{scope}[scale=1, auto]	
			\foreach \pos/\name/\sign/\charge in {{(0.7,1)/a/pc/+}, {(0,0)/b/pc/+}, {(1.4,0)/c/pc/+}}
				\node[\sign] (\name) at \pos {$\charge$}; % {$\name$};
			\node at (0.7,-0.5) {$XC_6$};
			\foreach \edgetype/\source/ \dest /\weight in {pedge/a/b/{}, pedge/c/b/{}, pedge/a/c/{}}
			\path[\edgetype] (\source) -- node[weight] {$\weight$} (\dest);
		\end{scope}
		\end{tikzpicture}
		\begin{tikzpicture}
		\begin{scope}[scale=1, auto]	
			\foreach \pos/\name/\sign/\charge in {{(0.7,1)/a/pc/+}, {(0,0)/b/pc/+}, {(1.4,0)/c/pc/+}}
				\node[\sign] (\name) at \pos {$\charge$}; % {$\name$};
			\node at (0.7,-0.5) {$XC_7$};
			\foreach \edgetype/\source/ \dest /\weight in {pedge/a/b/{}, nedge/c/b/{}, pedge/a/c/{}}
			\path[\edgetype] (\source) -- node[weight] {$\weight$} (\dest);
		\end{scope}
		\end{tikzpicture}
		\begin{tikzpicture}[scale=1, auto]
		\begin{scope}[xshift=9cm, yshift=-0.5cm]	
			\foreach \pos/\name/\sign/\charge in {{(0.7,1)/a/pc/+}, {(0,0)/b/pc/+}, {(1.4,0)/c/pc/+}}
				\node[\sign] (\name) at \pos {$\charge$}; % {$\name$};
			\node at (0.7,-0.5) {$XC_8$};
			\foreach \edgetype/\source/ \dest /\weight in {pedge/a/b/{}, wedge/c/b/{}, pedge/a/c/{}}
			\path[\edgetype] (\source) -- node[weight] {$\weight$} (\dest);
		\end{scope}
		\end{tikzpicture}
		
		\begin{tikzpicture}[scale=1, auto]
		\begin{scope}[xshift=6cm, yshift=-0.5cm]	
			\foreach \pos/\name/\sign/\charge in {{(0.7,1)/a/pc/+}, {(0,0)/b/zc/{}}, {(1.4,0)/c/zc/{}}}
				\node[\sign] (\name) at \pos {$\charge$}; % {$\name$};
			\node at (0.7,-0.5) {$XC_9$};
			\foreach \edgetype/\source/ \dest /\weight in {pedge/a/b/{}, wwedge/c/b/{}, pedge/a/c/{}}
			\path[\edgetype] (\source) -- node[weight] {$\weight$} (\dest);
		\end{scope}
		\end{tikzpicture}
		\begin{tikzpicture}
		\begin{scope}[scale=1, auto]	
			\foreach \pos/\name/\sign/\charge in {{(0,1)/a/pc/+}, {(1,1)/b/pc/+}, {(1,0)/c/vertex/{}}, {(0,-0.5)/d/empty/{}}}
				\node[\sign] (\name) at \pos {$\charge$}; % {$\name$};
			\node at (0.5,-0.5) {$XC_{10}$};
			\foreach \edgetype/\source/ \dest /\weight in {pedge/a/b/{}, pedge/c/b/{}}
			\path[\edgetype] (\source) -- node[weight] {$\weight$} (\dest);
		\end{scope}
		\end{tikzpicture}
		\begin{tikzpicture}
		\begin{scope}[scale=1, auto]	
			\foreach \pos/\name/\sign/\charge in {{(0,0)/a/vertex/{}}, {(0,1)/b/vertex/{}}, {(1,1)/c/pc/+}, {(1,0)/d/vertex/{}}}
				\node[\sign] (\name) at \pos {$\charge$}; % {$\name$};
			\node at (0.5,-0.5) {$XC_{11}$};
			\foreach \edgetype/\source/ \dest /\weight in {pedge/a/b/{}, pedge/c/b/{}, pedge/c/d/{}}
			\path[\edgetype] (\source) -- node[weight] {$\weight$} (\dest);
		\end{scope}
		\end{tikzpicture}
		\begin{tikzpicture}
			\begin{scope}[scale=1, auto]
				\foreach \pos/\name/\sign/\charge in {{(0,0)/a1/zc/{}}, {(1,1)/b1/zc/{}}, {(2,1)/c1/pc/+}, {(0,1)/c2/zc/{}}}
					\node[\sign] (\name) at \pos {$\charge$};
				\foreach \edgetype/\source/ \dest in {wwedge/b1/a1, pedge/b1/c1, pedge/b1/c2}
					\path[\edgetype] (\source) -- (\dest);
				\node at (1,-0.5) {$XC_{12}$};
			\end{scope}
		\end{tikzpicture}
		
		\begin{tikzpicture}
			\begin{scope}[scale=1, auto]
				\foreach \pos/\name/\sign/\charge in {{(0,0)/a1/pc/+}, {(1,1)/b1/zc/{}}, {(2,1)/c1/zc/{}}, {(2,0)/d1/zc/{}}, {(0,1)/c2/zc/{}}}
					\node[\sign] (\name) at \pos {$\charge$};
				\foreach \edgetype/\source/ \dest in {pedge/a1/b1, pedge/b1/c1, pedge/d1/c1, pedge/b1/c2}
					\path[\edgetype] (\source) -- (\dest);
				\node at (1,-0.5) {$XC_{13}$};
			\end{scope}
		\end{tikzpicture}
		% \begin{tikzpicture}
		% 	\begin{scope}[scale=1, auto]
		% 		\foreach \pos/\name/\sign/\charge in {{(0,0)/a1/zc/{}}, {(1,1)/b1/zc/{}}, {(2,1)/c1/zc/{}}, {(2,0)/d1/zc/{}}, {(0,1)/c2/zc/{}}}
		% 			\node[\sign] (\name) at \pos {$\charge$};
		% 		\foreach \edgetype/\source/ \dest in {wwedge/b1/a1, pedge/b1/c1, pedge/d1/c1, pedge/b1/c2}
		% 			\path[\edgetype] (\source) -- (\dest);
		% 		\node at (1,-1) {$XC_{13}$};
		% 	\end{scope}
		% \end{tikzpicture}
		\begin{tikzpicture}
			\begin{scope}[scale=1, auto]
				\foreach \pos/\name/\sign/\charge in {{(0,0)/a1/pc/+}, {(1,1)/b1/zc/{}}, {(2,1)/c1/zc/{}}, {(2,0)/d1/pc/+}, {(0,1)/c2/zc/{}}}
					\node[\sign] (\name) at \pos {$\charge$};
				\foreach \edgetype/\source/ \dest in {pedge/a1/b1, pedge/b1/c1, pedge/d1/c1, pedge/b1/c2}
					\path[\edgetype] (\source) -- (\dest);
				\node at (1,-0.5) {$XC_{14}$};
			\end{scope}
		\end{tikzpicture}
		\begin{tikzpicture}
			\begin{scope}[scale=1, auto]
				\foreach \pos/\name/\sign/\charge in {{(0,0)/a1/pc/-}, {(1,1)/b1/zc/{}}, {(2,1)/c1/zc/{}}, {(2,0)/d1/pc/+}, {(0,1)/c2/zc/{}}}
					\node[\sign] (\name) at \pos {$\charge$};
				\foreach \edgetype/\source/ \dest in {pedge/a1/b1, pedge/b1/c1, pedge/d1/c1, pedge/b1/c2}
					\path[\edgetype] (\source) -- (\dest);
				\node at (1,-0.5) {$XC_{15}$};
			\end{scope}
		\end{tikzpicture}
		\begin{tikzpicture}
			\begin{scope}[scale=1, auto]
				\foreach \pos/\name/\sign/\charge in {{(0,0)/a1/zc/{}}, {(1,1)/b1/zc/{}}, {(2,1)/c1/zc/{}}, {(2,0)/d1/pc/+}, {(0,1)/c2/zc/{}}}
					\node[\sign] (\name) at \pos {$\charge$};
				\foreach \edgetype/\source/ \dest in {wwedge/b1/a1, pedge/b1/c1, pedge/d1/c1, pedge/b1/c2}
					\path[\edgetype] (\source) -- (\dest);
				\node at (1,-0.5) {$XC_{16}$};
			\end{scope}
		\end{tikzpicture}
		\caption{some non-cyclotomic charged $\Z[i]$-graphs.}
		\label{fig:xgraphsCI}
	\end{figure}

	\begin{figure}[h!tbp]
		\centering
		\begin{tikzpicture}
		\begin{scope}[scale=1, auto]
			\foreach \pos/\name/\sign/\charge in {{(0,1)/a/pc/+}, {(0,0)/b/pc/-}}
				\node[\sign] (\name) at \pos {$\charge$}; % {$\name$};
			\node at (0,-0.5) {$YC_1$};
			\foreach \edgetype/\source/ \dest /\weight in {pedge/a/b/{}}
			\path[\edgetype] (\source) -- node[weight] {$\weight$} (\dest);
		\end{scope}
		\end{tikzpicture}
		\begin{tikzpicture}
		\begin{scope}[scale=1, auto]
			\foreach \pos/\name/\sign/\charge in {{(0,1)/a/pc/+}, {(0,0)/b/pc/-}}
				\node[\sign] (\name) at \pos {$\charge$}; % {$\name$};
			\node at (0,-0.5) {$YC_2$};
			\foreach \edgetype/\source/ \dest /\weight in {wwedge2/a/b/{}}
			\path[\edgetype] (\source) -- node[weight] {$\weight$} (\dest);
		\end{scope}
		\end{tikzpicture}
		\begin{tikzpicture}
		\begin{scope}[scale=1, auto]
			\foreach \pos/\name/\sign/\charge in {{(0,1)/a/pc/+}, {(0,0)/b/zc/{}}}
				\node[\sign] (\name) at \pos {$\charge$}; % {$\name$};
			\node at (0,-0.5) {$YC_3$};
			\foreach \edgetype/\source/ \dest /\weight in {wwedge2/a/b/{}}
			\path[\edgetype] (\source) -- node[weight] {$\weight$} (\dest);
		\end{scope}
		\end{tikzpicture}
		\begin{tikzpicture}
		\begin{scope}[scale=1, auto]	
			\foreach \pos/\name/\sign/\charge in {{(0,1)/a/pc/+}, {(1,1)/b/zc/{}}, {(0,0)/c/pc/+}}
				\node[\sign] (\name) at \pos {$\charge$}; % {$\name$};
			\node at (0.5,-0.5) {$YC_4$};
			\foreach \edgetype/\source/ \dest /\weight in {pedge/a/b/{}, pedge/c/b/{}}
			\path[\edgetype] (\source) -- node[weight] {$\weight$} (\dest);
		\end{scope}
		\end{tikzpicture}
		\begin{tikzpicture}
		\begin{scope}[scale=1, auto]	
			\foreach \pos/\name/\sign/\charge in {{(0,1)/a/pc/+}, {(1,1)/b/zc/{}}, {(0,0)/c/pc/-}}
				\node[\sign] (\name) at \pos {$\charge$}; % {$\name$};
			\node at (0.5,-0.5) {$YC_5$};
			\foreach \edgetype/\source/ \dest /\weight in {pedge/a/b/{}, pedge/c/b/{}}
			\path[\edgetype] (\source) -- node[weight] {$\weight$} (\dest);
		\end{scope}
		\end{tikzpicture}
		% \begin{tikzpicture}
		% \begin{scope}[scale=1, auto]	
		% 	\foreach \pos/\name/\sign/\charge in {{(0,1)/a/pc/+}, {(1,1)/b/zc/{}}, {(0,0)/c/zc/{}}}
		% 		\node[\sign] (\name) at \pos {$\charge$}; % {$\name$};
		% 	\node at (0.5,-0.5) {$YC_6$};
		% 	\foreach \edgetype/\source/ \dest /\weight in {pedge/a/b/{}, wwedge/c/b/{}}
		% 	\path[\edgetype] (\source) -- node[weight] {$\weight$} (\dest);
		% \end{scope}
		% \end{tikzpicture}
		\begin{tikzpicture}[scale=1, auto]
		\begin{scope}[xshift=6cm, yshift=-0.5cm]	
			\foreach \pos/\name/\sign/\charge in {{(0.7,1)/a/pc/+}, {(0,0)/b/zc/{}}, {(1.4,0)/c/zc/{}}}
				\node[\sign] (\name) at \pos {$\charge$}; % {$\name$};
			\node at (0.7,-0.5) {$YC_6$};
			\foreach \edgetype/\source/ \dest /\weight in {pedge/a/b/{}, nedge/c/b/{}, pedge/a/c/{}}
			\path[\edgetype] (\source) -- node[weight] {$\weight$} (\dest);
		\end{scope}
		\end{tikzpicture}
		\begin{tikzpicture}[scale=1, auto]
		\begin{scope}[xshift=6cm, yshift=-0.5cm]	
			\foreach \pos/\name/\sign/\charge in {{(0.7,1)/a/pc/+}, {(0,0)/b/zc/{}}, {(1.4,0)/c/zc/{}}}
				\node[\sign] (\name) at \pos {$\charge$}; % {$\name$};
			\node at (0.7,-0.5) {$YC_7$};
			\foreach \edgetype/\source/ \dest /\weight in {pedge/a/b/{}, wwnedge/c/b/{}, pedge/a/c/{}}
			\path[\edgetype] (\source) -- node[weight] {$\weight$} (\dest);
		\end{scope}
		\end{tikzpicture}
		\begin{tikzpicture}[scale=1, auto]
		\begin{scope}	
			\foreach \pos/\name/\type/\charge in {{(0,0)/a/pc/{2}}}
				\node[\type] (\name) at \pos {$\charge$}; % {$\name$};
			\node at (0,-0.8) {$YC_8$};
		\end{scope}
		\end{tikzpicture}
		\caption{some charged cyclotomic $\Z[i]$-graphs that are contained as subgraphs of fixed maximal connected cyclotomic $\Z[i]$-graphs.}
		\label{fig:xgraphsCII}
	\end{figure}
	\begin{table}[ht]
		\begin{center}
		\begin{tabular}{c|c}
			Excluded subgraph & Maximal cyclotomics \\
			\hline
			$YC_1$ & $C^{+-}_4$, $S_4$, $S_4^\dag$, $S_7$, $S_8$, and $S_8^\prime$ \\
			$YC_2$ & $S_4$ \\
			$YC_3$ & $C_3$ \\
			$YC_4$ & $C_6^{++}$ and $S_7$\\
			$YC_5$ & $C_6^{+-}$ and $S_8^\prime$ \\
			$YC_6$ & $S_7$ and $S_8^\prime$ \\
			$YC_7$ & $S_4^\dag$ \\
			$YC_8$ & $S_1$
		\end{tabular}
		\end{center}
		\caption{Excluded subgraphs from Figure~\ref{fig:xgraphsCII} and (up to equivalence) their containing maximal connected cyclotomic $\Z[i]$-graphs.}
		\label{tab:exgraphsicharge}
	\end{table}
	
	In Table~\ref{tab:exgraphsicharge} we list each excluded subgraph of type II in Figure~\ref{fig:xgraphsCII} along with every maximal connected cyclotomic $\Z[i]$-graph that contains it.
	Let $\mathcal L_3$ consist of the excluded subgraphs of type II in Figures~\ref{fig:xgraphsi}, \ref{fig:xgraphs2}, and \ref{fig:xgraphsCII}.
	Note that, up to equivalence, there is exactly one charged $\Z[i]$-triangle that \emph{can} be a subgraph of an $\mathcal L_3$-free $\Z[i]$-graph, namely the triangle
	
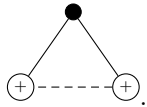
\begin{center}
		\begin{tikzpicture}
		\begin{scope}[scale=1, auto]	
			\foreach \pos/\name/\sign/\charge in {{(0.7,1)/a/zc/{}}, {(0,0)/b/pc/+}, {(1.4,0)/c/pc/+}}
				\node[\sign] (\name) at \pos {$\charge$}; % {$\name$};
			\foreach \edgetype/\source/ \dest /\weight in {pedge/a/b/{}, nedge/c/b/{}, pedge/a/c/{}}
			\path[\edgetype] (\source) -- node[weight] {$\weight$} (\dest);
		\end{scope}
		\end{tikzpicture}.
	\end{center}
	In this section, the notion of a saturated vertex will depend on the list $\mathcal L_3$.
	
	\subsection{Inductive lemmata} % (fold)
	\label{sub:inductive_lemmata_charged}
	Define $P_{2r}$ and $P^\prime_{2r}$ with the following $\Z$-graph
	
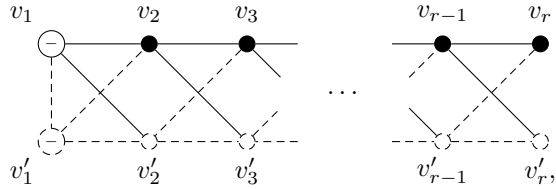
\begin{center}
	\begin{tikzpicture}
		\begin{scope}[auto, scale=1.3]
				\foreach \pos/\name/\sign/\charge in {{(-1,0)/bgn2/ghostc/-}, {(-1,1)/bgn1/pc/-}}
					\node[\sign] (\name) at \pos {$\charge$};
			\foreach \type/\pos/\name in {{ghost/(0,0)/a2}, {vertex/(0,1)/a1}, {vertex/(1,1)/b1}, {ghost/(1,0)/b2}, {empty/(1.6,1)/b11}, {empty/(1.6,0)/b21}, {empty/(1.4,0.6)/b12}, {empty/(1.4,0.4)/b22}, {empty/(2.4,1)/c11}, {empty/(2.4,0)/c21}, {empty/(2.6,0.6)/c12}, {empty/(2.6,0.4)/c22}, {vertex/(3,1)/c1}, {ghost/(3,0)/c2}, {vertex/(4,1)/d1}, {ghost/(4,0)/d2}}
				\node[\type] (\name) at \pos {};
			\foreach \pos/\name in {{(2,0.5)/\dots}, {(-1.3,1.3)/v_1}, {(-1.3,-0.3)/v_1^\prime}, {(0,1.3)/v_2}, {(0,-0.3)/v_2^\prime}, {(1,1.3)/v_3}, {(1,-0.3)/v_3^\prime}, {(3,1.3)/v_{r-1}}, {(3,-0.3)/v_{r-1}^\prime}, {(4,1.3)/v_r}, {(4,-0.3)/v_r^\prime,}}
				\node at \pos {$\name$};
			\foreach \edgetype/\source/ \dest in {pedge/bgn1/a1, pedge/bgn1/a2, nedge/bgn1/bgn2, nedge/bgn2/a1, nedge/bgn2/a2, nedge/b1/a2, pedge/a1/b1, pedge/a1/b2, nedge/a2/b2, nedge/b21/b2, pedge/b1/b11, pedge/b1/b12, nedge/b2/b22, pedge/c11/c1, nedge/c12/c1, pedge/c22/c2, nedge/c21/c2, nedge/d1/c2, pedge/c1/d1, pedge/c1/d2, nedge/c2/d2}
				\path[\edgetype] (\source) -- (\dest);
		\end{scope}
	\end{tikzpicture}	
\end{center}
	where $r \geqslant 1$.
	Here, the set of hollow vertices of $P_{2r}$ is the set $V(P^\prime_{2r}) \backslash V(P_{2r})$.
	Clearly both $P_{2r}$ and $P^\prime_{2r}$ are cyclotomic since they are contained in a $\Z$-graph equivalent to $C^{++}_{2(r+1)}$.
	Having chosen Gram vectors $\mathbf{v}_1, \dots \mathbf{v}_r$, by an argument similar to Lemma~\ref{lem:adeqvec}, we can write
	\begin{align}
		\label{eqn:forv1chargep}
		\mathbf{v}^\prime_1 &= -\mathbf{v}_1; \\
		\label{eqn:forv1charge}
		\mathbf{v}_2^\prime &= - \mathbf{v}_2 + 2 \mathbf{v}_1  ;
	\end{align}
	and
	\begin{equation}
		\label{eqn:forv2charge}
		\mathbf{v}_3^\prime = -\mathbf{v}_3 + 2 \mathbf{v}_2 - 2\mathbf{v}_1.
	\end{equation}
	
	\begin{lemma}\label{lem:gramvectsindcharge}
		In $P_{2r}$ for $r \geqslant 3$, we can write the Gram vector for each hollow vertex in terms of Gram vectors of the vertices as follows:
		\begin{equation*}
			\mathbf{v}^\prime_{t} = -\mathbf{v}_{t} - 2\sum_{j=1}^{t-1} (-1)^{t+j} \mathbf{v}_{j}, \quad \text{ for } t \in \left \{1,\dots, r\right \}.
		\end{equation*}
	\end{lemma}
	\begin{proof}
		By induction using equations \eqref{eqn:forv1chargep}, \eqref{eqn:forv1charge}, and \eqref{eqn:forv2charge}, and Lemma~\ref{lem:maxDeg4}.  
	\end{proof}
	
	\begin{lemma}[Saturated vertices]\label{lem:saturationcharged}
		Let $G$ be an $\mathcal L_3$-free $\Z[i]$-graph containing $P_{2r}$ with $r \geqslant 3$.
		Then, for each vertex $v \in V^\prime_4(P_{2r})$, we have $N_G(v) = N^\prime_G(v)$.
		Hence, each vertex in $V^\prime_4(P_{2r})$ is $P_{2r}^\prime$-saturated in $P_{2r}$.
	\end{lemma}
	\begin{proof}
		Fix Gram vectors for $P_{2r}^\prime$.
		Consider a vertex $v \in V(G)\backslash V(P_{2r})$.
		Suppose that $v$ is adjacent to the vertex $v_j \in V^\prime_4(P_{2r})$ for some $j \in \left \{1, \ldots, r-1 \right \}$.
		Without loss of generality, either $\inprod{\mathbf{v}, \mathbf{v}_j} = 1$ or $\inprod{\mathbf{v}, \mathbf{v}_j} = 1 + i$.
		
		Suppose first that $j = 1$.
		If $v$ is charged, then the excluded subgraphs $YC_2$ and $XC_1$ rule out the possibility of $\inprod{\mathbf{v}, \mathbf{v}_1} = 1 + i$, and so we assume $\inprod{\mathbf{v}, \mathbf{v}_1} = 1$.
		Moreover, $YC_1$ forces $v$ to have charge $-1$.
		Therefore the inner product $\inprod{\mathbf{v} + \mathbf{v}_1^\prime,\mathbf{v} + \mathbf{v}_1^\prime}$ is zero and hence $\mathbf{v} = -\mathbf{v}_1^\prime$.
		On the other hand, if $v$ is uncharged, then the excluded subgraph $YC_3$ rules out the possibility of $\inprod{\mathbf{v}, \mathbf{v}_1} = 1 + i$, and so we assume $\inprod{\mathbf{v}, \mathbf{v}_1} = 1$.
		The exclusion of the triangles containing exactly one charged vertex $XC_2$, $XC_3$, and $YC_6$ forces $\mathbf{v}$ to be orthogonal to $\mathbf{v}_2$ and by taking the inner product of $\mathbf{v}$ and equation~\eqref{eqn:forv1charge} we find that $\mathbf{v} = \mathbf{v}_2^\prime$.
		Hence, the vertex $v_1$ is $P_{2r}^\prime$-saturated in $P_{2r}$.
		
		Now suppose that $j=2$.
		If $v$ is adjacent to $v_1$, then, since $v_1$ is $P_{2r}^\prime$-saturated in $P_{2r}$ and $v_1^\prime$ is adjacent to $v_2$, we must have $\mathbf{v}$ switch-equivalent to $\mathbf{v}_1^\prime$.
		Otherwise, if $v$ is not adjacent to $v_1$, the excluded subgraphs $YC_3$, $YC_4$, and $YC_5$ imply that $v$ is uncharged.
		Moreover, $XB_1$, $XB_2$, $XB_3$, and $XC_{12}$ rule out the possibility of $\inprod{\mathbf{v}, \mathbf{v}_2} = 1 + i$, so we assume $\inprod{\mathbf{v}, \mathbf{v}_2} = 1$.
		The exclusion of uncharged triangles $YA_4$ and $YA_5$ forces $\inprod{\mathbf{v}, \mathbf{v}_3} = 0$ and using equation~\eqref{eqn:forv2charge}, we deduce that $\inprod{\mathbf{v} - \mathbf{v}^\prime_3, \mathbf{v} - \mathbf{v}^\prime_3} = 0$.
		Hence $\mathbf{v} = \mathbf{v}_3^\prime$.
		
		Thus $v_1$ and $v_2$ are $P_{2r}^\prime$-saturated in $P_{2r}$.
		If $r = 3$, we are done, hence let $r > 3$.
		We assume that, for $2 < t < r$, each vertex $v_j \in V^\prime_4(P_{2r})$ with $1 \leqslant j < t$ is $P_{2r}^\prime$-saturated in $P_{2r}$.
		It suffices now to show that $v_t$ is $P_{2r}^\prime$-saturated in $P_{2r}$.
		Suppose a vertex $v \in V(G)\backslash V(P_{2r})$ is adjacent to $v_t$.
		We split into cases.
		\paragraph{Case 1} % (fold)
		$v$ is adjacent to $v_{t-2}$.
		By our inductive hypothesis, $v_{t-2}$ is $P_{2r}^\prime$-saturated in $P_{2r}$ and thus $\mathbf{v}$ is switch-equivalent to the Gram vector of some hollow vertex. 
		Moreover, the hollow vertex in question must be adjacent to both $v_t$ and $v_{t-2}$.
		Hence $\mathbf{v}$ is switch-equivalent to $\mathbf{v}_{t-1}^\prime$.		
		% subparagraph subcase_1 (end)

		\paragraph{Case 2} % (fold)
		$v$ is not adjacent to $v_{t-2}$.
		Since we have excluded triangles having at most one charge, $v$ is adjacent to neither $v_{t-1}$ nor $v_{t+1}$.
		If $t=3$ then the excluded subgraphs $XC_{14}$, $XC_{15}$, and $YC_3$ preclude the possibility of $v$ having a charge and the exclusion of $XC_{16}$ means that we can assume $\inprod{ \mathbf{v}, \mathbf{v}_t} = 1$.
		Otherwise, if $t > 3$ then the excluded subgraphs $XB_7$, $XC_{13}$, and $YC_3$ enable us to assume that $v$ is uncharged and $\inprod{ \mathbf{v}, \mathbf{v}_t} = 1$.
		Hence, since $t \geqslant 3$, we can assume that $v$ is uncharged and $\inprod{ \mathbf{v}, \mathbf{v}_t} = 1$.
		If $t = 3$ then, since $\mathbf{v}$ is orthogonal to both $\mathbf{v}_1 = \mathbf{v}_{t-2}$ and $\mathbf{v}_2 = \mathbf{v}_{t-1}$, by Lemma~\ref{lem:gramvectsindcharge}, $\mathbf{v}$ is also orthogonal to $\mathbf{v}_{t-1}^\prime$.
		Suppose that $t > 3$.
		If $v$ were adjacent to $v_j$ for some $j \in \left \{1,\dots,t-3 \right \}$ then, by our inductive hypothesis, $\mathbf{v}$ would be equivalent to the vector corresponding to some hollow vertex adjacent to $v_j$.
		But, since no hollow vertex is adjacent to both $v_t$ and $v_j$ with $j \in \left \{1,\dots,t-3 \right \}$, the Gram vector $\mathbf{v}$ must be orthogonal to $\mathbf{v}_j$ for all $j \in \left \{1,\dots,t-3 \right \}$.
		Therefore, by Lemma~\ref{lem:gramvectsindcharge}, we have $\inprod{\mathbf{v}, \mathbf{v}_{t-1}^\prime} = 0$.
		Appealing to Lemma~\ref{lem:adeqvec}, write
		\begin{equation}
			\label{eqn:vtweightadeq}
			2 \mathbf{v}_t = \mathbf{v}_{t-1} - \mathbf{v}_{t-1}^\prime + \mathbf{v}_{t+1} + \mathbf{v}_{t+1}^\prime.
		\end{equation}   
		By taking the inner product of $\mathbf{v}$ and equation~\eqref{eqn:vtweightadeq} we find that $\mathbf{v} = \mathbf{v}_{t+1}^\prime$.
		As required.
	\end{proof}
	
	Let $G$ be an $\mathcal L_3$-free $\Z[i]$-graph containing $P_{2r}$ with $r \geqslant 3$.
	By the symmetry of $P_{2r}^\prime$, it follows from Lemma~\ref{lem:saturationcharged} that each vertex in $V_4^\prime(\mathcal V_G(P_{2r}))$ is $P_{2r}^\prime$-saturated in $\mathcal V_G(P_{2r})$.
	
	\begin{lemma}\label{lem:othogtohollowscharge}
		Let $G$ be an $\mathcal L_3$-free $\Z[i]$-graph containing $P_{2r}$ with $r\geqslant 3$, where $v_r$ is adjacent to a vertex $v\in V(G)\backslash \mathcal V_G(P_{2r})$.
		Then $\mathbf{v}$ is orthogonal to the vectors $\mathbf{v}_j$, and $\mathbf{v}_j^\prime$, for $j \in \left \{1,\ldots,r-1\right \}$.
	\end{lemma}
	\begin{proof}
		By Lemma~\ref{lem:saturationcharged}, the vertices $v_1$, \ldots, $v_{r-1}$ are $P_{2r}^\prime$-saturated in $P_{2r}$.
		Since $v \not \in \mathcal V_G(P_{2r})$, the Gram vector $\mathbf{v}$ is orthogonal to $\mathbf{v}_j$ for all $j \in \left \{1,\ldots, r-1 \right \}$.
		For each $j \in \left \{1,\ldots, r-1 \right \}$, we can write $\mathbf{v}_j^\prime$ as a linear combination of the Gram vectors $\mathbf{v}_1, \ldots, \mathbf{v}_j$ as in Lemma~\ref{lem:gramvectsindcharge}.
		Hence we have the result.
	\end{proof}
	
	\begin{lemma}\label{lem:unchargedinductcharge}
		Let $G$ be an $\mathcal L_3$-free $\Z[i]$-graph containing $P_{2r}$ with $r\geqslant 3$ where $v_r$ is adjacent to an uncharged vertex $v\in V(G)\backslash \mathcal V_G(P_{2r})$.
		Then either $G$ is contained in $C_{2r+1}$ or $G$ contains $P_{2(r+1)}$.
	\end{lemma}
	\begin{proof}
		Similar to the proof of Lemma~\ref{lem:unchargedinduct}.
	\end{proof}
	
	\begin{lemma}\label{lem:chargedinductcharge}
		Let $G$ be an $\mathcal L_3$-free $\Z[i]$-graph containing $P_{2r}$ with $r\geqslant 3$ where $v_r$ is adjacent to a charged vertex $v\in V(G)\backslash \mathcal V_G(P_{2r})$.
		Then $G$ is contained in either $C_{2(r+1)}^{++}$ or $C_{2(r+1)}^{+-}$.
	\end{lemma}
	\begin{proof}
		Since we have excluded $YC_3$, we can assume that $\inprod{\mathbf{v}, \mathbf{v}_r} = 1$.
		By Lemma~\ref{lem:adeqvec} we have
		\begin{equation}
			\label{eqn:forakcharge}
			2 \mathbf{v}_{r-1} = \mathbf{v}_{r-2} - \mathbf{v}_{r-2}^\prime + \mathbf{v}_r + \mathbf{v}_r^\prime.
		\end{equation}
		By Lemma~\ref{lem:othogtohollowscharge}, $\mathbf{v}$ is orthogonal to $\mathbf{v}_{r-1}$, $\mathbf{v}_{r-2}$, and $\mathbf{v}_{r-2}^\prime$.
		Take the inner product of $\mathbf{v}$ and equation~\eqref{eqn:forakcharge} to give $\inprod{\mathbf{v}, \mathbf{v}_r} = -\inprod{\mathbf{v}, \mathbf{v}_r^\prime}$.
		We split into cases when $v$ has charge $1$ and $-1$ respectively.
	
		\paragraph{Case 1}
			Suppose that $v$ has charge $1$.
			We can write $\mathbf{v}$ in terms of $\mathbf{v}_r$, $\mathbf{v}_r^\prime$, and some vector $\xi$.
			\begin{equation}
				\label{eqn:2forvch2}
				2 \mathbf{v} = \mathbf{v}_r - \mathbf{v}_r^\prime + \xi,
			\end{equation}
			where $\xi$ has length $2\sqrt{2}$ and is orthogonal to both $\mathbf{v}_r$ and $\mathbf{v}_r^\prime$.
			Let $v^\prime$ be a hollow vertex of $P_{2r} \cup \left \{v\right \}$ with Gram vector $\mathbf{v}^\prime = \mathbf{v} - \xi$.
			The $\Z$-graph $P_{2r}^\prime \cup \left \{v, v^\prime \right \}$ is equivalent to the visibly cyclotomic $\Z$-graph $C_{2(r+1)}^{+-}$ and is therefore also cyclotomic.
			It remains to check that both $v$ and $v_r$ are $( P_{2r}^\prime \cup \left \{v, v^\prime \right \} )$-saturated in $P_{2r} \cup \left \{v\right \}$.
			First we treat the vertex $v$.
			Suppose that a vertex $x \in V(G) \backslash V(P_{2r}\cup \left \{v\right \})$ is adjacent to $v$.
			
			Suppose that $x$ is charged.
			The excluded subgraphs $XC_1$ and $YC_2$ rule out the possibility of $\inprod{\mathbf{x}, \mathbf{v}} = 1 + i$, and so we assume $\inprod{\mathbf{x}, \mathbf{v}} = 1$.
			Moreover, $YC_1$ forces $x$ to have charge $1$ and $XC_{10}$ forces $x$ to be adjacent to $v_r$.
			The exclusion of $XC_4$, $XC_5$, and $YC_3$ means that we must have $\inprod{\mathbf{x},\mathbf{v}_r} = -1$.
			Now, if $x$ were adjacent to $v_j$ for some $j \in \left \{1,\dots,r-1 \right \}$ then, since such a vertex $v_j$ is $( P_{2r}^\prime \cup \left \{v, v^\prime \right \} )$-saturated in $P_{2r} \cup \left \{v\right \}$, $x$ would be switch-equivalent to some hollow vertex adjacent to $v_j$.
			Such hollow vertices are uncharged, hence, since $x$ is charged, its Gram vector $\mathbf{x}$ must be orthogonal to $\mathbf{v}_j$ for all $j \in \left \{1,\dots,r-1 \right \}$.
			By Lemma~\ref{lem:saturationcharged}, $\mathbf{x}$ is also orthogonal to $\mathbf{v}_j^\prime$ for all $j \in \left \{1,\dots,r-1 \right \}$.
			In particular, $\mathbf{x}$ is orthogonal to $\mathbf{v}_{r-1}$, $\mathbf{v}_{r-2}$, and $\mathbf{v}_{r-2}^\prime$.
			From taking the inner product of $\mathbf{x}$ and equation~\eqref{eqn:forakcharge} we have $\inprod{\mathbf{x},\mathbf{v}_r^\prime} = -\inprod{\mathbf{x},\mathbf{v}_r}$.
			Hence $\inprod{\mathbf{x},\mathbf{v}_r^\prime} = 1$, and the inner product of $\mathbf{x}$ and equation~\eqref{eqn:2forvch2} yields $\inprod{\mathbf{x},\xi} = 4$.
			Therefore we have $\mathbf{x} = -\mathbf{v}^\prime$.
			
			On the other hand, if $x$ is uncharged, then the excluded subgraph $YC_3$ rules out the possibility of $\inprod{\mathbf{x}, \mathbf{v}} = 1 + i$, and so we assume $\inprod{\mathbf{x}, \mathbf{v}} = 1$.
			And the exclusion of the triangles having exactly one charge forces $\mathbf{x}$ to be orthogonal to $\mathbf{v}_r$.
			Since $XC_{11}$ has been excluded, $x$ must be adjacent to the vertex $v_{r-1}$ which is $( P_{2r}^\prime \cup \left \{v, v^\prime \right \} )$-saturated in $P_{2r} \cup \left \{v\right \}$, and hence $\mathbf{x}$ must be switch-equivalent to $\mathbf{v}_r^\prime$.
			Therefore, we have proved that $v$ is $( P_{2r}^\prime \cup \left \{v, v^\prime \right \} )$-saturated in $P_{2r} \cup \left \{v\right \}$ and it remains to show that $v_r$ is $( P_{2r}^\prime \cup \left \{v, v^\prime \right \} )$-saturated in $P_{2r} \cup \left \{v\right \}$.
			
			Suppose that $x$ is adjacent to $v_r$.
			Since we have excluded $YC_3$, $YC_4$, $YC_5$, $XC_{12}$, $XC_{13}$, and $XC_{15}$, the vertex $x$ must be adjacent to at least one of the vertices $v_{r-1}$ or $v$.
			Both of these vertices are $( P_{2r}^\prime \cup \left \{v, v^\prime \right \} )$-saturated in $P_{2r} \cup \left \{v\right \}$ and hence $x$ is switch-equivalent to some hollow vertex as required.
			
			We have, then, that each vertex of $\mathcal V_G(P_{2r} \cup \left \{v \right \})$ is $( P_{2r}^\prime \cup \left \{v, v^\prime \right \} )$-saturated in $P_{2r} \cup \left \{v\right \}$. 
			Since $\mathcal V_G(P_{2r} \cup \left \{v \right \})$ is $( P_{2r}^\prime \cup \left \{v, v^\prime \right \} )$-saturated in $P_{2r} \cup \left \{v\right \}$, the vertices of $G$ correspond to vertices of $P_{2r}^\prime \cup \left \{v, v^\prime \right \}$.
			Two uncharged vertices cannot be switch-equivalent to the same vertex since $YB_4$ has been excluded as a subgraph.
			Suppose two charged vertices $x$ and $y$ are switch-equivalent to the same hollow vertex.
			They must have the same charge.
			If $x$ and $y$ have charge $1$ then $\abs{\inprod{\mathbf{x},\mathbf{y}}} = 3$ which violates Lemma~\ref{lem:maxDeg4}.
			The vertices $v_1$ and $v_1^\prime$ have charge $-1$ and both are switch-equivalent to the same vertex.
			But since we have excluded $XC_6$, $XC_7$, and $XC_8$, no three vertices of charge $-1$ can be switch-equivalent to the same vertex.
			The graph $P_{2r}^\prime \cup \left \{v, v^\prime \right \}$ is $C_{2(r+1)}^{+-}$ and hence, $G$ is equivalent to a subgraph of $C_{2(r+1)}^{+-}$.
			
			\paragraph{Case 2}
			Suppose that $v$ has charge $-1$.
			Argument is similar to Case~1, but this time $\xi = 0$.
			We deduce that $G$ is contained in a $\Z[i]$-graph equivalent to $C_{2(r+1)}^{++}$.
	\end{proof}
	% subsection inductive_lemmata (end)

	\subsection{Charged $\mathcal L_3$-free $\Z[i]$-graphs on up to 5 vertices}

	We have exhaustively computed all charged $\mathcal L_3$-free $\Z[i]$-graphs on up to $5$ vertices.
	Out of these graphs, the ones on $5$ vertices contain a subgraph equivalent to $P_{6}$ ($3$ vertices).
	The growing process is similar to that described in Section~\ref{sec:unchargedS9}, but in this case we can start the process with a vertex having charge $-1$.
	As before, this computation can be carried out by hand.
	
	By the above computation and by iteratively applying Lemma~\ref{lem:unchargedinductcharge} and Lemma~\ref{lem:chargedinductcharge} we have the following lemma.
	
	\begin{lemma}\label{lem:L3free}
		Let $G$ be a charged $\mathcal L_3$-free $\Z[i]$-graph.
		Then $G$ is equivalent to a subgraph of one of the maximal cyclotomic $\Z[i]$-graphs $C_{2k}^{++}$, $C_{2k}^{+-}$, or $C_{2k-1}$ for some $k \geqslant 2$.
	\end{lemma}
	
	Together with the computation of the maximal connected cyclotomic $\Z[i]$-graphs containing the graphs from the list $\mathcal L_3$ (see Figure~\ref{fig:xgraphsCII}), we have proved Theorem~\ref{thm:classchzi}.
	
	\section{The Eisenstein integers} % (fold)
	\label{sec:the_eisenstein_integers}
	
	The classification of cyclotomic matrices over $\Z[\omega]$ is very similar to the classification over $\Z[i]$.
	In this section we outline the differences that need to be considered for this classification.
	
		\subsection{Uncharged case}

		% \begin{figure}[htbp]
		% 	\centering
		% 	\begin{tikzpicture}
		% 	\begin{scope}[yshift=-0.5cm]	
		% 		\foreach \pos/\name in {{(0,0)/a}, {(0,1)/b}, {(1,1)/c}, {(1,0)/d},{(2,0)/f}}
		% 			\node[vertex] (\name) at \pos {}; % {$\name$};
		% 		\node at (1,-0) {}; %{$XA_1$};
		% 		\foreach \edgetype/\source/ \dest /\weight in {pedge/a/b/{}, pedge/b/c/{}, pedge/c/d/{}, pedge/a/d/{}, pedge/d/f/{}}
		% 		\path[\edgetype] (\source) -- node[weight] {$\weight$} (\dest);
		% 	\end{scope}
		% 	\end{tikzpicture}
		% 	\caption{Excluded $\Z[\omega]$-subgraph of type I.}
		% 	\label{fig:xgraphsw1}
		% \end{figure}
		
		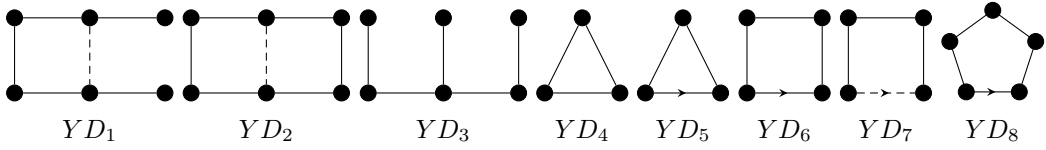
\begin{figure}[h!tbp]
			\centering
			\begin{tikzpicture}
			\begin{scope}[yshift=-0.5cm]	
				\foreach \pos/\name in {{(0,0)/a}, {(0,1)/b}, {(1,1)/c}, {(1,0)/d}, {(2,1)/e}, {(2,0)/f}}
					\node[vertex] (\name) at \pos {}; % {$\name$};
				\node at (1,-0.5) {$YD_1$};
				\foreach \edgetype/\source/ \dest /\weight in {pedge/a/b/{}, pedge/b/c/{}, nedge/c/d/{}, pedge/a/d/{}, pedge/d/f/{}, pedge/c/e/{}}
				\path[\edgetype] (\source) -- node[weight] {$\weight$} (\dest);
			\end{scope}
			\end{tikzpicture}
			\begin{tikzpicture}
			\begin{scope}[xshift=2.5cm, yshift=-0.5cm]	
				\foreach \pos/\name in {{(0,0)/a}, {(0,1)/b}, {(1,1)/c}, {(1,0)/d}, {(2,1)/e}, {(2,0)/f}}
					\node[vertex] (\name) at \pos {}; % {$\name$};
				\node at (1,-0.5) {$YD_2$};
				\foreach \edgetype/\source/ \dest /\weight in {pedge/a/b/{}, pedge/b/c/{}, nedge/c/d/{}, pedge/a/d/{}, pedge/d/f/{}, pedge/c/e/{}, pedge/e/f/{}}
				\path[\edgetype] (\source) -- node[weight] {$\weight$} (\dest);
			\end{scope}
			\end{tikzpicture}
			\begin{tikzpicture}
			\begin{scope}[xshift=2.75cm, yshift=-0.5cm]	
				\foreach \pos/\name in {{(0,0)/a}, {(0,1)/b}, {(1,1)/c}, {(1,0)/d}, {(2,1)/e}, {(2,0)/f}}
					\node[vertex] (\name) at \pos {}; % {$\name$};
				\node at (1,-0.5) {$YD_3$};
				\foreach \edgetype/\source/ \dest /\weight in {pedge/a/b/{}, pedge/c/d/{}, pedge/a/d/{}, pedge/d/f/{}, pedge/e/f/{}}
				\path[\edgetype] (\source) -- node[weight] {$\weight$} (\dest);
			\end{scope}
			\end{tikzpicture}
			\begin{tikzpicture}
			\begin{scope}[xshift=8cm, yshift=-0.5cm]	
				\foreach \pos/\name/\sign/\charge in {{(0.5,1)/a/zc/{}}, {(0,0)/b/zc/{}}, {(1,0)/c/zc/{}}}
					\node[\sign] (\name) at \pos {$\charge$}; % {$\name$};
				\node at (0.5,-0.5) {$YD_4$};
				\foreach \edgetype/\source/ \dest /\weight in {pedge/a/b/{}, pedge/c/b/{}, pedge/a/c/{}}
				\path[\edgetype] (\source) -- node[weight] {$\weight$} (\dest);
			\end{scope}
			\end{tikzpicture}
			\begin{tikzpicture}
			\begin{scope}[xshift=9.5cm, yshift=-0.5cm]	
				\foreach \pos/\name/\sign/\charge in {{(0.5,1)/a/zc/{}}, {(0,0)/b/zc/{}}, {(1,0)/c/zc/{}}}
					\node[\sign] (\name) at \pos {$\charge$}; % {$\name$};
				\node at (0.5,-0.5) {$YD_5$};
				\foreach \edgetype/\source/ \dest /\weight in {pedge/a/b/{}, wedge/b/c/{}, pedge/a/c/{}}
				\path[\edgetype] (\source) -- node[weight] {$\weight$} (\dest);
			\end{scope}
			\end{tikzpicture}
			\begin{tikzpicture}
			\begin{scope}[xshift=5cm, yshift=-0.5cm]		
				\foreach \pos/\name in {{(0,0)/a}, {(0,1)/b}, {(1,1)/c}, {(1,0)/d}}
					\node[vertex] (\name) at \pos {}; % {$\name$};
				\node at (0.5,-0.5) {$YD_6$};
				\foreach \edgetype/\source/ \dest /\weight in {pedge/a/b/{}, pedge/b/c/{}, pedge/c/d/{}, wedge/a/d/{}}
				\path[\edgetype] (\source) -- node[weight] {$\weight$} (\dest);
			\end{scope}
			\end{tikzpicture}
			\begin{tikzpicture}
			\begin{scope}[xshift=6.5cm, yshift=-0.5cm]		
				\foreach \pos/\name in {{(0,0)/a}, {(0,1)/b}, {(1,1)/c}, {(1,0)/d}}
					\node[vertex] (\name) at \pos {}; % {$\name$};
				\node at (0.5,-0.5) {$YD_7$};
				\foreach \edgetype/\source/ \dest /\weight in {pedge/a/b/{}, pedge/b/c/{}, pedge/c/d/{}, wnedge/a/d/{}}
				\path[\edgetype] (\source) -- node[weight] {$\weight$} (\dest);
			\end{scope}
			\end{tikzpicture}
			\begin{tikzpicture}
				\begin{scope}
					\newdimen\rad
					\rad=0.6cm

				    % Indicate the boundary of the regular polygons
					\foreach \x in {90,162,234,306,378}
					{
				    	\draw (\x:\rad) node[vertex] {};
				    }
					\foreach \x in {90,162,306,378}
					{
						\draw[pedge] (\x:\rad) -- (\x+72:\rad);
				    }
				\draw[wedge] (234:\rad) -- (306:\rad);
				\end{scope}
				\begin{scope}[xshift=-0.5cm, yshift=-0.5cm]
					\node at (0.5,-0.5) {$YD_8$};
				\end{scope}
			\end{tikzpicture}
			\caption{some cyclotomic $\Z[\omega]$-graphs that are contained as subgraphs of fixed maximal connected cyclotomic $\Z[\omega]$-graphs.}
			\label{fig:xgraphsw}
		\end{figure}
		\begin{table}[htbp]
			\begin{center}
			\begin{tabular}{c|c}
				Excluded subgraph & Maximal cyclotomics \\
				\hline
				$YD_1$ & $S_{12}$, $S_{14}$, and $S_{16}$ \\
				$YD_2$ & $S_{12}$, $S_{14}$, and $S_{16}$ \\
				$YD_3$ & $S_{12}$, $S_{14}$, and $S_{16}$ \\
				$YD_4$ & $S_5$, $T_{6}$, and $S_7$ \\
				$YD_5$ & $T^{(\omega)}_{6}$ \\
				$YD_6$ & $T^{(\omega)}_{8}$ \\
				$YD_7$ & $T^{(\omega)}_{8}$, $S_{10}$, and $S_{12}$ \\
				$YD_8$ & $T^{(\omega)}_{10}$
			\end{tabular}
			\end{center}
			\caption{Excluded subgraphs from Figure~\ref{fig:xgraphsw} and (up to equivalence) their containing maximal connected cyclotomic $\Z[\omega]$-graphs.}
			\label{tab:exgraphsw}
		\end{table}
		
		The uncharged case follows Section~\ref{sec:growZi}.
		In Table~\ref{tab:exgraphsw} we list each excluded subgraph of type II in Figure~\ref{fig:xgraphsw} along with every maximal connected cyclotomic $\Z[\omega]$-graph that contains it.
		Form a list of excluded subgraphs consisting of charged vertices and the excluded subgraphs from Figures~\ref{fig:xgraphsit1} and \ref{fig:xgraphsw}.
		Note that in Figure~\ref{fig:xgraphsw}, the $\Z[\omega]$-graph $YD_4$ is \emph{not} equivalent to $YD_5$ whereas over $\Z[i]$ these two graphs are equivalent.
		We effectively have the same list of excluded subgraphs as we had working over $\Z[i]$.
		The key requisites of the lemmata of Section~\ref{sec:growZi} are the set of excluded subgraphs and that the action of the group of units of $\Z[\omega]$ acts transitively on the set $S \backslash \left \{ 0 \right \}$.
		Let $S$ be the set containing $0$ and the units of $\Z[\omega]$, namely, $S = \left \{0,\pm 1, \pm \omega, \pm \bar \omega \right \}$.
		Then by following Section~\ref{sec:growZi} with this new set $S$ and our list of excluded subgraphs, we obtain a proof of the classification of cyclotomic $S$-graphs.
		
		The only elements of $\Z[\omega]$ of norm greater than $1$ and at most $4$ are the associates of $1+\omega$ or $2$.
		A simple computation confirms that any cyclotomic graph containing a subgraph equivalent to a weight-$(1+\omega)$ edge or a weight-$2$ edge must itself be equivalent to a subgraph of $S_4^\ddag$ or $S_2$ respectively.
		Lemma~\ref{lem:maxDeg4} takes care of the remainder of the elements of $\Z[\omega]$ and we have completed the proof of Theorem~\ref{thm:classunzw}.		
		
		\subsection{Charged case}
		
		In Table~\ref{tab:exgraphswcharge} we list each excluded subgraph of type II in Figure~\ref{fig:xgraphswcharge} along with every maximal connected cyclotomic $\Z[\omega]$-graph that contains it.
		Form a list of excluded subgraphs consisting of the excluded subgraphs from Figures~\ref{fig:xgraphsit1}, \ref{fig:xgraphsw}, \ref{fig:xgraphswCI}, and \ref{fig:xgraphswcharge}.
		Again, there exist charged excluded subgraphs that are not equivalent over $\Z[\omega]$ but are equivalent over $\Z[i]$.
		As in the uncharged case, we have the requisites for the lemmata of  Section~\ref{sec:growZicharge}; using excluded subgraphs and Lemma~\ref{lem:maxDeg4} we can rule out matrices that have an entry of norm greater than $1$.
		We have effectively the same list of excluded subgraphs and, in fact, the argument is simpler in this case, since there are no elements in $\Z[\omega]$ having norm $2$, whereas over $\Z[i]$ we had to consider edge-weights of norm $2$.

		\begin{figure}[htbp]
			\centering
			\begin{tikzpicture}
			\begin{scope}[scale=1, auto]
				\foreach \pos/\name/\sign/\charge in {{(0,1)/a/pc/+}, {(0,0)/b/zc/{}}}
					\node[\sign] (\name) at \pos {$\charge$}; % {$\name$};
				\node at (0,-0.5) {}; %{$XC_1$};
				\foreach \edgetype/\source/ \dest /\weight in {wwwedge2/a/b/{}}
				\path[\edgetype] (\source) -- node[weight] {$\weight$} (\dest);
			\end{scope}
			\end{tikzpicture}
			\begin{tikzpicture}
			\begin{scope}[scale=1, auto]
				\foreach \pos/\name/\sign/\charge in {{(0,1)/a/pc/+}, {(0,0)/b/pc/+}}
					\node[\sign] (\name) at \pos {$\charge$}; % {$\name$};
				\node at (0,-0.5) {}; %{$XC_1$};
				\foreach \edgetype/\source/ \dest /\weight in {wwwedge2/a/b/{}}
				\path[\edgetype] (\source) -- node[weight] {$\weight$} (\dest);
			\end{scope}
			\end{tikzpicture}
			\begin{tikzpicture}
			\begin{scope}[scale=1, auto]
				\foreach \pos/\name/\sign/\charge in {{(0.7,1)/a/pc/+}, {(0,0)/b/zc/{}}, {(1.4,0)/c/zc/{}}}
					\node[\sign] (\name) at \pos {$\charge$}; % {$\name$};
				\node at (0.7,-0.5) {}; %{$XC_1$};
				\foreach \edgetype/\source/ \dest /\weight in {pedge/a/b/{}, pedge/c/b/{}, pedge/a/c/{}}
				\path[\edgetype] (\source) -- node[weight] {$\weight$} (\dest);
			\end{scope}
			\end{tikzpicture}
			\begin{tikzpicture}[scale=1, auto]
			\begin{scope}[xshift=6cm, yshift=-0.5cm]	
				\foreach \pos/\name/\sign/\charge in {{(0.7,1)/a/pc/+}, {(0,0)/b/zc/{}}, {(1.4,0)/c/zc/{}}}
					\node[\sign] (\name) at \pos {$\charge$}; % {$\name$};
				\node at (0.7,-0.5) {}; %{$XC_2$};
				\foreach \edgetype/\source/ \dest /\weight in {pedge/a/b/{}, wedge/c/b/{}, pedge/a/c/{}}
				\path[\edgetype] (\source) -- node[weight] {$\weight$} (\dest);
			\end{scope}
			\end{tikzpicture}
			\begin{tikzpicture}
			\begin{scope}[scale=1, auto]	
				\foreach \pos/\name/\sign/\charge in {{(0.7,1)/a/zc/{}}, {(0,0)/b/pc/+}, {(1.4,0)/c/pc/+}}
					\node[\sign] (\name) at \pos {$\charge$}; % {$\name$};
				\node at (0.7,-0.5) {}; %{$XC_3$};
				\foreach \edgetype/\source/ \dest /\weight in {pedge/a/b/{}, pedge/c/b/{}, pedge/a/c/{}}
				\path[\edgetype] (\source) -- node[weight] {$\weight$} (\dest);
			\end{scope}
			\end{tikzpicture}
			\begin{tikzpicture}[scale=1, auto]
			\begin{scope}[xshift=9cm, yshift=-0.5cm]	
				\foreach \pos/\name/\sign/\charge in {{(0.7,1)/a/zc/{}}, {(0,0)/b/pc/+}, {(1.4,0)/c/pc/+}}
					\node[\sign] (\name) at \pos {$\charge$}; % {$\name$};
				\node at (0.7,-0.5) {}; %{$XC_4$};
				\foreach \edgetype/\source/ \dest /\weight in {pedge/a/b/{}, wedge/c/b/{}, pedge/a/c/{}}
				\path[\edgetype] (\source) -- node[weight] {$\weight$} (\dest);
			\end{scope}
			\end{tikzpicture}
			\begin{tikzpicture}[scale=1, auto]
			\begin{scope}[xshift=12cm, yshift=-0.5cm]	
				\foreach \pos/\name/\sign/\charge in {{(0.7,1)/a/zc/{}}, {(0,0)/b/pc/+}, {(1.4,0)/c/pc/+}}
					\node[\sign] (\name) at \pos {$\charge$}; % {$\name$};
				\node at (0.7,-0.5) {}; %{$XC_5$};
				\foreach \edgetype/\source/ \dest /\weight in {pedge/a/b/{}, wnedge/c/b/{}, pedge/a/c/{}}
				\path[\edgetype] (\source) -- node[weight] {$\weight$} (\dest);
			\end{scope}
			\end{tikzpicture}
			
			\begin{tikzpicture}
			\begin{scope}[scale=1, auto]	
				\foreach \pos/\name/\sign/\charge in {{(0.7,1)/a/pc/+}, {(0,0)/b/pc/+}, {(1.4,0)/c/pc/+}}
					\node[\sign] (\name) at \pos {$\charge$}; % {$\name$};
				\node at (0.7,-0.5) {}; %{$XC_6$};
				\foreach \edgetype/\source/ \dest /\weight in {pedge/a/b/{}, pedge/c/b/{}, pedge/a/c/{}}
				\path[\edgetype] (\source) -- node[weight] {$\weight$} (\dest);
			\end{scope}
			\end{tikzpicture}
			\begin{tikzpicture}
			\begin{scope}[scale=1, auto]	
				\foreach \pos/\name/\sign/\charge in {{(0.7,1)/a/pc/+}, {(0,0)/b/pc/+}, {(1.4,0)/c/pc/+}}
					\node[\sign] (\name) at \pos {$\charge$}; % {$\name$};
				\node at (0.7,-0.5) {}; %{$XC_7$};
				\foreach \edgetype/\source/ \dest /\weight in {pedge/a/b/{}, nedge/c/b/{}, pedge/a/c/{}}
				\path[\edgetype] (\source) -- node[weight] {$\weight$} (\dest);
			\end{scope}
			\end{tikzpicture}
			\begin{tikzpicture}[scale=1, auto]
			\begin{scope}[xshift=9cm, yshift=-0.5cm]	
				\foreach \pos/\name/\sign/\charge in {{(0.7,1)/a/pc/+}, {(0,0)/b/pc/+}, {(1.4,0)/c/pc/+}}
					\node[\sign] (\name) at \pos {$\charge$}; % {$\name$};
				\node at (0.7,-0.5) {}; %{$XC_8$};
				\foreach \edgetype/\source/ \dest /\weight in {pedge/a/b/{}, wedge/c/b/{}, pedge/a/c/{}}
				\path[\edgetype] (\source) -- node[weight] {$\weight$} (\dest);
			\end{scope}
			\end{tikzpicture}
			\begin{tikzpicture}[scale=1, auto]
			\begin{scope}[xshift=9cm, yshift=-0.5cm]	
				\foreach \pos/\name/\sign/\charge in {{(0.7,1)/a/pc/+}, {(0,0)/b/pc/+}, {(1.4,0)/c/pc/+}}
					\node[\sign] (\name) at \pos {$\charge$}; % {$\name$};
				\node at (0.7,-0.5) {}; %{$XC_82$};
				\foreach \edgetype/\source/ \dest /\weight in {pedge/a/b/{}, wnedge/c/b/{}, pedge/a/c/{}}
				\path[\edgetype] (\source) -- node[weight] {$\weight$} (\dest);
			\end{scope}
			\end{tikzpicture}
			
			\begin{tikzpicture}
			\begin{scope}[scale=1, auto]	
				\foreach \pos/\name/\sign/\charge in {{(0,1)/a/pc/+}, {(1,1)/b/pc/+}, {(1,0)/c/vertex/{}}}
					\node[\sign] (\name) at \pos {$\charge$}; % {$\name$};
				\foreach \edgetype/\source/ \dest /\weight in {pedge/a/b/{}, pedge/c/b/{}}
				\path[\edgetype] (\source) -- node[weight] {$\weight$} (\dest);
				\node at (1,-0) {}; %{$XC_{11}$};
			\end{scope}
			\end{tikzpicture}
			\begin{tikzpicture}
			\begin{scope}[scale=1, auto]	
				\foreach \pos/\name/\sign/\charge in {{(0,0)/a/vertex/{}}, {(0,1)/b/vertex/{}}, {(1,1)/c/pc/+}, {(1,0)/d/vertex/{}}}
					\node[\sign] (\name) at \pos {$\charge$}; % {$\name$};
				\foreach \edgetype/\source/ \dest /\weight in {pedge/a/b/{}, pedge/c/b/{}, pedge/c/d/{}}
				\path[\edgetype] (\source) -- node[weight] {$\weight$} (\dest);
				\node at (1,-0) {}; %{$XC_{11}$};
			\end{scope}
			\end{tikzpicture}
			\begin{tikzpicture}
				\begin{scope}[scale=1, auto]
					\foreach \pos/\name/\sign/\charge in {{(0,0)/a1/pc/+}, {(1,1)/b1/zc/{}}, {(2,1)/c1/zc/{}}, {(2,0)/d1/zc/{}}, {(0,1)/c2/zc/{}}}
						\node[\sign] (\name) at \pos {$\charge$};
					\foreach \edgetype/\source/ \dest in {pedge/a1/b1, pedge/b1/c1, pedge/d1/c1, pedge/b1/c2}
						\path[\edgetype] (\source) -- (\dest);
					\node at (1,-0) {}; %{$XC_{11}$};
				\end{scope}
			\end{tikzpicture}
			\begin{tikzpicture}
				\begin{scope}[scale=1, auto]
					\foreach \pos/\name/\sign/\charge in {{(0,0)/a1/pc/+}, {(1,1)/b1/zc/{}}, {(2,1)/c1/zc/{}}, {(2,0)/d1/pc/+}, {(0,1)/c2/zc/{}}}
						\node[\sign] (\name) at \pos {$\charge$};
					\foreach \edgetype/\source/ \dest in {pedge/a1/b1, pedge/b1/c1, pedge/d1/c1, pedge/b1/c2}
						\path[\edgetype] (\source) -- (\dest);
					\node at (1,-0) {}; %{$XC_{12}$};
				\end{scope}
			\end{tikzpicture}
			\begin{tikzpicture}
				\begin{scope}[scale=1, auto]
					\foreach \pos/\name/\sign/\charge in {{(0,0)/a1/pc/-}, {(1,1)/b1/zc/{}}, {(2,1)/c1/zc/{}}, {(2,0)/d1/pc/+}, {(0,1)/c2/zc/{}}}
						\node[\sign] (\name) at \pos {$\charge$};
					\foreach \edgetype/\source/ \dest in {pedge/a1/b1, pedge/b1/c1, pedge/d1/c1, pedge/b1/c2}
						\path[\edgetype] (\source) -- (\dest);
					\node at (1,-0) {}; %{$XC_{13}$};
				\end{scope}
			\end{tikzpicture}
			\caption{some non-cyclotomic charged $\Z[\omega]$-graphs.}
			\label{fig:xgraphswCI}
		\end{figure}
		
		\begin{figure}[h!tbp]
			\centering
			\begin{tikzpicture}
			\begin{scope}[scale=1]
				\foreach \pos/\name/\sign/\charge in {{(0,1)/a/zc/{}}, {(0,0)/b/zc/{}}}
					\node[\sign] (\name) at \pos {$\charge$}; % {$\name$};
				\node at (0,-0.5) {$YE_1$};
				\foreach \edgetype/\source/ \dest /\weight in {pedge/a/b/2}
				\path[\edgetype] (\source) -- node[weight2] {$\weight$} (\dest);
			\end{scope}
			\end{tikzpicture}
			\begin{tikzpicture}
			\begin{scope}[scale=1, auto]
				\foreach \pos/\name/\sign/\charge in {{(0,1)/a/zc/{}}, {(0,0)/b/zc/{}}}
					\node[\sign] (\name) at \pos {$\charge$}; % {$\name$};
				\node at (0,-0.5) {$YE_2$};
				\foreach \edgetype/\source/ \dest /\weight in {wwwedge/a/b/{}}
				\path[\edgetype] (\source) -- node[weight] {$\weight$} (\dest);
			\end{scope}
			\end{tikzpicture}
			\begin{tikzpicture}
			\begin{scope}[scale=1, auto]
				\foreach \pos/\name/\sign/\charge in {{(0,1)/a/pc/+}, {(0,0)/b/pc/-}}
					\node[\sign] (\name) at \pos {$\charge$}; % {$\name$};
				\node at (0,-0.5) {$YE_3$};
				\foreach \edgetype/\source/ \dest /\weight in {pedge/a/b/{}}
				\path[\edgetype] (\source) -- node[weight] {$\weight$} (\dest);
			\end{scope}
			\end{tikzpicture}
			\begin{tikzpicture}
			\begin{scope}[scale=1, auto]
				\foreach \pos/\name/\sign/\charge in {{(0,1)/a/pc/+}, {(0,0)/b/pc/-}}
					\node[\sign] (\name) at \pos {$\charge$}; % {$\name$};
				\node at (0,-0.5) {$YE_4$};
				\foreach \edgetype/\source/ \dest /\weight in {wwwedge2/a/b/{}}
				\path[\edgetype] (\source) -- node[weight] {$\weight$} (\dest);
			\end{scope}
			\end{tikzpicture}
			\begin{tikzpicture}
			\begin{scope}[scale=1, auto]	
				\foreach \pos/\name/\sign/\charge in {{(0,1)/a/pc/+}, {(1,1)/b/zc/{}}, {(0,0)/c/pc/+}}
					\node[\sign] (\name) at \pos {$\charge$}; % {$\name$};
				\node at (0.5,-0.5) {$YE_5$};
				\foreach \edgetype/\source/ \dest /\weight in {pedge/a/b/{}, pedge/c/b/{}}
				\path[\edgetype] (\source) -- node[weight] {$\weight$} (\dest);
			\end{scope}
			\end{tikzpicture}
			\begin{tikzpicture}
			\begin{scope}[scale=1, auto]	
				\foreach \pos/\name/\sign/\charge in {{(0,1)/a/pc/+}, {(1,1)/b/zc/{}}, {(0,0)/c/pc/-}}
					\node[\sign] (\name) at \pos {$\charge$}; % {$\name$};
				\node at (0.5,-0.5) {$YE_6$};
				\foreach \edgetype/\source/ \dest /\weight in {pedge/a/b/{}, pedge/c/b/{}}
				\path[\edgetype] (\source) -- node[weight] {$\weight$} (\dest);
			\end{scope}
			\end{tikzpicture}
			\begin{tikzpicture}[scale=1, auto]
			\begin{scope}[xshift=6cm, yshift=-0.5cm]	
				\foreach \pos/\name/\sign/\charge in {{(0.7,1)/a/pc/+}, {(0,0)/b/zc/{}}, {(1.4,0)/c/zc/{}}}
					\node[\sign] (\name) at \pos {$\charge$}; % {$\name$};
				\node at (0.7,-0.5) {$YE_7$};
				\foreach \edgetype/\source/ \dest /\weight in {pedge/a/b/{}, nedge/c/b/{}, pedge/a/c/{}}
				\path[\edgetype] (\source) -- node[weight] {$\weight$} (\dest);
			\end{scope}
			\end{tikzpicture}
			\begin{tikzpicture}[scale=1, auto]
			\begin{scope}[xshift=6cm, yshift=-0.5cm]	
				\foreach \pos/\name/\sign/\charge in {{(0.7,1)/a/pc/+}, {(0,0)/b/zc/{}}, {(1.4,0)/c/zc/{}}}
					\node[\sign] (\name) at \pos {$\charge$}; % {$\name$};
				\node at (0.7,-0.5) {$YE_8$};
				\foreach \edgetype/\source/ \dest /\weight in {pedge/a/b/{}, wnedge/c/b/{}, pedge/a/c/{}}
				\path[\edgetype] (\source) -- node[weight] {$\weight$} (\dest);
			\end{scope}
			\end{tikzpicture}
			\begin{tikzpicture}[scale=1, auto]
			\begin{scope}	
				\foreach \pos/\name/\type/\charge in {{(0,0)/a/pc/{2}}}
					\node[\type] (\name) at \pos {$\charge$}; % {$\name$};
				\node at (0,-0.8) {$YE_9$};
			\end{scope}
			\end{tikzpicture}
			\caption{some charged cyclotomic $\Z[\omega]$-graphs that are contained as subgraphs of fixed maximal connected cyclotomic $\Z[\omega]$-graphs.}
			\label{fig:xgraphswcharge}
		\end{figure}
		\begin{table}[h!t]
			\begin{center}
			\begin{tabular}{c|c}
				Excluded subgraph & Maximal cyclotomics \\
				\hline
				$YE_1$ & $S_2$ \\
				$YE_2$ & $S_4^\ddag$ \\
				$YE_3$ & $C^{+-}_4$, $S_6$, $S_6^\dag$, $S_7$, $S_8$, and $S_8^\prime$ \\
				$YE_4$ & $S_2^\dag$ \\
				$YE_5$ & $S_5$, $C_6^{++}$, and $S_7$\\
				$YE_6$ & $C_6^{+-}$ and $S_8^\prime$ \\
				$YE_7$ & $S_6^\dag$, $S_7$, and $S_8^\prime$ \\
				$YE_8$ & $S_5$ \\
				$YE_9$ & $S_1$
			\end{tabular}
			\end{center}
			\caption{Excluded subgraphs from Figure~\ref{fig:xgraphswcharge} and (up to equivalence) their containing maximal connected cyclotomic $\Z[\omega]$-graphs.}
			\label{tab:exgraphswcharge}
		\end{table}
	
	% section the_eisenstein_integers (end)
	
	\section{Other quadratic integer rings}
	\label{sec:extensions}
	
	Finally, we outline how the method used in this paper can be used to classify cyclotomic matrices over other quadratic integer rings; the method of using Gram vectors and excluded subgraphs also goes through.
	In particular, we can offer a simpler proof of Taylor's classification \cite{GTay:cyclos10} of cyclotomic matrices over the imaginary quadratic integer rings $R = \mathcal O_{\Q(\sqrt{d})}$ where $d \ne -1$ and $d \ne -3$.
	We outline the idea of this proof.
	Since cyclotomic $\Z$-graphs have been classified, we can restrict to considering cyclotomic $R$-graphs that contain at least one weight-$\alpha$ edge, where $\alpha \not \in \Z$.
	By Lemma~\ref{lem:maxDeg4}, we need only consider the rings $R$ that contain at least one element $\alpha \not \in \Z$ with norm at most $4$.
	Namely, these are the rings $\mathcal O_{\Q(\sqrt{d})}$ where $d \in \{ -15, -11, -7, -3, -2, -1, 2, 3, 5 \}$.
	In each of the quadratic integer rings where $d \ne -1$ and $d \ne -3$ the only units are $\pm 1$.
	Therefore we can restrict attention to $R$-graphs that contain at least one weight-$\alpha$ edge where $\alpha \in R \backslash \Z$ has norm greater than $1$.
	In fact, for the rings $\mathcal O_{\Q(\sqrt{d})}$ where $d \in \{ -15, -11 \}$, the cyclotomic matrices can be classified by exhaustive computation.
	For the imaginary quadratic integer rings $\mathcal O_{\Q(\sqrt{d})}$ for $d \in \{ -7, -2 \}$, the proof then resembles Sections~\ref{sec:proof2} and \ref{sec:growZicharge}.
	
	We remark that in the real quadratic case, a bit more work needs to be done.
	When $d>1$, it is not necessary that a symmetric $\mathcal O_{\Q(\sqrt{d})}$-matrix will have an integral characteristic polynomial.
	For details of the real quadratic case see the author's later paper \cite{Greaves:CycloRQ11} or his thesis \cite{Greaves:Thesis12}.
	
		\section{Acknowledgement} % (fold)
	\label{sec:acknowledgement}
		The author is grateful for the advice and support of James McKee and for the comments of the referee.
	% section acknowledgement (end)
	
\bibliographystyle{plain}
\bibliography{../bib}
\end{document}